\documentclass{article}
\usepackage{authblk}
\usepackage[utf8]{inputenc}
\usepackage{amsmath,amsthm, amssymb, amsfonts}
\usepackage{fullpage}
\usepackage[pdfencoding=auto]{hyperref}
\usepackage{graphicx}
\usepackage{todonotes}
\usepackage{comment}
\usepackage{mathtools}
\usepackage{wrapfig}
\usepackage{microtype}
\usepackage[]{algorithm2e}

\usepackage{xcolor}

\allowdisplaybreaks[1]
\usepackage{sistyle} 

\newcommand{\aninorm}{\mathcal{N}}

\newcommand{\Z}{\mathbb{Z}}
\newcommand{\N}{\mathbb{N}}

\newcommand{\R}{\mathbb{R}}
\newcommand{\vn}{{\vec{n}}}

\newcommand{\step}[2]{\smallskip \textbf{Step #1: (#2)} \smallskip \\}

\newcommand{\Bp}{B \! ,p}
\newcommand{\DBp}{D\! B \! }

\let\emptyset\varnothing

\newtheorem{theorem}{Theorem}[section]
\newtheorem*{theorem*}{Theorem}
\newtheorem{remark}[theorem]{Remark}
\newtheorem{definition}[theorem]{Definition}
\newtheorem{proposition}[theorem]{Proposition}

\newtheorem{assumption}[theorem]{Assumption}
\newtheorem*{remark*}{Remark}
\newtheorem*{proposition*}{Proposition}

\numberwithin{equation}{section}

\DeclareMathOperator{\suppp}{supp \,}
\DeclareMathOperator*{\argmin}{arg\,min}

\definecolor{darkcandyapplered}{rgb}{0.64, 0.0, 0.0}

\title{\textbf{$\Gamma$-convergence of a shearlet-based Ginzburg--Landau energy}}
\author{\sc{Philipp Christian Petersen}}
\author{\sc{Endre S\"uli}}
\affil{\small{Mathematical Institute, University of Oxford, Woodstock Road, Oxford OX2 6GG, UK}}
\date{}

\begin{document}
\maketitle
\begin{abstract}
We introduce two shearlet-based Ginzburg--Landau energies, based on the continuous and the discrete shearlet transform. The energies result from replacing the elastic energy term of a classical Ginzburg--Landau energy by the weighted $L^2$-norm of a shearlet transform. The asymptotic behaviour of sequences of these energies is analysed within the framework of $\Gamma$-convergence and the limit energy is identified. We show that the limit energy of a characteristic function is an anisotropic surface integral over the interfaces of that function. We demonstrate that the anisotropy of the limit energy can be controlled by weighting the underlying shearlet transforms according to their directional parameter.
\end{abstract}


\smallskip


\section{Introduction}

The Ginzburg--Landau energy \cite{ginzburg1955theory}, introduced initially to model phase-transitions in superconductors, has recently found many applications in image processing \cite{bertozzi2007inpainting, burger2009cahn, dobrosotskaya2013analysis, grossauer2003using}.
In \cite{dobrosotskaya2011wavelet}, a wavelet-based Ginzburg--Landau energy was introduced, which,
in contrast to the original Ginzburg--Landau energy, is derivative-free and the solution of the associated minimisation problem does not require the solution of a partial differential equation.
The wavelet-based Ginzburg--Landau energy is constructed by substituting the elastic energy term of a standard Ginzburg--Landau energy by a wavelet-based Besov seminorm. This seminorm is formed by considering a weighted $L^2$-norm of the semi-continuous wavelet transform.
It was demonstrated in \cite{dobrosotskaya2011wavelet} that this generalised energy admits the following attractive theoretical property:
a sequence of wavelet-based Ginzburg--Landau energies $\Gamma$-converges to a limit that, for piecewise constant functions, can be described by an anisotropic surface integral over the interfaces.
Thereby, the functional introduces a directional anisotropy.
This is particularly convenient since modelling such anisotropies in the classical Ginzburg--Landau energy necessitates the solution of quasilinear PDEs for the numerical minimisation.
The wavelet-based formulation, on the other hand, can be minimised by solving a linear system.

Unfortunately, the wavelet-based functional can only produce very simple anisotropies.
It was conjectured in \cite{dobrosotskaya2013analysis} that a similar approach using directional systems, such as curvelets \cite{candes2000curvelets} or shearlets \cite{labate2005sparse, KGLConstrCmptShear2012},
would produce energies such that the anisotropy of the limiting energy could be more precisely controlled.
Indeed, this approach was numerically analysed in \cite{czaja2013composite}, but no theoretical analysis of the $\Gamma$-convergence was included.
One issue preventing direct translation of the theoretical arguments used for the convergence results of wavelet-based Ginzburg--Landau energies is that neither compactly supported shearlets,
nor any known shearlet system on a bounded domain, not resulting from truncation, forms a tight frame.
The proof of the convergence results for wavelet-based energies in \cite{dobrosotskaya2013analysis} is mostly based on a projection procedure of an underlying function to a space spanned by wavelets up to a maximal scale.
Such a projection operator must necessarily involve the dual frame. If the frame is not tight, then the dual frame is usually not known analytically. This severely complicates the analysis.

In this work, we present the first proof of the $\Gamma$-convergence of a sequence of shearlet-based Ginzburg--Landau energies in Theorem \ref{thm:main} and Theorem \ref{thm:discrete}.
The analysed energies are constructed similarly to the wavelet case, by replacing the
elasticity term of a Ginzburg--Landau energy by a weighted $L^2$-norm of a shearlet transform. We present results on two types of shearlet transforms, the continuous shearlet transform and the discrete shearlet transform.
We precisely identify the $\Gamma$-limit of the sequences of energies and demonstrate that the limit energy applied to characteristic functions of sets with finite perimeter
is a surface integral over the boundaries which can be controlled by weighting the individual shearlet elements according to their orientation parameter.
Our arguments are based on spatial and directional localisation, which demonstrates that for certain functions a reweighted cone-adapted shearlet transform is asymptotically equivalent
to a so-called homogeneous shearlet transform. Contrary to the cone-adapted shearlet transform, the homogeneous shearlet transform is an isometry.
This shows that after spatial and directional localisation the shearlet transform is locally and asymptotically the analysis operator of a tight frame.


The arguments in \cite{dobrosotskaya2011wavelet} describe the $\Gamma$-convergence of a sequence of wavelet-based functionals to an anisotropic perimeter functional by studying the quotient of a wavelet-based Besov seminorm and the $H^1$ seminorm. It is shown that for a particular choice of sequences the quotient converges and the limit can be identified.
To obtain convergence of the quotient for all sequences, a lifting argument is used. This then establishes the asymptotic equivalence of the wavelet-based Ginzburg--Landau energy and the standard Ginzburg--Landau energy up to a multiplicative constant which is given by the limit of the quotient described above. Finally, the fact that the standard Ginzburg--Landau energy converges to the perimeter functional is used to establish the $\Gamma$-limit of the wavelet-based functional.

Our approach is considerably different. We employ a different technique to construct a recovery sequence in Subsection \ref{sec:recSequenceForFinitePerimeter}, which is not based on projections, but on spatial and directional filtering.
The liminf condition in the definition of $\Gamma$-convergence then results from our construction of the shearlet-based energy, without any lifting step.

In fact, we cannot apply the lifting technique of \cite{dobrosotskaya2011wavelet}, since the argument there is incorrect. To elaborate on this, we provide a counterexample to the main theorem of
\cite{dobrosotskaya2011wavelet} and identify a miscalculation in the lifting argument in Appendix \ref{app:Counterex}.

\section{The continuous shearlet-based Ginzburg--Landau energy}

Before we start introducing shearlet systems and presenting the new construction of a shearlet-based Ginzburg--Landau energy, we first recall some, mostly standard notation in Subsection \ref{sec:notation} below. This section may be skipped and only referred to if a symbol or notion is not clear. In Subsection \ref{sec:ContShearSys}, we then recall the definition and classical results on shearlet systems. In Subsection \ref{sec:NormEquivalence}, we introduce a continuous shearlet system on bounded domains. Finally, in Subsection \ref{sec:SGLE} we introduce the shearlet-based Ginzburg--Landau energy and establish its $\Gamma$-convergence to an anisotropic perimeter functional in Theorem \ref{thm:main}.

\subsection{Notation}\label{sec:notation}
Let $d \in \N$. For $D \subset \R^d$ and $1\leq p\leq \infty$, we denote by $L^p(D)$ the standard Lebesgue spaces with parameter $p$.
We denote by $H^1(D)$ the Sobolev space of once weakly differentiable functions with $\|f\|_{H^1(D)}^2 \coloneqq \|f\|_{L^2(D)}^2 + |f|_{H^1(D)}^2\coloneqq \|f\|_{L^2(D)}^2 + \|\nabla f\|_{L^2(D)}^2$. 
Moreover, we denote by $C^\infty_c((0,1)^2)$ the space of infinitely many times differentiable functions on $(0,1)^2$ the support of which is a closed subset of $(0,1)^2$ (with respect to the topology on $\R^2$). Finally, we denote by $H^1_0((0,1)^2)$ the closure of $C^\infty_c((0,1)^2)$ in the $H^1((0,1)^2)$ norm. 

We denote by $BV$ the space of all functions of bounded variation on $(0,1)^2$ and by SBV the space of all special functions of bounded variation on $(0,1)^2$, see e.g. \cite[Definition 3.1 and Section 4.1]{ambrosio2000functions} for the definitions of these spaces.
The Fourier transform of a function $f\in L^1(\R^d)$ is defined by
$$
\mathcal{F}(f)(\xi) \coloneqq \hat{f}(\xi) \coloneqq \int_{\R^2} f(x)\, \mathrm{e}^{-2\pi i \langle x, \xi\rangle}, \, \,\mathrm{d}x\ \text{ for all } \xi \in \R^d.
$$
It is well known that $\mathcal{F}: L^1(\R^d) \cap L^2(\R^d) \to L^\infty(\R^d)$ can be extended to an isomorphism on $L^2(\R^d)$, which yields a Fourier transform on $L^2(\R^d)$. We denote the Fourier transform on $L^2(\R^d)$ by the same symbol as that on $L^1(\R^d)$. For a set $K \subset \R^d$, we denote by $\chi_K$ the characteristic function of $K$.
For $f \in L^p(\R^d)$, $g \in L^1(\R^d)$, we write $f*g$ for the convolution of $f$ and $g$. For $\phi_1, \phi_2\in L^2(\R)$ we define their tensor product by $(\phi_1 \otimes \phi_2)(x)\coloneqq \phi_1(x_1)\,\phi_2(x_2)$, where $x = (x_1,x_2) \in \R^2$.

We use the symbol $\mathcal{H}_1$ for the one-dimensional Hausdorff measure. For a vector $z \in \R^d$, we denote by $z_i$ the entry at its $i$-th coordinate.
We denote by $|\cdot|$ the Euclidean norm on $\R^d$ and $\|x\|_\infty \coloneqq \sup_{i = 1, \dots, d} |x_i|$ for $x \in \R^d$. $\mathbb{S}^1$ denotes the unit sphere in $\R^2$. For $\varepsilon >0$ and $x \in \R^d$, we denote by $B_\varepsilon(x)$ the open ball of radius $\varepsilon$ centered at $x$.
If $f,g \colon X \to \R^+$ and there is a $c>0$ such that $f(x) \leq c g(x)$ for all $x \in X$, then we write $f \lesssim g$. If $f \lesssim g$ and $g \lesssim f$, then we write $f \sim g$.

\subsection{Continuous shearlet systems}\label{sec:ContShearSys}
A shearlet system over $\R^2$ is a set of functions, generated from shifting, scaling, and shearing of one or more generator functions.
For the scaling and shearing operations, we require the following definitions: for $a \in \R^+$, $s \in \R$, we define the \emph{anisotropic scaling matrix} $A_a$ and the \emph{shear matrix} $S_s$ by
\begin{align*}
    A_a \coloneqq \begin{pmatrix} a & 0 \\ 0 & a^{\frac12} \end{pmatrix}\quad  \text{ and } \quad S_s \coloneqq \begin{pmatrix} 1 & s \\ 0 & 1 \end{pmatrix}.
\end{align*}
We define two shearlet transforms, beginning with the homogeneous shearlet transform, \cite{kutyniok2009resolution, dahlke2008uncertainty}.
\begin{definition}
Let $\psi \in L^2(\R^2)$; then we define
\begin{align*}
\mathcal{SH}: L^2\left(\R^2\right) &\to L^\infty\left(\R^+\times \R \times \R^2\right),\\ \quad
f &\mapsto \big((a,s,t) \mapsto \mathcal{SH}(f)(a,s,t)\coloneqq \langle f, \psi_{a,s,t} \rangle\big),
\end{align*}
where $\langle \cdot, \cdot \rangle$ denotes the inner product in $L^2(\R^2)$ and
$$
	\psi_{a,s,t}(x) \coloneqq a^{-\frac{3}{4}}\psi	\left(A_a^{-1} S_s( x -t)\right), \quad \text{  for } x \in \R^2.
$$
We call $\mathcal{SH}$ the \emph{homogeneous shearlet transform}. Moreover, we call \emph{$\mathcal{SH}(f)$ the homogeneous shearlet transform of $f$}.
\end{definition}
It is clear from the Cauchy--Schwarz inequality that $\mathcal{SH}$ is well-defined. More importantly, under suitable assumptions on $\psi$, $\mathcal{SH}$ is also well-defined as a map to $L^2(\R^+\times \R \times \R^2, a^{-3}\, \mathrm{d}a \,  \mathrm{d}s \, \mathrm{d}t)$
and is, in fact, an isometry; we recall the relevant results below.
\begin{definition}[\cite{dahlke2008uncertainty}]
A function $\psi \in L^2(\R^2)$ is called an \emph{admissible shearlet} if
$$
\int_{\R^2} \frac{|\widehat{\psi}(\xi)|^2}{|\xi_1|^2} \,\mathrm{d}\xi \eqqcolon C_\psi < \infty.
$$
\end{definition}
The following result shows the previously announced norm-equivalence and, for $C_\psi = 1$, the isometry property of the shearlet transform.
\begin{theorem}[\cite{dahlke2008uncertainty, grohs2011continuous}]\label{thm:HomFrameL2}
Let $\psi\in L^2(\R^2)$ be an {admissible shearlet}; then, for all $f \in L^2(\R^2)$,
$$
\|f\|_{L^2(\R^2)}^2 = C_\psi^{-1} \int_{\R^2}\int_{\R} \int_{\R^+} |\mathcal{SH}(f)(a,s,t)|^2 a^{-3} \, \mathrm{d}a \,  \mathrm{d}s \, \mathrm{d}t.
$$
\end{theorem}
To prevent unequal treatment of directions or shearlet elements $\psi_{a,s,t}$ with very long supports if $s$ becomes large, a cone-adapted shearlet system was introduced in \cite{labate2005sparse}. We define
$$
Q \coloneqq \left(\begin{array}{cc}
    0 & 1 \\
    1 & 0
\end{array} \right),
$$
$\widetilde{A}_a\coloneqq Q A_a$, and
$
\widetilde{\psi}_{a,s,t}\coloneqq a^{-\frac{3}{4}} \psi(Q \widetilde{A}_a^{-1} S_{s}^T (\cdot - t) )$ for $(a,s,t) \in \R^+\times \R \times \R^2$. 
In the sequel, we denote the Fourier transform of $\psi_{a,s,t}$ and $\widetilde{\psi}_{a,s,t}$ by $\widehat{\psi}_{a,s,t}$ and $\widehat{\widetilde{\psi}}_{a,s,t}$, respectively, instead of the more accurate, but aesthetically suboptimal symbols  
$\widehat{\psi_{a,s,t}}$, $\widehat{\widetilde{\psi}_{a,s,t}}$.

Cone-adapted systems are associated with a transform which, similarly to the homogeneous shearlet transform, admits a certain norm-equivalence with the $L^2(\R^2)$ norm.
\begin{theorem}[\cite{grohs2011continuous}]\label{eq:defOfContinuousFrame}
Let $\psi\in L^2(\R^2)$ be an admissible shearlet such that
$$
\left|\widehat{\psi}(\xi)\right| \lesssim \frac{\min\{|\xi_1|^2,1\}}{(1+|\xi_1|^2)(1+|\xi_2|^2)},\quad \text{ for all } \xi = (\xi_1,\xi_2) \in \R^2.
$$
Then, there exist $\mathfrak{a}^*>0$ and $\mathfrak{s}^*>0$ such that, for all $\mathfrak{a}\geq \mathfrak{a}^*$ and $\mathfrak{s} \geq \mathfrak{s}^*$ and every function $K \in L^2(\R^2)$ with $\widehat{K}(\xi) \in [A,B]$ for $0<A\leq B$, for all $\xi \in [-1,1]^2$, we have that
\begin{align}
\|f\|_{L^2(\R^2)}^2 \sim \int_{\R^2} |\langle f, K(\cdot -t) \rangle|^2 \, \,\mathrm{d}t \ +& \int_{\R^2}\int_{-\mathfrak{s}}^{\mathfrak{s}} \int_{0}^{\mathfrak{a}} |\langle f, \psi_{a,s,t} \rangle |^2 a^{-3} \, \mathrm{d}a \,  \mathrm{d}s \, \mathrm{d}t \label{eq:uniformConstants}\\
&\qquad  + \int_{\R^2}\int_{-\mathfrak{s}}^{\mathfrak{s}} \int_{0}^{\mathfrak{a}} |\langle f, \widetilde{\psi}_{a,s,t} \rangle |^2 a^{-3} \, \mathrm{d}a \,  \mathrm{d}s \, \mathrm{d}t, \nonumber
\end{align}
holds for all $f\in L^2(\R^2)$.
\end{theorem}
\begin{remark}
The implied constants in \eqref{eq:uniformConstants} can be chosen to depend on $K, \psi, \mathfrak{a}^*$, and $\mathfrak{s}^*$ only. Indeed, it is shown in Appendix \ref{app:Bessel} that the upper bound 
\begin{align*}
\int_{\R^2} |\langle f, K(\cdot -t) \rangle|^2 \, \,\mathrm{d}t \ +& \int_{\R^2}\int_{-\mathfrak{s}}^{\mathfrak{s}} \int_{0}^{\mathfrak{a}} |\langle f, \psi_{a,s,t} \rangle |^2 a^{-3} \, \mathrm{d}a \,  \mathrm{d}s \, \mathrm{d}t \\
&\qquad  + \int_{\R^2}\int_{-\mathfrak{s}}^{\mathfrak{s}} \int_{0}^{\mathfrak{a}} |\langle f, \widetilde{\psi}_{a,s,t} \rangle |^2 a^{-3} \, \mathrm{d}a \,  \mathrm{d}s \, \mathrm{d}t \leq C \|f\|_{L^2(\R^2)}^2
\end{align*}
holds for a constant $C>0$ independent of $\mathfrak{a}$ and $\mathfrak{s}$. On the other hand, the lower bound 
\begin{align*}
C' \|f\|_{L^2(\R^2)}^2 \leq \int_{\R^2} |\langle f, K(\cdot -t) \rangle|^2 \, \,\mathrm{d}t \ +& \int_{\R^2}\int_{-\mathfrak{s}}^{\mathfrak{s}} \int_{0}^{\mathfrak{a}} |\langle f, \psi_{a,s,t} \rangle |^2 a^{-3} \, \mathrm{d}a \,  \mathrm{d}s \, \mathrm{d}t \\
&\qquad  + \int_{\R^2}\int_{-\mathfrak{s}}^{\mathfrak{s}} \int_{0}^{\mathfrak{a}} |\langle f, \widetilde{\psi}_{a,s,t} \rangle |^2 a^{-3} \, \mathrm{d}a \,  \mathrm{d}s \, \mathrm{d}t. \nonumber
\end{align*}
clearly holds for all $\mathfrak{a}\geq \mathfrak{a}^*$ and $\mathfrak{s}\geq \mathfrak{s}^*$ as soon as it holds for $\mathfrak{s}= \mathfrak{s}^*$ and $\mathfrak{a} = \mathfrak{a}^*$.
\end{remark}

We have seen that the homogeneous and the cone-adapted shearlet transforms satisfy certain norm-equivalences to the $L^2$-norm.
Additionally, it is also possible to characterise Sobolev norms by these transforms in the following way.
Note that we have, for $f \in H^1(\R^2)$, by Parseval's identity, that
$$
|f|_{H^1(\R^2)}^2 = \left\|\frac{\partial f}{\partial x_1} \right\|_{L^2(\R^2)}^2 + \left\|\frac{\partial f}{\partial x_2} \right\|_{L^2(\R^2)}^2 = (2\pi)^2\left(\left\|\xi_1 \hat{f}\right\|_{L^2(\R^2)}^2 + \left\|\xi_2 \hat{f}\right\|_{L^2(\R^2)}^2\right).
$$
Let $\psi \in L^2(\R^2)$ be such that
\begin{align}\label{eq:TheAssumption}
\int_{\R^2} \frac{|\widehat{\psi}(\xi)|^2}{|\xi_1|^4} \,\mathrm{d}\xi = (2\pi)^2,
\end{align}
and set $\widehat{\mu}(\xi) \coloneqq \widehat{\psi}(\xi)/(2\pi\xi_1)$, then $C_\mu = 1$.
By Parseval's identity, we observe that
\begin{align*}
\int_{\R^2}\int_{\R} \int_{\R^+} a^{-2}|\langle f, \psi_{a,s,t} \rangle|^2 a^{-3} \, \mathrm{d}a \, \mathrm{d}s \, \,\mathrm{d}t = & \ \int_{\R^2}\int_{\R} \int_{\R^+} \left|\left\langle \frac{\widehat{\psi}_{a,s,t}}{2\pi a \xi_1} , 2\pi\xi_1 \hat{f} \right\rangle \right|^2 a^{-3} \, \mathrm{d}a \, \mathrm{d}s \, \mathrm{d}t\\
= & \ \int_{\R^2}\int_{\R} \int_{\R^+} \left|\left\langle \widehat{\mu}_{a,s,t} , 2\pi\xi_1 \hat{f} \right\rangle \right|^2 a^{-3} \, \mathrm{d}a \, \mathrm{d}s \, \mathrm{d}t = (2\pi)^2\left\|\xi_1 \hat{f}\right\|_{L^2(\R^2)}^2.
\end{align*}
Additionally, we will show the following result in Appendix \ref{app:H1Frame}. 
\begin{theorem}\label{thm:EmbeddingH1R2}
Let $\psi\in L^2(\R^2)$ be an admissible shearlet; then, there exist $\mathfrak{a}^*>0$, $\mathfrak{s}^*>0$ such that for all $\mathfrak{a}\geq \mathfrak{a}^*$, $\mathfrak{s} \geq \mathfrak{s}^*$ and every function $K \in L^2(\R^2)$ with $|\widehat{K}(\xi)|/|\xi| \in [A,B]$ for all $\xi \in [-1,1]^2$, and $|\widehat{K}(\xi)| \leq B$, for all $\xi \in \R^2$, where $0<A\leq B$,
there are $0<c_1\leq c_2$ such that we have, for all $f\in H^1(\R^2)$,
\begin{align}\label{eq:embeddingH1R2}
c_1 |f|_{H^1(\R^2)}^2 \leq \int_{\R^2} |\langle f, K(\cdot -t) \rangle|^2 \, \,\mathrm{d}t \ +  & \int_{\R^2}\int_{-\mathfrak{s}}^{\mathfrak{s}} \int_{0}^{\mathfrak{a}} a^{-2}|\langle f, \psi_{a,s,t} \rangle |^2 a^{-3} \, \mathrm{d}a \, \mathrm{d}s \, \mathrm{d}t\\
&\qquad  + \int_{\R^2}\int_{-\mathfrak{s}}^{\mathfrak{s}} \int_{0}^{\mathfrak{a}} a^{-2} \left|\left\langle f, \widetilde{\psi}_{a,s,t} \right\rangle \right|^2 a^{-3} \, \mathrm{d}a \, \mathrm{d}s \, \mathrm{d}t \leq c_2 |f|_{H^1(\R^2)}^2. \nonumber
\end{align}
The constants $c_1,c_2$ may depend on $\mathfrak{a}^*$ and $\mathfrak{s}^*$, but are independent of $\mathfrak{a}$ and $\mathfrak{s}$.
\end{theorem}
Note that, by Parseval's identity, if $|\widehat{K}(\xi)| \leq B$ for all $\xi \in \R^2$, then 
\begin{align} \label{eq:estimateOfKTerm}
\int_{\R^2} |\langle f, K(\cdot -t) \rangle|^2 \, \,\mathrm{d}t = \int_{\R^2} | f * K(t) |^2 \, \,\mathrm{d}t = \int_{\R^2} |\hat{f}(\xi)|^2 |\widehat{K}(\xi)|^2 \, \,\mathrm{d}\xi \leq B^2 \|f\|_{L^2(\R^2)}^2.
\end{align}

The last step towards constructing a shearlet-based Besov seminorm is to introduce directional weights to control the anisotropy of the resulting seminorms.
\begin{definition}\label{def:directionalWeight}
Let $\mathfrak{s}> \mathfrak{s}^* > 0$. A \emph{directional weight} $\omega: \{-1,1\} \times \R \to \R$ is a Lipschitz continuous map such that, for $\iota \in \{-1,1\}$,
$$
 \suppp \omega(\iota, \cdot ) \subset [-\mathfrak{s}, \mathfrak{s}] \qquad \text{ and } \qquad \min_{\iota \in \{-1,1\}}\min_{s \in [-\mathfrak{s}^*,\mathfrak{s}^*]} \omega(\iota, s) >0.
$$
\end{definition}

If $\mathfrak{a}> \mathfrak{a}^* > 0$ and $\mathfrak{s}> \mathfrak{s}^*>0$, and $\psi, K \in L^2(\R^2)$ are such that \eqref{eq:embeddingH1R2} holds and $\omega$ is a directional weight, then it is not hard to see that there exist $0 < c_1\leq c_2$ such that
\begin{align} \label{eq:R2EquivalenceShearlets}
c_1 |f|_{H^1(\R^2)}^2 \leq \int_{\R^2} |\langle f, K_t \rangle|^2 \, \,\mathrm{d}t \ +  & \sum_{\iota = -1,1} \int_{\R^2}\int_{-\mathfrak{s}}^{\mathfrak{s}} \int_{0}^{\mathfrak{a}} a^{-2}|\langle f, \psi_{a,s,t, \iota}^\omega \rangle |^2 a^{-3} \, \mathrm{d}a \, \mathrm{d}s \, \mathrm{d}t \leq c_2 |f|_{H^1(\R^2)}^2.
\end{align}

Finally, we introduce the shearlet-based Besov seminorm.

\begin{definition}
Let $\mathfrak{a}> \mathfrak{a}^* > 0$, and $\mathfrak{s}> \mathfrak{s}^*>0$, and let $\psi, K \in L^2(\R^2)$ be such that \eqref{eq:embeddingH1R2} holds. Let $\omega$ be a directional weight. We define, for $a \in \R^+$, $s\in \R$, $t\in \R^2$:
$$
K_t \coloneqq K(\cdot -t), \quad  \psi_{a,s,t, 1}^{\omega} \coloneqq  \omega(1, s)\psi_{a,s,t}, \quad \text{ and } \quad \psi_{a,s,t, -1}^{\omega} \coloneqq  \omega(-1, s)\widetilde{\psi}_{a,s,t}.
$$
Moreover, for $f\in L^2(\R^2)$, we define the \emph{shearlet-based Besov seminorm $|\cdot|_{B}$ on $\R^2$} by
\begin{align*}
|f|_{B}^2 \coloneqq & \sum_{\iota = -1,1}\int_{\R^2}\int_{-\mathfrak{s}}^{\mathfrak{s}} \int_{0}^{\mathfrak{a}} a^{-2} |\langle f, \psi_{a,s,t, \iota}^\omega \rangle |^2 a^{-3} \, \mathrm{d}a \, \mathrm{d}s \, \mathrm{d}t \in [0,\infty].
\end{align*}
\end{definition}

\subsection{Shearlet transform on a bounded domain and norm-equivalence}\label{sec:NormEquivalence}
To define a shearlet-based Besov seminorm on a bounded domain, we first need to adapt the shearlet system from $\R^2$ accordingly.
A simple procedure to generate a shearlet system on $(0,1)^2$ is to make the elements of the system periodic using the following approach: let $\mathfrak{a}> \mathfrak{a}^* > 0$, and $\mathfrak{s}> \mathfrak{s}^*>0$, and let $\psi, K \in L^2(\R^2)$ be such that \eqref{eq:embeddingH1R2} holds. Now define
\begin{align} \label{eq:constructionOfPeriodicElements}
\psi_{a,s,t,\iota}^{\omega, per} \coloneqq \sum_{k \in \Z^2} \psi_{a,s,t,\iota}^{\omega}(\cdot - k)\qquad \text{ and }\qquad K_t^{per} \coloneqq \sum_{k \in \Z^2} K_t(\cdot - k).
\end{align}
In a different setting of discrete shearlet systems, there have been alternative constructions of shearlets on bounded domains, see \cite{grohs2015anisotropic, petersen2017approximation}. The associated constructions are more involved than the approach via periodisation above. This is because in the discrete setting the approximation properties of the associated systems are analysed, which is not a concern in our approach. Moreover, the approaches in \cite{grohs2015anisotropic, petersen2017approximation} are based on discrete, boundary-adapted wavelet systems, which do not generalise to continuous parameter sets.

Based on the periodic construction of \eqref{eq:constructionOfPeriodicElements}, we proceed to define the \emph{shearlet-based Besov seminorm on $(0,1)^2$} by
\begin{align*}
|f|_{\Bp}^2 \coloneqq \sum_{\iota = -1,1} \int_{(0,1)^2}\int_{-\mathfrak{s}}^{\mathfrak{s}} \int_{0}^{\mathfrak{a}} a^{-2} \left|\left\langle f, \psi_{a,s,t, \iota}^{\omega, per} \right\rangle \right|^2 a^{-3} \, \mathrm{d}a \, \mathrm{d}s \, \mathrm{d}t \in [0,\infty], \quad \text{ for } f \in L^2\left((0,1)^2\right).
\end{align*}
Sometimes it will be necessary to also study a \emph{shearlet-based Besov-functional with nonperiodic elements} which, for $r>0$, is defined as follows:
\begin{align*}
|f|_{B, [-r,r]}^2 \coloneqq \sum_{\iota = -1,1} \int_{[-r,r]^2}\int_{-\mathfrak{s}}^{\mathfrak{s}} \int_{0}^{\mathfrak{a}} a^{-2} \left|\left\langle f, \psi_{a,s,t, \iota}^{\omega} \right\rangle \right|^2 a^{-3} \, \mathrm{d}a \, \mathrm{d}s \, \mathrm{d}t \in [0,\infty],\quad  \text{ for } f \in L^2\left(\R^2\right).
\end{align*}
We shall be interested in the question as to whether $|f|_{\Bp}$ is equivalent to the Sobolev seminorm $|f|_{H^1((0,1)^2)}$.

Since $\langle f, \psi_{a,s,t, \iota}^{\omega, per} \rangle = \langle f^{per}, \psi_{a,s,t, \iota}^{\omega} \rangle$, where $f^{per}$ is a periodised version of $f$, we deduce from \eqref{eq:embeddingH1R2}
that $|\cdot|_{\Bp}$ is \emph{not} equivalent to $|\cdot|_{H^1((0,1)^2)}$ if the situation $f^{per}\not \in H^1(\R^2)$ is possible.
In \cite{dobrosotskaya2011wavelet}, a similar construction is presented for a wavelet-based Besov seminorm.
It is claimed that the Besov seminorm on $[0,1]^2$ defined by a sufficiently regular, periodic wavelet is equivalent to the $H^1((0,1)^2)$ seminorm.
As has just been demonstrated this assertion is invalid if $f^{per}\not \in H^1(\R^2)$ is possible, i.e., if $f$ has non-matching boundary conditions.
However, the result does hold true for compactly supported wavelets over $H_0^1((0,1)^2)$ or if another form of matching boundary condition at $0$ and $1$ is prescribed.
A proof of such a result is not given in \cite{dobrosotskaya2011wavelet}, but
following the arguments of the proposition below and replacing shearlets by wavelets when appropriate one can deduce the asserted equivalence of seminorms in the regime of \cite{dobrosotskaya2011wavelet} as well.

\begin{proposition}\label{prop:equivalence}
Let $A>0$, $B>0$, and $\psi \in L^2(\R^2)$ be such that $\suppp \psi \subset [0,1]^2$ and such that, for $K \in L^2(\R^2)$ with $|\widehat{K}(\xi)|/|\xi| \in [A,B]$ for all $\xi \in [-1,1]^2$ and for all $f\in H^1(\R^2)$, \eqref{eq:R2EquivalenceShearlets} holds. Then, there exists a $B' > 0$ such that
\begin{align}\label{eq:normEqBdDomain}
|f|_{H^1((0,1)^2)}^2 \sim |f|_{\Bp}^2, \quad \text{ for all } f \in H^1_0\left((0,1)^2\right) \text{ with }~ \|f\|_{L^2((0,1)^2)}^2 \leq \frac{|f|_{H^1((0,1)^2)}^2}{2B'}
\end{align}
and
\begin{align}\label{eq:normInEqBdDomain}
|f|_{\Bp}^2 \lesssim |f|_{H^1((0,1)^2)}^2, \quad \text{ for all } f \in H^1_0\left((0,1)^2\right).
\end{align}
\end{proposition}
\begin{proof}
By \eqref{eq:estimateOfKTerm} and \eqref{eq:R2EquivalenceShearlets}, there exist $c_1,c_2 > 0$ such that, for all $f\in H^1(\R^2)$,
\begin{align}\label{eq:VeryImprovedSemiNormInequality}
c_1 \left(|f|_{H^1(\R^2)}^2 - \frac{B}{c_1}\|f\|_{L^2(\R^2) }^2\right) \leq
\sum_{\iota = -1,1} \int_{\R^2}\int_{-\mathfrak{s}}^{\mathfrak{s}} \int_{0}^{\mathfrak{a}} a^{-2} |\langle f, \psi_{a,s,t, \iota}^{\omega} \rangle |^2 a^{-3} \, \mathrm{d}a \, \mathrm{d}s \, \mathrm{d}t \leq c_2 |f|_{H^1(\R^2)}^2.
\end{align}
For a function $f\in H^1_0((0,1)^2)$ and $N_1,N_2 \in \Z$ with $N_1<N_2$, we define $f_{[N_1,N_2]}(x)\coloneqq f(x- \lfloor x \rfloor)$ for all $x \in [N_1,N_2]^2$ and 0 otherwise, where the floor function $\lfloor .  \rfloor$ is applied component-wise. 
It can easily be verified that $f_{[N_1,N_2]}\in H_0^1((N_1,N_2)^2)$ and $|f_{[N_1,N_2]}|_{H^1(\R^2)}^2 = (|N_2|^2-|N_1|^2) |f|_{H^1((0,1)^2)}^2$. Moreover, it is not hard to see that for $N_1 \leq M_1 \leq M_2 \leq N_2$ we have that
\begin{align} \label{eq:TheEstimateOnTheDifferenceOfTwoPeriodicExtensions}
\left|f_{[N_1,N_2]} - f_{[M_1,M_2]}\right|_{H^1(\R^2)}^2 \leq ((N_2-N_1)(M_1 - N_1) + (N_2-N_1)(N_2-M_2)) |f|_{H^1((0,1)^2)}^2.
\end{align}
Without the assumption on the boundary condition for $f$, the conclusion $f_{[N_1,N_2]}\in H_0^1((N_1,N_2)^2)$ would not hold and the proof would fail.

It is not hard to see that there exists an $R \coloneqq R(\mathfrak{a}, \mathfrak{s})\in \N$ such that $\suppp \psi_{a,s,t,\iota} \subset B_{R}(t)$ for all $(a,s,t,\iota) \in (0,\mathfrak{a}] \times [-\mathfrak{s}, \mathfrak{s}] \times \R^2 \times \{-1,1\}$.
For $f \in H_0^1((0,1)^2)$, we have
\begin{align}
|f|_{\Bp}^2 &= \sum_{\iota = -1,1} \int_{0}^{\mathfrak{a}} \int_{-\mathfrak{s}}^{ \mathfrak{s}} a^{-2} \int_{(0,1)^2} \left|\left\langle f, \psi_{a,s,t, \iota}^{\omega, per}\right\rangle_{L^2(\R^2)}\right|^2 a^{-3} \, \mathrm{d}t  \, \mathrm{d}s \, \mathrm{d}a\nonumber\\
& = \sum_{\iota = -1,1} \int_{0}^{\mathfrak{a}} \int_{-\mathfrak{s}}^{ \mathfrak{s}} a^{-2} \int_{(0,1)^2} \left|\left\langle f_{[-R, 1+ R] }, \psi_{a,s,t, \iota}^{\omega}\right\rangle_{L^2(\R^2)}\right|^2 a^{-3} \, \mathrm{d}t  \, \mathrm{d}s \, \mathrm{d}a\label{eq:upperbound}
\eqqcolon  \mathrm{I}.
\end{align}
By the definition of the functions $\psi_{a,s,t, \iota}^{\omega}$, we have that, for every $t^*\in \R^2$,
$$
\left \langle f , \psi_{a,s,t, \iota}^{\omega} \right \rangle =  \left \langle f(\cdot - t^*) , \psi_{a,s,t + t^*, \iota}^{\omega} \right \rangle.
$$
Therefore, we observe that
\begin{align*}
\mathrm{I} =& \ \frac{1}{N^2} \sum_{\iota = -1,1} \int_{0}^{\mathfrak{a}} \int_{-\mathfrak{s}}^{ \mathfrak{s}} a^{-2} \int_{[0,N]^2} \left|\left\langle f_{[-R,N + R]}, \psi_{a,s,t, \iota}^{\omega}\right\rangle_{L^2(\R^2)}\right|^2 a^{-3} \, \, \,\mathrm{d}t \, \mathrm{d}s \, \mathrm{d}a\\
=& \ \frac{1}{N^2} \sum_{\iota = -1,1} \int_{0}^{\mathfrak{a}} \int_{-\mathfrak{s}}^{ \mathfrak{s}} a^{-2} \int_{\R^2} \left|\left\langle f_{[-R,N + R]}, \psi_{a,s,t, \iota}^{\omega}\right\rangle_{L^2(\R^2)}\right|^2 a^{-3} \, \, \,\mathrm{d}t \, \mathrm{d}s \, \mathrm{d}a\\
& \quad  - \frac{1}{N^2} \sum_{\iota = -1,1} \int_{0}^{\mathfrak{a}} \int_{-\mathfrak{s}}^{ \mathfrak{s}} a^{-2} \int_{\R^2 \setminus [0,N]^2} \left|\left\langle f_{[-R,N + R]}, \psi_{a,s,t, \iota}^{\omega}\right\rangle_{L^2(\R^2)}\right|^2 a^{-3} \, \, \,\mathrm{d}t \, \mathrm{d}s \, \mathrm{d}a \eqqcolon\mathrm{II}.
\end{align*}
We proceed by using that $\langle f_{[R,N-R]}, \psi_{a,s,t, \iota}\rangle = 0$, for all $t\not \in  [0,N]^2$ since $\suppp \psi_{a,s,t, \iota} \cap [R,N-R]^2 = \emptyset$.
Moreover, by the seminorm-equivalence of \eqref{eq:VeryImprovedSemiNormInequality}, $c_2 |\,.\,|_{H^1(\R^2)}^2 \geq |\,.\,|_{B}^2$ and $|\,.\,|_{B}^2 \geq c_1 (|\,.\,|_{H^1(\R^2)}^2 - B/c_1 \|\,.\,\|_{L^2(\R^2)}^2)$. We compute that
\begin{align*}
 \mathrm{II} = & \ \frac{1}{N^2} \sum_{\iota = -1,1} \int_{0}^{\mathfrak{a}} \int_{-\mathfrak{s}}^{ \mathfrak{s}} a^{-2} \int_{\R^2} \left|\left\langle f_{[-R,N + R]}, \psi_{a,s,t, \iota}^{\omega}\right\rangle_{L^2(\R^2)}\right|^2 a^{-3} \, \, \,\mathrm{d}t \, \mathrm{d}s \, \mathrm{d}a\\
& \quad  - \frac{1}{N^2} \sum_{\iota = -1,1} \int_{0}^{\mathfrak{a}} \int_{-\mathfrak{s}}^{ \mathfrak{s}} a^{-2} \int_{\R^2} \left|\left\langle f_{[-R,N + R]} - f_{[R,N- R]}, \psi_{a,s,t, \iota}^{\omega}\right\rangle_{L^2(\R^2)}\right|^2 a^{-3} \, \, \,\mathrm{d}t \, \mathrm{d}s \, \mathrm{d}a\\
 \geq & \ \frac{c_1}{N^2} \left(\left|f_{[-R,N + R]}\right|_{H^1(\R^2)}^2 - \frac{B}{c_1}\left\|f_{[-R,N + R]}\right\|_{L^2(\R^2)}^2\right) - \frac{c_2}{N^2} \left|f_{[-R,N+ R]} - f_{[R,N- R]}\right|_{H^1(\R^2)}^2\\
= & \ \frac{c_1 ((N+2R)^2)}{N^2}\left(\left|f_{[0,1]}\right|_{H^1((0,1)^2)}^2 - \frac{B}{c_1}\left\|f_{[0,1]^2}\right\|_{L^2(\R^2)}^2\right) - \frac{c_2}{N^2} \left|f_{[-R,N+ R]} - f_{[R,N- R]}\right|_{H^1(\R^2)}^2\\
\geq & \ \frac{c_1 ((N+2R)^2)}{2N^2}\left|f_{[0,1]}\right|_{H^1((0,1)^2)}^2  - \frac{2(N+2R) 2R c_2}{2N^2} \left|f_{[0,1]}\right|_{H^1(\R^2)}^2,
\end{align*}
where the last line holds thanks to the second bound \eqref{eq:normEqBdDomain} with $B' \coloneqq B/c_1$ and \eqref{eq:TheEstimateOnTheDifferenceOfTwoPeriodicExtensions}.
Since $N$ was arbitrary, this completes the proof of the lower bound of the Besov seminorm by the Sobolev seminorm.
The upper bound follows trivially from the norm-equivalence on $\R^2$, equation \eqref{eq:upperbound}, and $|f_{[0,2]}|_{H^1(\R^2)} \leq 2|f|_{H^1((0,1)^2)}$.
\end{proof}

\subsection{The shearlet-based Ginzburg--Landau energy} \label{sec:SGLE}

We start by recalling the definition and some properties of the classical anisotropic Ginzburg--Landau energy.
Let $\aninorm$ be a norm on $\R^2$ and $\mathcal{W}(u) = u^2 (1-u)^2$; then, for $u \in BV$,
\begin{align*}
   \mathrm{GL}_\varepsilon^{\aninorm}(u) \coloneqq \left\{ \begin{array}{ll}
   \varepsilon \int_{(0,1)^2} \aninorm(\nabla u)^2 \, \,\mathrm{d}x + \frac{1}{4\varepsilon}\int_{(0,1)^2}\mathcal{W}(u)(x) \, \,\mathrm{d}x, & \text{ if } u \in H^1\left((0,1)^2\right),\\
    \infty, & \text{ if } u \in BV \setminus H^1\left((0,1)^2\right). \end{array}\right.
\end{align*}
The Ginzburg--Landau energies as defined above $\Gamma$-converge to an anisotropic perimeter functional for $\varepsilon \to 0$, see \cite[Theorem 4.13]{braides1998approximation}.
Let us recall the definition of $\Gamma$-convergence before making the previous statement more precise.

\begin{definition}
For a space $X$, we say that a sequence of functionals $F_\varepsilon: X \to [-\infty, \infty]$ \emph{$\Gamma$-converges} to a functional $F: X \to [-\infty, \infty]$ for $\varepsilon \to 0$, if the following two conditions are satisfied:

\begin{enumerate}
\item \emph{lim-inf condition}: for all $u \in X$ and all sequences $(u_\varepsilon)_{\varepsilon>0} \subset X$ such that $u_\varepsilon \to u$, for $\varepsilon \to 0$, we have that
\begin{align}\label{eq:EverySequence}
\liminf_{\varepsilon\to 0} F_\varepsilon(u_\varepsilon) \geq F(u);
\end{align}
\item \emph{recovery sequence}: there exists a \emph{recovery sequence} $(\widetilde{u}_\varepsilon)_{\varepsilon>0}\subset X$ with $\widetilde{u}_\varepsilon \to u$, for $\varepsilon \to 0$, such that
\begin{align*}
\limsup_{\varepsilon\to 0} F_\varepsilon(\widetilde{u}_\varepsilon) \leq F(u).
\end{align*}
\end{enumerate}
\end{definition}

For a norm $\aninorm$ on $\R^2$, we define, the \emph{perimeter functional}
$$
P_\aninorm(u)\coloneqq \left\{ \begin{array}{ll}
   c \int_{S(u)} \aninorm(\vn_x)\, \mathrm{d}\mathcal{H}_1, & \text{ if } u \in SBV \text{ and } u \in \{0,1\} \text{ a.e.},\\
    \infty, &\text{ otherwise}, \end{array}\right.
$$
where $c \coloneqq \int_{0}^1 \sqrt{\mathcal{W}(s)}\, \mathrm{d} s$, $S(u)$ denotes the jump-set of $u$ (see \cite[Definition 3.91]{ambrosio2000functions}) and $\vn_x$ is the measure theoretic outer normal of $S(u)$ at $x$ (see \cite[Definition 3.54]{ambrosio2000functions} for a definition). As was already previously announced, we have that $\mathrm{GL}_\varepsilon^{\aninorm}$ $\Gamma$-converges to $P_\aninorm$ by \cite[Theorem 4.13]{braides1998approximation}.

Next, we aim at constructing a corresponding shearlet-based Ginzburg--Landau energy. We begin by placing some assumptions on the underlying shearlet systems and the weight $\omega$.
\begin{assumption}\label{assump:1}
We require the generator function $\psi\in L^2(\R^2)$ to satisfy \eqref{eq:TheAssumption} and
$\suppp \psi \subset [0,1]^2$.
Additionally, let $\mathfrak{s}^*>0$ and $\mathfrak{a}^*>0$ be such that the shearlet transform with $\mathfrak{s}>\mathfrak{s}^*$ and $\mathfrak{a}>\mathfrak{a}^*$ satisfies \eqref{eq:R2EquivalenceShearlets}.
\end{assumption}
For any norm $\aninorm$ in $\R^2$, we say that $\omega$ is a \emph{directional weight associated with $\aninorm$} if 
\begin{align}\label{eq:associatedNormDefinition}
 \max\left\{|\vn_1|, |\vn_2|\right\}\left(|\vn_1| \omega\left(1, \frac{\vn_2}{\vn_1}\right)^2 + |\vn_2|\omega\left(-1, \frac{\vn_1}{\vn_2}\right)^2 \right) &= \aninorm( \vn )^2, \quad \text{ for all } \vn \in \mathbb{S}^1.
\end{align}
where we define $\omega(\pm 1, \pm \infty) \coloneqq 0$. 

Given $\aninorm$ and $\mathfrak{s}>\mathfrak{s}^* > 0$ there always exists a canonical choice of $\omega$ so that $\omega$ satisfies the conditions of Definition \ref{def:directionalWeight}. The construction of such an $\omega$ is slightly technical and is therefore deferred to Appendix \ref{app:canonicalWeight}.

For a given norm and an appropriate directional weight, we can define the following space of functions.

\begin{definition}
Let $\aninorm$ be a norm on $\R^2$ and let $|\cdot|_{\Bp}$ be a shearlet-based Besov seminorm from a weighted shearlet system with directional weight $\omega$, where $\omega$ is a directional weight associated with $\aninorm$. We define, for $U(x) \coloneqq x/\log_2(2 + x)$, the set
$$
\mathcal{B}_{\aninorm} \coloneqq \left\{u \in H^1(\aninorm): |u|_{\Bp}^2 \geq \int_{(0,1)^2} \aninorm\left(\nabla u(x)\right)^2 \, \,\mathrm{d}x  - U\left(|u|_{H^1((0,1)^2)}^2\right)
\right\}.
$$
\end{definition}
$\mathcal{B}_{\aninorm}$ is the set of functions that we will use to define the shearlet-based Ginzburg--Landau energy.
Admittedly, the definition of $\mathcal{B}_{\aninorm}$ is very specific to the shearlet system.
Nonetheless, we will observe that $\mathcal{B}_{\aninorm}$ contains all functions that admit a smooth phase-transition
along a polygonal curve as well as smooth perturbations of such functions.
For a more detailed analysis of this set, we refer to Subsection \ref{sec:TheSetB}.
\begin{definition} \label{def:ShearBasedGizLandEnergy}
Let $\aninorm$ be a norm on $\R^2$ and let $|\cdot|_{\Bp}$ be a shearlet-based Besov seminorm from a weighted shearlet system with weight $\omega$, where $\omega$ is a directional weight associated with $\aninorm$. We define, for $u \in BV$,
\begin{align} \label{eq:ShearBasedFunctional}
\mathrm{{SGL}}^\omega_\varepsilon(u) \coloneqq \left\{ \begin{array}{l l}
   \varepsilon| u |_{\Bp}^2 + \frac{1}{4\varepsilon}\int_{(0,1)^2}\mathcal{W}(u)(x) \, \,\mathrm{d}x, & \text{ if } u \in \mathcal{B}_\aninorm,\\
    \infty, & \text{ if } u \in BV \setminus \mathcal{B}_\aninorm. \end{array}\right.
\end{align}
\end{definition}
It turns out that the sequence of energies $\mathrm{SGL}^\omega_\varepsilon$ described above $\Gamma$-converges
to a perimeter functional with norm $\aninorm$, where $\aninorm$ is such that $\omega$ is the associated directional weight.

\begin{theorem}\label{thm:main}
Let $\aninorm$ be a norm on $\R^2$ and let $|\cdot|_{\Bp}$ be a shearlet-based Besov seminorm with a shearlet system satisfying Assumption \ref{assump:1}, where $\psi = \psi^1 \otimes \phi^1$ and
$$
	\left|\widehat{\phi^1}(\xi)\right| \lesssim (1+|\xi|)^{-4} \quad \text{ and } \quad \left|\widehat{\psi^1}(\xi)\right| \lesssim 		\frac{\min\{ |\xi|, 1\}^4}{(1+|\xi|)^4}, \quad  \text{ for all } \xi \in \R.
$$
Moreover, let the directional weight of $|\cdot|_{\Bp}$ be associated with $\aninorm$.
Then, for all $u \in SBV$, $u(x) \in \{0,1\}$ for almost every $x \in (0,1)^2$, and $(u_\varepsilon)_{\varepsilon>0} \subset H^1((0,1)^2)$ such that $u_\varepsilon \to u$ in $L^1((0,1)^2)$, for $\varepsilon \to 0$, we have
\begin{align*}
	\liminf_{\varepsilon\to 0} \mathrm{SGL}^\omega_\varepsilon(u_\varepsilon) \geq P_\aninorm(u).
\end{align*}
Additionally, for every set $D\subset (0,1)^2$ of finite perimeter we have that there exists a sequence $(\widetilde{u}_\varepsilon)_{\varepsilon>0}\subset H^1_0((0,1)^2)$ such that $\widetilde{u}_\varepsilon \to \chi_D$ in $L^1((0,1)^2)$, for $\varepsilon \to 0$, and
$$
	\lim_{\varepsilon\to 0} \mathrm{SGL}^\omega_\varepsilon(\widetilde{u}_\varepsilon) = P_\aninorm(\chi_D).
$$
\end{theorem}
\begin{proof}
 We shall develop the proof in Section \ref{sec:proofOfThmMain} of the paper. The liminf condition will be verified in Proposition \ref{prop:limInf} and the existence of a recovery sequence will be demonstrated in Proposition \ref{prop:PieceSmoothConvergence}.
\end{proof}

\section{Proof of Theorem \ref{thm:main}} \label{sec:proofOfThmMain}

\subsection{The liminf condition}

The definitions of $\mathcal{B}_{\aninorm}$ and $\mathrm{SGL}^\omega_\varepsilon$ immediately allow us to obtain the liminf condition, equation \eqref{eq:EverySequence}, for the shearlet-based Ginzburg--Landau energy.
\begin{proposition}\label{prop:limInf}
Let $(u_\varepsilon)_{\varepsilon>0} \subset H^1((0,1)^2)$, $u\in SBV$, $u(x) \in \{0,1\}$ for almost every $x \in (0,1)^2$, and let $u_\varepsilon \to u$ in $L^1((0,1)^2)$, for $\varepsilon \to 0$; then,
\begin{align}\label{eq:TheLimInfCondition}
\liminf_{\varepsilon>0} \mathrm{SGL}^\omega_\varepsilon(u_\varepsilon)  \geq P_\aninorm(u).
\end{align}
\end{proposition}
\begin{proof}




There exists a subsequence $(u_{\varepsilon_n})_{n \in \N}$ of $(u_{\varepsilon})_{\varepsilon > 0}$ such that 
\begin{align}\label{eq:SubsequenceArgument}
\lim_{n \to \infty} \mathrm{SGL}^\omega_{\varepsilon_n}(u_{\varepsilon_n}) = \liminf_{\varepsilon \to 0} \mathrm{SGL}^\omega_\varepsilon(u_\varepsilon).
\end{align}



If $\liminf_{n  \to \infty} \mathrm{SGL}^\omega_{\varepsilon_n}(u_{\varepsilon_n}) = \infty$, then \eqref{eq:TheLimInfCondition} follows trivially. 
Hence, we assume that 
\begin{align}\label{eq:finitenessAssumptionSubsequence}
\lim_{n \to \infty} \mathrm{SGL}^\omega_{\varepsilon_n}(u_{\varepsilon_n}) < \infty.
\end{align}
We have by the definition of $(\mathrm{SGL}^\omega_\varepsilon)_{\varepsilon>0}$ and \eqref{eq:finitenessAssumptionSubsequence} that there exists an $N_0 \in \N$ such that for every $n \in \N$, $n \geq N_0$, 
\begin{align}
    \mathrm{SGL}^\omega_{\varepsilon_n}(u_{\varepsilon_n})= \varepsilon_n| u_{\varepsilon_n} |_{\Bp}^2 + \frac{1}{4{\varepsilon_n}}\int_{(0,1)^2}\mathcal{W}(u_{\varepsilon_n})(x)\, \mathrm{d}x \geq
     \mathrm{GL}_{\varepsilon_n}(u_{\varepsilon_n}) - {\varepsilon_n} U\left(|u_{\varepsilon_n}|_{H^1((0,1)^2)}^2\right). \label{eq:LowerBoundWithU}
\end{align}
If $(u_{\varepsilon_n})_{n \in \N}$ is such that ${\varepsilon_n} U(|u_{\varepsilon_n}|_{H^1((0,1)^2)}^2) \to 0$ for $n \to \infty$, then, by \eqref{eq:LowerBoundWithU},
\begin{align} \label{eq:LowerBoundByPerimeterFunctional}
\liminf_{n \to \infty} \mathrm{SGL}^\omega_{\varepsilon_n}(u_{\varepsilon_n})
 \geq \liminf_{n \in \N} \mathrm{GL}(u_{\varepsilon_n}) \geq P_\aninorm(u).
\end{align}
If $(u_{\varepsilon_n})_{n \in \N}$ is such that
${\varepsilon_n} U(|u_{\varepsilon_n}|_{H^1((0,1)^2)}^2) \not\to 0$ as $n \to \infty$, then, since $U(x) = o(x)$ for $x \to \infty$, we have that $\limsup_{n \in \N} \varepsilon_n|u_{\varepsilon_n}|_{H^1((0,1)^2)}^2 = \infty$. In the remainder of the proof, we will show that if $\limsup_{n \to \infty} {\varepsilon_n}|u_{\varepsilon_n}|_{H^1((0,1)^2)}^2 = \infty$, then we have that $\lim_{n \to 0} \mathrm{SGL}^\omega_{\varepsilon_n}(u_{\varepsilon_n}) = \infty$, which contradicts \eqref{eq:finitenessAssumptionSubsequence}.

We assume that $\limsup_{n \to \infty} {\varepsilon_n}|u_{\varepsilon_n}|_{H^1((0,1)^2)}^2 = \infty$ and consider two cases:

\textbf{Case 1:} $\limsup_{n\to \infty} \|u_{{\varepsilon_n}}\|_{L^2((0,1)^2)} < \infty$.

In this case, we have that, since $\limsup_{n \to \infty} {\varepsilon_n}|u_{{\varepsilon_n}}|_{H^1((0,1)^2)}^2 = \infty$, there exists a subsequence $(u_{{\varepsilon'_n}})_{n\in \N}$ of $(u_{{\varepsilon_n}})_{n\in \N}$ such that 
$\lim_{n \to \infty}|u_{{\varepsilon'_n}}|_{H^1((0,1)^2)}^2 = \infty$. Therefore, for any $B'>0$, there exists an $N_1\in \N$ such that 
$$
	\sup_{m \in \N} \|u_{{\varepsilon'_m}}\|_{L^2((0,1)^2)} \leq \frac{|u_{{\varepsilon'_n}}|_{H^1((0,1)^2)}^2}{2B'},
$$
for all $n > N_1$. We conclude by the seminorm-equivalence of \eqref{eq:normEqBdDomain}, that 
$$
	\liminf_{n \to \infty} \varepsilon'_n |u_{{\varepsilon'_n}}|_{\Bp}^2 \gtrsim \lim_{n \to \infty} \varepsilon'_n |u_{{\varepsilon'_n}}|_{H^1((0,1)^2)}^2 = \infty,
$$
and hence $\lim_{n \to \infty} \mathrm{SGL}^\omega_{\varepsilon_n}(u_{\varepsilon_n}) = \lim_{n \to \infty} \mathrm{SGL}^\omega_{\varepsilon'_n}(u_{\varepsilon'_n}) = \infty$.

\textbf{Case 2:} $\limsup_{n \to \infty} \|u_{\varepsilon_n}\|_{L^2((0,1)^2)} = \infty$.

In this case, we define $R_{\varepsilon_n} \coloneqq \{ x \in (0,1)^2: |u_{{\varepsilon_n}}(x)|< 2\}$ and obtain that
\begin{align} \label{eq:TheNormThatEqualsInfinity}
\infty = \limsup_{n \to \infty} \|u_{\varepsilon_n}\|_{L^2((0,1)^2)} \leq \limsup_{n \to \infty} \left(\left\|u_{\varepsilon_n}\chi_{R_{\varepsilon_n}}\right\|_{L^2((0,1)^2)} + \left\|u_{\varepsilon_n} \chi_{(0,1)^2 \setminus R_{\varepsilon_n}}\right\|_{L^2((0,1)^2)}\right).
\end{align}
We have that $\limsup_{n \to \infty} \|u_{\varepsilon_n}	\chi_{R_{\varepsilon_n}}\|_{L^2((0,1)^2)}^2 \leq 4 < \infty.$
As a consequence of \eqref{eq:TheNormThatEqualsInfinity}, we therefore conclude that
\begin{align}\label{eq:192874618638}
 \limsup_{n \to \infty} \left\|u_{\varepsilon_n} \chi_{(0,1)^2 \setminus R_{\varepsilon_n}}\right\|_{L^2((0,1)^2)} = \infty,
\end{align}
and hence
\begin{align*}
\int_{(0,1)^2} \mathcal{W}(u_{\varepsilon_n})(x) \, \,\mathrm{d}x &\geq \int_{(0,1)^2 \setminus R_{\varepsilon_n}} u_{\varepsilon_n}(x)^2(1-u_{\varepsilon_n}(x))^2 \, \,\mathrm{d}x \\
&\geq \int_{(0,1)^2 \setminus R_{\varepsilon_n}} |u_{\varepsilon_n}(x)|^2 \, \,\mathrm{d}x \to \infty \text{ for } n \to \infty,
\end{align*}
by \eqref{eq:192874618638}, at least up to a subsequence. This implies that $\lim_{n \to \infty} \mathrm{SGL}^\omega_{\varepsilon_n}(u_{\varepsilon_n}) = \infty$. 
\end{proof}

We see that the definition of $\mathcal{B}_{\aninorm}$ together with the convergence of the classical Ginzburg--Landau energy, directly yields the liminf condition of the $\Gamma$-convergence for $(\mathrm{SGL}_{\varepsilon}^\omega)_{\varepsilon>0}$.
The existence of a recovery sequence, on the other hand, necessitates a more refined analysis of the behaviour of $\mathrm{SGL}^\omega_\varepsilon$ and in particular $| \cdot |_{\Bp}$.
We will start by analysing $| \cdot |_{\Bp}$ for characteristic functions of sets with linear boundaries, then extend the estimate to characteristic sets of polygons and finally use a density argument to conclude the
behaviour for characteristic functions of sets of finite perimeter.

\subsection{Estimating the shearlet-based Ginzburg--Landau energy for Heaviside functions}
Let $H \coloneqq \chi_{\R^- \times \R}$ be the \emph{vertical Heaviside function} and $\widetilde{H}\coloneqq \chi_{\R \times \R^-}$ be the \emph{horizontal Heaviside function}. We define for $\vec{n} \in \mathbb S^1$ the \emph{Heaviside function with normal $\vec{n}$} by
\begin{align*}
H^{\vec{n}}(x) \coloneqq \left \{ \begin{array}{cc}
    H(S_{\vn_2/\vn_1}x), & \text{ if } \left|\frac{\vn_2}{\vn_1}\right| \leq 1 , \\
    \widetilde{H}({S}_{\vn_1/\vn_2}^Tx), & \text{ if } \left|\frac{\vn_1}{\vn_2}\right| < 1,
\end{array} \right. \text{ for } x \in \R^2.
\end{align*}
One readily verifies that $H^{\vec{n}}$ is well-defined and has a jump across a line through the origin with normal $\vec{n}$. Note, that if $\vn_2/\vn_1>0$, then $H(S_{\vn_2/\vn_1}x) = \chi_{x_1+ \vn_2/\vn_1 x_2 \leq 0} = \widetilde{H}(S^T_{\vn_1/\vn_2} x)$.
If $\vn_2/\vn_1<0$, then we instead have that $H(S_{\vn_2/\vn_1}x) = \chi_{x_1+ \vn_2/\vn_1 x_2 \leq 0} = 1-\widetilde{H}(S^T_{\vn_1/\vn_2} x)$ almost everywhere.
In the sequel, we always analyse inner products of $H^{\vec{n}}(x)$ with functions $\psi$ that have vanishing moments and by the previous considerations we will always have
$\langle H(S_{\vn_2/\vn_1}), \psi \rangle = \langle \widetilde{H}(S^T_{\vn_1/\vn_2} \cdot ), \psi \rangle$ whenever $\vn_1 \neq 0$ and $\vn_2 \neq 0$.
Additionally, we will be interested in functions that admit a smooth sharp transition across a line.
A function $\Xi: [-1/2,1/2] \to [-1/2,1/2]$ is called a \emph{transition profile} if
$$
\Xi \in C^1\left(\left[-\frac12,\frac12\right]\right), \quad \Xi\left(-\frac12\right) = -\frac12, \quad \text{  and } \quad \Xi\left(\frac12\right) = \frac12.
$$
Prototypes of functions with smooth sharp transitions across a line can now be defined as follows: for $\varepsilon>0$, $\rho > 0$, and $x \in \R^2$ let 
\begin{align*}
H_{\varepsilon, \Xi, \rho}(x)\coloneqq &\left\{ \begin{array}{ll}
    1, & \text{ if } x_1 < - \frac{\rho\varepsilon}{2}, \\
    \frac{1}{2}  - \Xi\left(\frac{x_1}{\rho \varepsilon}\right),& \text{ if } -\frac{\rho\varepsilon}{2} \leq x_1 \leq \frac{\rho\varepsilon}{2},\\
    0, & \text{ if }  \frac{\rho\varepsilon}{2}< x_1, \\
\end{array}\right.  \\
\widetilde{H}_{\varepsilon, \Xi, \rho}(x)\coloneqq& \left\{ \begin{array}{ll}
    1, & \text{ if } x_2 < - \frac{\rho\varepsilon}{2}, \\
    \frac{1}{2}  - \Xi\left(\frac{x_2}{\rho \varepsilon}\right),& \text{ if } - \frac{\rho\varepsilon}{2} \leq x_2 \leq  \frac{\rho\varepsilon}{2},\\
    0, & \text{ if }  \frac{\rho\varepsilon}{2}< x_2. \\
\end{array}\right.
\end{align*}
In the definition of $H_{\varepsilon, \Xi, \rho}$ above, $\varepsilon$ and $\rho$ have practically the same role. However, when we study the $\Gamma$-convergence of energies applied to these functions, we will later use $\varepsilon$ 
tied to the corresponding $\varepsilon$ parameter of the energies $(SGL_\varepsilon^\omega)_{\varepsilon >0}$ and use $\rho$ as a fine-tuning parameter.

In the sequel, we also need to analyse functions with smooth transitions
accros lines that are not parallel to the horizontal or vertical axes.
To model this, we rotate the functions $H_{\varepsilon, \Xi, \rho}$ and $\widetilde{H}_{\varepsilon, \Xi, \rho}$ described above, but to make the analysis in the sequel simpler,
we present a construction based on shearing instead of rotation.
We define, for $x \in \R^2$,
\begin{align*}
H_{\varepsilon, \Xi, \rho}^{\vn, 1}(x)&\coloneqq H_{\varepsilon, \Xi, \rho/{|\vn_1|}}( S_{\vn_2/\vn_1}x ), \text{ if } \vn_1 \neq 0, \quad \text{ and } \quad
H_{\varepsilon, \Xi, \rho}^{\vn, -1}(x)\coloneqq \widetilde{H}_{\varepsilon, \Xi, \rho/{|\vn_2|}}\left( S^T_{\vn_1/\vn_2}x \right)\text{ if } \vn_2 \neq 0.
\end{align*}
Clearly, $H_{\varepsilon, \Xi, \rho}^{\vn, 1}$ and $H_{\varepsilon, \Xi, \rho}^{\vn, -1}$ are functions that are constant along directions perpendicular to $\vn$ and behave like $\Xi(\cdot/(\rho \varepsilon))$ or $\Xi(- \cdot/(\rho \varepsilon))$
along $\vn$. This implies that for all functions $\psi$ with a vanishing moment we have
\begin{align*}
\left\langle H_{\varepsilon, \Xi, \rho}^{\vn, 1}, \psi \right\rangle  = \left\langle 1-H_{\varepsilon, \widetilde{\Xi}, \rho}^{\vn, -1}, \psi \right\rangle = \left\langle H_{\varepsilon, \widetilde{\Xi}, \rho}^{\vn, -1}, \psi \right\rangle,
\end{align*}
as long as $\vn_1 \neq 0$ and $\vn_2 \neq 0$, where $\widetilde{\Xi} = \Xi$ or $\widetilde{\Xi} = 1-\Xi(-\cdot)$ depending on $\vn$. The exact form of $\widetilde{\Xi}$ will never appear in the sequel. The only relation needed is that
$$
\left|H_{\varepsilon, \widetilde{\Xi}, \rho}^{\vn, -1}\right|_{H^1((-r,r)^2)} = \left|H_{\varepsilon, \Xi, \rho}^{\vn, -1}\right|_{H^1((-r,r)^2)}.
$$
Finally, we define, for $x \in \R^2$,
\begin{align}\label{def:PhaseTrans}
    H_{\varepsilon, \Xi, \rho}^{\vn}(x) \coloneqq  \left \{ \begin{array}{cc}
    H_{\varepsilon, \Xi, \rho}^{\vn, 1}(x),  & \text{ if } \left|\frac{\vn_2}{\vn_1}\right| \leq 1 , \\
    H_{\varepsilon, \Xi, \rho}^{\vn, -1}(x), & \text{ if } \left|\frac{\vn_1}{\vn_2}\right| < 1.
\end{array} \right.
\end{align}
It is not hard to see, considering the length of the jump curve and the rotation invariance of the $H^1$ seminorm, that
\begin{align}
\left|H_{\varepsilon,\Xi, \rho}^\vn\right|_{H^1((-r,r)^2)}^2 =& \ \sqrt{\left(1+ \min\left\{ \left| \frac{\vn_2}{\vn_1} \right|^2, \left| \frac{\vn_1}{\vn_2} \right|^2\right \} \right)}\left|H_{\varepsilon, \Xi, \rho}^{(1,0)}\right|_{H^1((-r,r)^2)}^2 + \mathcal{O}(\varepsilon), \text{ for }\varepsilon \to 0.\label{eq:rotation}
\end{align}
Additionally, we observe that, for $\lambda\neq 0 $ and $\varepsilon>0$ sufficiently small
\begin{align}\label{eq:rescaling}
\left|H_{\varepsilon, \Xi, \rho/{|\lambda|}}^{(1,0)}\right|_{H^1((-r,r)^2)}^2 =  |\lambda| \left|H_{\varepsilon, \Xi, \rho}^{(1,0)}\right|_{H^1((-r,r)^2)}^2.
\end{align}

Understanding the asymptotic behaviour of the shearlet-based Besov seminorm of $H_{\varepsilon,\Xi, \rho}^\vn$
requires a bit more technical work.
We will observe below that the shearlet-based Besov seminorm of $H_{\varepsilon,\Xi, \rho}^\vn$
is asymptotically equivalent to the $H^1$ seminorm of $H_{\varepsilon,\Xi, \rho}^\vn$.
We study the nonperiodic Besov functional here because
we will later only apply the following proposition to analyse functions supported away
from the boundary of our domain. 

\begin{proposition}\label{prop:HeavisideConvergence}
Let $\psi \in L^2(\R^2)$, $\mathfrak{s} >\mathfrak{s}^*>0$, and $\mathfrak{a} >\mathfrak{a}^*>0$ satisfy Assumption \ref{assump:1} and let $\psi(x_1,x_2) = \psi^1(x_1)\phi^1(x_2)$, for all $(x_1,x_2) \in \R^2$ and
\begin{align}\label{eq:theDecayAssumption}
|\widehat{\phi^1}(\xi)| \lesssim (1+|\xi|)^{-4}\quad \text{ and } \quad |\widehat{\psi^1}(\xi)| \lesssim \min\left\{ |\xi|, 1\right\}^4/(1+|\xi|)^4, \quad \text{ for all } \xi \in \R.
\end{align}
Let $\omega$ be a directional weight associated with a norm $\mathcal{N}$ and let $r\in (0,1/2)$, $\rho >0$.
Then, we have that, for all $\varepsilon > 0$ with $r/((\mathfrak{s} + 2)\rho) > \varepsilon$:
\begin{align}\label{eq:PropHeavisideEquation}
|H_{\varepsilon, \Xi, \rho}^\vn|_{B, [-r,r]^2}^2 = & \ \aninorm(\vn)^2 \left|H_{\varepsilon, \Xi, \rho}^\vn\right|_{H^1((-r,r)^2)}^2\\
&\qquad +  r \, o\left(\frac{1}{{\varepsilon}}\log_2\left(\frac{1}{{\varepsilon}}\right)^{-1}\right) + o\left(\frac{1}{\sqrt{\varepsilon}}\right) + \mathcal{O} \left(\frac{1}{(r- \rho \varepsilon(\mathfrak{s} + 2))^4} \right), \nonumber
\end{align}
where the implied constant in the first asymptotic term depends quadratically
on $1+\|\Xi'\|_{\infty}/\rho$, and the implied constants in the second and third asymptotic terms
depend quadratically on $1+\|\Xi'\|_{\infty}$.
\end{proposition}
\begin{proof}
The proof proceeds in several steps.

\step{1}{Splitting into cones}
First, we split the shearlet-based Besov-functional into two parts, each associated to one of the cones of the shearlet system:
\begin{align*}
\mathrm{I}_1^{\varepsilon, \rho, \omega, \vn} \coloneqq& \  \int_0^\mathfrak{a} \int_{-\mathfrak{s}}^{\mathfrak{s}} \int_{[-r,r]^2} a^{-2} \left|\left\langle H_{\varepsilon, \Xi, \rho}^\vn, \psi_{a,s,t,1}^\omega \right\rangle\right|^2 a^{-3} \, \, \,\mathrm{d}t \, \mathrm{d}s \, \mathrm{d}a, \\
\mathrm{I}_{-1}^{\varepsilon, \rho, \omega, \vn} \coloneqq& \  \int_0^\mathfrak{a} \int_{-\mathfrak{s}}^{\mathfrak{s}} \int_{[-r,r]^2} a^{-2}  \left|\left\langle H_{\varepsilon, \Xi, \rho}^\vn, \psi_{a,s,t,-1}^\omega\right\rangle\right|^2 a^{-3} \, \, \,\mathrm{d}t \, \mathrm{d}s \, \mathrm{d}a.
\end{align*}
We have that $|H_{\varepsilon, \Xi, \rho}^\vn|_{B, [-r,r]^2}^2 =\mathrm{I}_1^{\varepsilon, \rho, \omega, \vn} + \mathrm{I}_{-1}^{\varepsilon, \rho, \omega, \vn}$.
If $\vn_1 \neq 0$, then, by the considerations after the definition of $H_{\varepsilon, \Xi, \rho}^\vn$ in \eqref{def:PhaseTrans}, we have that
\begin{align*}
\mathrm{I}_{1}^{\varepsilon, \rho, \omega, \vn} &=  \int_0^\mathfrak{a} \int_{-\mathfrak{s}}^{\mathfrak{s}} \int_{[-r,r]^2} a^{-2} \left|\left\langle H_{\varepsilon, \widetilde{\Xi}, \rho/{|\vn_1|}}( S_{\vn_2/\vn_1}x ), \psi_{a,s,t,1}^\omega\right \rangle\right|^2a^{-3} \, \, \,\mathrm{d}t \, \mathrm{d}s \, \mathrm{d}a,
\end{align*}
where $\widetilde{\Xi} = \Xi$ or $\widetilde{\Xi} = 1-\Xi(-\cdot)$. We assume in the sequel that $\widetilde{\Xi} = \Xi$. The case $\widetilde{\Xi} = 1-\Xi(-\cdot)$ follows similarly. If $\vn_1 = 0$, then $H_{\varepsilon, \Xi, \rho}^\vn= H_{\varepsilon, \Xi, \rho}^{\vn, -1} = \widetilde{H}_{\varepsilon, \Xi, \rho/{|\vn_2|}}$ is constant along the $x_1$-direction. As each $\psi_{a,s,t,1}$ has a vanishing moment in the $x_1$-direction, this then implies that $\mathrm{I}_{1}^{\varepsilon, \rho, \omega, \vn}  = 0$.

Following the same argument as before, we get that if $\vn_2 \neq 0$, then
\begin{align*}
\mathrm{I}_{-1}^{\varepsilon, \rho, \omega, \vn}  &= \int_0^\mathfrak{a} \int_{-\mathfrak{s}}^{\mathfrak{s}} \int_{[-r,r]^2} a^{-2}  \left|\left \langle \widetilde{H}_{\varepsilon, \widetilde{\Xi}, \rho/{|\vn_2|}}\left( S^T_{\vn_1/\vn_2}x \right), \psi_{a,s,t,-1}^\omega\right \rangle\right|^2 a^{-3} \, \, \,\mathrm{d}t \, \mathrm{d}s \, \mathrm{d}a,
\end{align*}
where $\widetilde{\Xi} = \Xi$ or $\widetilde{\Xi} = 1-\Xi(-\cdot)$. We assume that $\widetilde{\Xi} = \Xi$; the other case follows similarly. If $\vn_2 = 0$, then $\mathrm{I}_{-1}^{\varepsilon, \rho, \omega, \vn}  = 0$.

\step{2}{Partial integration}
We define
$$
h_{\varepsilon, \Xi, \rho/{|\vn_1|}}^1 \coloneqq \frac{\partial }{\partial x_1} H_{\varepsilon, \Xi, \rho/{|\vn_1|}} \quad \text{ and } \quad
h_{\varepsilon, \Xi, \rho/{|\vn_2|}}^{-1} \coloneqq \frac{\partial }{\partial x_2} \widetilde{H}_{\varepsilon, \Xi, \rho/{|\vn_2|}},
$$
and we denote from now on $\theta \coloneqq -\vn_2/\vn_1$ to shorten the notation. Setting $\mu \coloneqq \mu^1 \otimes \phi^1$, with $\widehat{\mu^1}(\xi) \coloneqq \frac{1}{2\pi i \xi} \widehat{\psi^1}(\xi)$, for $\xi \in \R$,
we have that $\mu^1$ is the antiderivative of $\psi^1$, and, since $\psi^1$ has a vanishing moment, we have that $\suppp \mu^1 \subset [0,1]$. Moreover, 
\begin{align} \label{eq:VanishingMomentsMu}
\left|\widehat{\phi^1}(\xi)\right| \leq (1+|\xi|)^{-4} \quad \text{ and } \quad \left|\widehat{\mu^1}(\xi)\right| \leq \frac{\min\{ |\xi|, 1\}^3}{(1+|\xi|)^4} \leq \frac{\min\{ |\xi|, 1\}^2}{(1+|\xi|)^4}, \quad  \text{ for all } \xi \in \R.
\end{align}
Next, we compute
\begin{align*}
\mathrm{I}_{1}^{\varepsilon, \rho, \omega, \vn}   = & \ \int_0^\mathfrak{a} \int_{-\mathfrak{s}}^{\mathfrak{s}} \int_{[-r,r]^2} \left|\left\langle h_{\varepsilon, \Xi, \rho/{|\vn_1|}}^1(S_{-\theta} \cdot), \mu_{a,s,t, 1}^\omega\right \rangle\right|^2 a^{-3}\, \,\mathrm{d}t \, \mathrm{d}s \, \mathrm{d}a\\
 = & \ \int_0^\mathfrak{a} \int_{-\mathfrak{s}}^{\mathfrak{s}} \int_{[-r,r]^2} \left|\left\langle h_{\varepsilon, \Xi, \rho/{|\vn_1|}}^1, \mu_{a,s+\theta,t,1}^{\omega_{\theta}}\right \rangle\right|^2 a^{-3}\, \,\mathrm{d}t \, \mathrm{d}s \, \mathrm{d}a\\
 = & \ \int_0^\mathfrak{a} \int_{-\mathfrak{s}+\theta}^{\mathfrak{s}+\theta} \int_{[-r,r]^2} \left|\left\langle h_{\varepsilon, \Xi, \rho/{|\vn_1|}}^1, \mu_{a,s,t,1}^{\omega_{\theta}}\right \rangle\right|^2 a^{-3}\, \,\mathrm{d}t \, \mathrm{d}s \, \mathrm{d}a,
\end{align*}
where $\omega_{\theta}(\iota, s) = \omega(\iota, s - \theta)$. By the same argument, we obtain that
\begin{align*}
\mathrm{I}_{-1}^{\varepsilon, \rho, \omega, \vn}  = & \ \int_0^\mathfrak{a} \int_{-\mathfrak{s}+\frac{1}{\theta}}^{\mathfrak{s}+\frac{1}{\theta}} \int_{[-r,r]^2} \left|\left\langle h_{\varepsilon, \Xi, \rho/{|\vn_2|}}^1, \mu_{a,s,t,-1}^{\omega_{\frac{1}{\theta}}}\right \rangle\right|^2 a^{-3}\, \,\mathrm{d}t \, \mathrm{d}s \, \mathrm{d}a.
\end{align*}
\step{3}{Removing non-aligned parts}
We have that $
h_{\varepsilon, \Xi, \rho/{|\vn_1|}}^1(x) = {\varepsilon^{-1}} |\vn_1|/\rho \cdot g(|\vn_1| x_1 /(\rho\varepsilon))
$ for $g(x) = \Xi'(x)$ if $x \in [-1/2, 1/2]$ and $g(x) = 0$ otherwise. Moreover, $\|g\|_{L^1(\R)} \leq \|\Xi'\|_\infty$ and thus $\|\hat{g}\|_\infty< \|\Xi'\|_\infty$.
The Fourier transform of $h_{\varepsilon, \Xi, \rho/{|\vn_1|}}^1$ is a tempered distribution and is given as 
$$
\mathcal{F}\left(\frac{\partial }{\partial x_1} H_{\varepsilon, \Xi, \rho/{|\vn_1|}}\right) = \hat{g}\left(\frac{\rho \varepsilon}{|\vn_1|} \left(\cdot\right)_1\right) \otimes \delta_{\xi_2 = 0},
$$
where $\delta_{\xi_2 = 0}$ denotes the $\delta$-distribution in $\xi_2$. 

Let $G \in C^\infty_c(\R^2)$ with $\|G\|_{L^1(\R^2)} = 1$. We set $G_m \coloneqq m^2 G(m \cdot)$ so that $(G_m)_{m\in \N}$ is a sequence of mollifiers. Since $\mu_{a,s,t,1}$ has compact support, we have that $\mu_{a,s,t,1} * G_m$ is a Schwartz function and
$$
\|\mu_{a,s,t,1} * G_m - \mu_{a,s,t,1}\|_{L^p(\R^2)} \to 0, \text{ for } m \to \infty,
$$
for all $p \in [1,\infty]$, see e.g. \cite[Lemma 2.18]{adams2003sobolev}. Parseval's identity and the convolution theorem show that
\begin{align*}
\left|\left\langle h_{\varepsilon, \Xi, \rho/{|\vn_1|}}^1, \mu_{a,s,t,1}\right \rangle\right| = & \lim_{m \to \infty} \left|\left\langle h_{\varepsilon, \Xi, \rho/{|\vn_1|}}^1, \overline{\mu_{a,s,t,1} * G_m}\right \rangle\right|\\
= & \lim_{m \to \infty} \left| \int_{\R} \hat{g}\left(\frac{\rho \varepsilon}{|\vn_1|} \xi\right) \overline{\mathcal{F}(\mu_{a,s,t,1} * G_m)(\xi, 0)} \mathrm{d}\xi \right|\\
= & \lim_{ m \to \infty} a^{\frac34}\left|\int_{\R} \hat{g}\left(\frac{\rho \varepsilon}{|\vn_1|} \xi\right) \overline{\widehat{\mu^1}(a \xi )\widehat{\phi^1}(\sqrt{a} s \xi) \phantom{\Large \widehat{\widehat{I}}} \widehat{G_m}(\xi/m, 0)}\, \,\mathrm{d}\xi \right|\\
\leq & \lim_{m \to \infty} a^{\frac34}\int_{\R} \left| \widehat{\mu}_1(a \xi )\widehat{\phi^1}(\sqrt{a} s \xi)\right | \cdot \left|\widehat{G_m}(\xi/m, 0) \right| \, \,\mathrm{d}\xi \cdot \|\Xi'\|_\infty,\\
\leq & \ a^{\frac34}\int_{\R} \left| \widehat{\mu}_1(a \xi )\widehat{\phi^1}(\sqrt{a} s \xi)\right | \, \,\mathrm{d}\xi \cdot \|\Xi'\|_\infty,
\end{align*}
where we have used that $\|\widehat{G_m}\|_{L^\infty(\R^2)} \leq \|{G_m}\|_{L^1(\R^2)}$, for all $m \in \N$. 
By \eqref{eq:VanishingMomentsMu}, we conclude for $s\neq 0$ that 
\begin{align}
a^{\frac34}\int_{\R} \left| \widehat{\mu}_1(a \xi )\widehat{\phi^1}(\sqrt{a} s \xi) \right| \, \,\mathrm{d}\xi &\leq a^{\frac34}\int_{\R} \frac{\min\{ |a \xi|, 1\}^2}{{(1+|a \xi|)^2} (1+|\sqrt{a} s \xi|)^4} \, \,\mathrm{d}\xi \nonumber \\
& \leq a^{\frac14} s^{-1} \int_{\R} \frac{|s^{-1} \sqrt{a} \xi|^2}{(1+|\xi|)^4} \, \,\mathrm{d}\xi \lesssim a^{\frac54}s^{-3},\label{eq:iugfwiug}
\end{align}
since $\int_{\R}|\xi|^2/(1+|\xi|)^4 d\xi < \infty$. Next, we remove terms from $\mathrm{I}_{1}^{\varepsilon, \rho, \omega, \vn} $ and $\mathrm{I}_{-1}^{\varepsilon, \rho, \omega, \vn} $ that 
asymptotically only grow relatively slowly for $\varepsilon \to 0$. Let $\nu_{s}>0$; then we define
\begin{align*}
\mathrm{K}_{1}^{\varepsilon, \rho, \vn, \nu_s} \coloneqq   &  \int_0^\mathfrak{a} \int_{(-\infty, -\nu_s]} \int_{[-r,r]^2}  \left|\left\langle h_{\varepsilon, \Xi, \rho/{|\vn_1|}}^1, \mu_{a,s,t,1}\right \rangle\right|^2a^{-3}\, \,\mathrm{d}t \, \mathrm{d}s \, \mathrm{d}a\\
& \qquad +   \int_0^\mathfrak{a} \int_{[\nu_s, \infty)} \int_{[-r,r]^2} \left|\left\langle h_{\varepsilon, \Xi, \rho/{|\vn_1|}}^1, \mu_{a,s,t,1}\right \rangle\right|^2 a^{-3}\, \,\mathrm{d}t \, \mathrm{d}s \, \mathrm{d}a.
\end{align*}
Invoking \eqref{eq:iugfwiug}, we get that
$$
\mathrm{K}_{1}^{\varepsilon, \rho, \vn, \nu_s}  \leq 2  \int_0^\mathfrak{a} \int_{(-\infty, -\nu_s]} \int_{[-r,r]^2} a^{-\frac12} s^{-6} \, \,\mathrm{d}t \, \mathrm{d}s \, \mathrm{d}a \lesssim \nu_s^{-5}.
$$
Clearly, the same estimate can be derived for
\begin{align*}
\mathrm{K}_{-1}^{\varepsilon, \rho, \vn, \nu_s} \coloneqq   &  \int_0^\mathfrak{a} \int_{(-\infty, -\nu_s]} \int_{[-r,r]^2} \left|\left\langle h_{\varepsilon, \Xi, \rho/{|\vn_2|}}^{-1}, \mu_{a,s,t,-1}\right \rangle\right|^2 a^{-3}\, \,\mathrm{d}t \, \mathrm{d}s \, \mathrm{d}a\\
& \qquad +  \int_0^\mathfrak{a} \int_{[\nu_s, \infty)} \int_{[-r,r]^2}  \left|\left\langle h_{\varepsilon, \Xi, \rho/{|\vn_2|}}^{-1}, \mu_{a,s,t,-1}\right \rangle\right|^2 a^{-3}\, \,\mathrm{d}t \, \mathrm{d}s \, \mathrm{d}a.
\end{align*}
We set from now on $\nu_s \coloneqq \varepsilon^{\frac{1}{11}}$ and observe that $\mathrm{K}_{1}^{\varepsilon, \rho, \vn, \nu_s}, \mathrm{K}_{-1}^{\varepsilon, \rho, \vn, \nu_s} = o(\varepsilon^{-1/2})$.
Hence, we conclude that
\begin{align*}
\mathrm{I}_{1}^{\varepsilon, \rho, \omega, \vn}   =  \int_0^\mathfrak{a} \int_{\max\{-\mathfrak{s}+\theta, -\nu_s\}}^{\min\{\mathfrak{s}+\theta, \nu_s\}} \int_{[-r,r]^2}
\left|\left\langle h_{\varepsilon, \Xi, \rho/{|\vn_1|}}^1, \mu_{a,s,t,1}^{\omega_{\theta}}\right \rangle\right|^2 a^{-3}\, \,\mathrm{d}t \, \mathrm{d}s \, \mathrm{d}a + o\left(\varepsilon^{-\frac{1}{2}}\right).
\end{align*}
Next, we observe that replacing the directional weight by a scalar only produces an asymptotically negligible error. Let $N\in \N$; then, 
\begin{align*}
& \int_0^\mathfrak{a} \int_{\max\{-\mathfrak{s}+\theta, -\nu_s\}}^{\min\{\mathfrak{s}+\theta, \nu_s\}} \int_{[-r,r]^2} \left|\left\langle h_{\varepsilon, \Xi, \rho/{|\vn_1|}}^{1}, \mu_{a,s,t,1}^{\omega_{\theta}}\right \rangle\right|^2 a^{-3}\, \,\mathrm{d}t \, \mathrm{d}s \, \mathrm{d}a\\
&\qquad -  \int_0^\mathfrak{a} \int_{\max\{-\mathfrak{s}+\theta, -\nu_s\}}^{\min\{\mathfrak{s}+\theta, \nu_s\}} \int_{[-r,r]^2}  \omega(1,-\theta)^2 \left|\left\langle h_{\varepsilon, \Xi, \rho/{|\vn_1|}}^{1}, \mu_{a,s,t,1}\right \rangle\right|^2 a^{-3}\, \,\mathrm{d}t \, \mathrm{d}s \, \mathrm{d}a\\
= & \frac{1}{N}\int_0^\mathfrak{a} \int_{\max\{-\mathfrak{s}+\theta, -\nu_s\}}^{\min\{\mathfrak{s}+\theta, \nu_s\}} \int_{[-r,r]\times [-Nr,Nr]} \left|\left\langle h_{\varepsilon, \Xi, \rho/{|\vn_1|}}^{1}, \mu_{a,s,t,1}^{\omega_{\theta}}\right \rangle\right|^2 a^{-3}\, \,\mathrm{d}t \, \mathrm{d}s \, \mathrm{d}a\\
&\qquad - \frac{1}{N} \int_0^\mathfrak{a} \int_{\max\{-\mathfrak{s}+\theta, -\nu_s\}}^{\min\{\mathfrak{s}+\theta, \nu_s\}} \int_{[-r,r]\times [-Nr,Nr]}   \omega(1,-\theta)^2 \left|\left\langle h_{\varepsilon, \Xi, \rho/{|\vn_1|}}^{1}, \mu_{a,s,t,1}\right \rangle\right|^2 a^{-3}\, \,\mathrm{d}t \, \mathrm{d}s \, \mathrm{d}a \eqqcolon \mathrm{II}.
\end{align*}
Now we observe, by the discussion before \eqref{eq:VanishingMomentsMu} that, for all $a \in (0,\mathfrak{a})$ and $s \in (-\mathfrak{s}+\theta, \mathfrak{s}+\theta)$ and $t \in [-r,r]\times [-Nr,Nr]$, we have that $\suppp \mu_{a,s,t,1} \subset [-r-c,r + c] \times [-Nr-c,Nr + c]$ for a constant $c$ depending on $\mathfrak{a}$ and $\mathfrak{s}$ and $\theta$. 
Therefore, we have that, for $D_{N,r} \coloneqq [-r-c,r + c] \times [-Nr-c,Nr + c]$,
\begin{align*}
\mathrm{II} = &\frac{1}{N}\int_0^\mathfrak{a} \int_{\max\{-\mathfrak{s}+\theta, -\nu_s\}}^{\min\{\mathfrak{s}+\theta, \nu_s\}} \int_{[-r,r]\times [-Nr,Nr]} \left|\left\langle \left(h_{\varepsilon, \Xi, \rho/{|\vn_1|}}^{1}\right)_{|D_{N,r}}, \mu_{a,s,t,1}^{\omega_{\theta}}\right \rangle\right|^2 a^{-3}\, \,\mathrm{d}t \, \mathrm{d}s \, \mathrm{d}a\\
&- \frac{1}{N} \int_0^\mathfrak{a} \int_{\max\{-\mathfrak{s}+\theta, -\nu_s\}}^{\min\{\mathfrak{s}+\theta, \nu_s\}} \int_{[-r,r]\times [-Nr,Nr]}   \omega(1,-\theta)^2 \left|\left\langle \left(h_{\varepsilon, \Xi, \rho/{|\vn_1|}}^{1}\right)_{|D_{N,r}}, \mu_{a,s,t,1}\right \rangle\right|^2 a^{-3}\, \,\mathrm{d}t \, \mathrm{d}s \, \mathrm{d}a \\
 = &\frac{1}{N}\int_0^\mathfrak{a} \int_{\max\{-\mathfrak{s}+\theta, -\nu_s\}}^{\min\{\mathfrak{s}+\theta, \nu_s\}} \int_{\R^2} \left|\left\langle \left(h_{\varepsilon, \Xi, \rho/{|\vn_1|}}^{1}\right)_{|D_{N,r}}, \mu_{a,s,t,1}^{\omega_{\theta}}\right \rangle\right|^2 a^{-3}\, \,\mathrm{d}t \, \mathrm{d}s \, \mathrm{d}a\\
&- \frac{1}{N} \int_0^\mathfrak{a} \int_{\max\{-\mathfrak{s}+\theta, -\nu_s\}}^{\min\{\mathfrak{s}+\theta, \nu_s\}} \int_{\R^2}   \omega(1,-\theta)^2 \left|\left\langle \left(h_{\varepsilon, \Xi, \rho/{|\vn_1|}}^{1}\right)_{|D_{N,r}}, \mu_{a,s,t,1}\right \rangle\right|^2 a^{-3}\, \,\mathrm{d}t \, \mathrm{d}s \, \mathrm{d}a.
\end{align*}
By the Bessel property, see Appendix \ref{app:Bessel}, \eqref{eq:BesselInequalityIndepOfsa}, we can deduce using the Lipschitz continuity of $\omega$ and $\varepsilon^{-10/11} = o(\varepsilon^{-1} \log_2(\varepsilon^{-1})^{-1})$, for $\varepsilon \to 0$, that
\begin{align*}
\mathrm{II} = & \frac{1}{N}\int_0^\mathfrak{a} \int_{\max\{-\mathfrak{s}+\theta, -\nu_s\}}^{\min\{\mathfrak{s}+\theta, \nu_s\}} \int_{[-r,r]^2} \left|\left\langle h_{\varepsilon, \Xi, \rho/{|\vn_1|}}^{1}, \mu_{a,s,t,1}^{\omega_{\theta}}\right \rangle\right|^2 a^{-3}\, \,\mathrm{d}t \, \mathrm{d}s \, \mathrm{d}a\\
&\qquad - \frac{1}{N} \int_0^\mathfrak{a} \int_{\max\{-\mathfrak{s}+\theta, -\nu_s\}}^{\min\{\mathfrak{s}+\theta, \nu_s\}} \int_{[-r,r]^2}  \omega(1,-\theta)^2 \left|\left\langle h_{\varepsilon, \Xi, \rho/{|\vn_1|}}^{1}, \mu_{a,s,t,1}\right \rangle\right|^2 a^{-3}\, \,\mathrm{d}t \, \mathrm{d}s \, \mathrm{d}a\\
\leq & \sup_{s \in [-\nu_s, \nu_s]} |\omega(1,-\theta)^2 - \omega(1, s-\theta)^2| \cdot \frac{1}{N} \left\|h_{\varepsilon, \Xi, \rho/{|\vn_1|}}^{1}\right\|_{L^2(D_{N,r})}^2 \\
\lesssim &\ \nu_s \left\| h_{\varepsilon, \Xi, \rho/{|\vn_1|}}^{1}(\cdot, 0) \right\|_{L^2((-r-c,r+c))}^2 \frac{2(Nr+c)}{N}\\
\lesssim & \ 
\nu_s \varepsilon^{-1} r  = r \cdot o\left(\frac{1}{\varepsilon} \log_2\left(\frac{1}{\varepsilon}\right)^{-1}\right), \text{ for } \varepsilon \to 0,  
\end{align*}
where the last estimate holds since $N$ was arbitrary and can be chosen larger than $c/r$. The implied constant in the last estimate depends quadratically on $\|\Xi'\|_\infty/\rho$. Consequentially, we have that
\begin{align}\label{eq:AHandyVersionOfI1}
\mathrm{I}_{1}^{\varepsilon, \rho, \omega, \vn}   = \omega(1, -\theta)^2 \int_0^\mathfrak{a}\int_{\max\{-\mathfrak{s}+\theta, -\nu_s\}}^{\min\{\mathfrak{s}+\theta, \nu_s\}} \int_{[-r,r]^2} \left|\left\langle h_{\varepsilon, \Xi, \rho/{|\vn_1|}}^{1}, \mu_{a,s,t,1}\right \rangle\right|^2  a^{-3}\, \,\mathrm{d}t \, \mathrm{d}s \, \mathrm{d}a\\
+ \ r \cdot o\left(\frac{1}{\varepsilon} \log_2\left(\frac{1}{\varepsilon}\right)^{-1}\right) + o\left(\varepsilon^{-\frac{1}{2}}\right),\nonumber
\end{align}
for $\varepsilon \to 0$. Similarly, 
\begin{align*}
\mathrm{I}_{-1}^{\varepsilon, \rho, \omega, \vn}   = \omega(-1, -1/\theta)^2  \int_0^\mathfrak{a}\int_{\max\{-\mathfrak{s}+1/\theta, -\nu_s\}}^{\min\{\mathfrak{s}+1/\theta, \nu_s\}} \int_{[-r,r]^2} \left|\left\langle h_{\varepsilon, \Xi, \rho/{|\vn_2|}}^{-1}, \mu_{a,s,t,-1}\right \rangle\right|^2  a^{-3}\, \,\mathrm{d}t \, \mathrm{d}s \, \mathrm{d}a\\  +  \
r \cdot o\left(\frac{1}{\varepsilon} \log_2\left(\frac{1}{\varepsilon}\right)^{-1}\right)  + o\left(\varepsilon^{-\frac{1}{2}}\right),
\end{align*}
for $\varepsilon \to 0$.

\step{4}{Rewriting as a non-cone-adapted shearlet transform}

We want to show that $\mathrm{I}_{1}^{\varepsilon, \rho, \omega, \vn} $ is asymptotically a rescaled version of the Sobolev seminorm of $H_{\varepsilon, \Xi, \rho}^{(1,0)}$.
Towards this goal, we plan to invoke \eqref{eq:TheAssumption} and use the resulting equivalence of the homogeneous shearlet transform and the $L^2$-norm.
The main obstacle here is to deal with the fact that the integral in \eqref{eq:AHandyVersionOfI1} is taken only over a finite domain.
Moreover, simply extending $ h_{\varepsilon, \Xi, \rho/{|\vn_1|}}^{1}$ to $\R^2$ does not yield an $L^2$ function. To overcome these issues, we use a blow-up technique, that was already applied in Proposition \ref{prop:equivalence}.

We define
\begin{align} \label{eq:DefinitionOfBeps}
\mathrm{B}_\varepsilon \coloneqq   \int_0^\mathfrak{a}\int_{\max\{-\mathfrak{s}+\theta, -\nu_s\}}^{\min\{\mathfrak{s}+\theta, \nu_s\}} \int_{[-r,r]^2}\left|\left\langle h_{\varepsilon, \Xi, \rho/{|\vn_1|}}^{1}, \mu_{a,s,t,1}\right \rangle\right|^2 a^{-3}\, \,\mathrm{d}t \, \mathrm{d}s \, \mathrm{d}a.
\end{align}
Since $H_{\varepsilon, \Xi, \rho/{|\vn_1|}}$ is constant in the $x_2$-direction we conclude that, for all $N\in \N$,
\begin{align*}
\mathrm{B}_\varepsilon & = \frac{1}{N}  \int_0^\mathfrak{a}\int_{\max\{-\mathfrak{s}+\theta, -\nu_s\}}^{\min\{\mathfrak{s}+\theta, \nu_s\}} \int_{[-r,r]\times [-N r,N r]} \left|\left\langle h_{\varepsilon, \Xi, \rho/{|\vn_1|}}^{1}, \mu_{a,s,t,1}\right \rangle\right|^2  a^{-3}\, \mathrm{d}a \, \mathrm{d}s \, \mathrm{d}t.
\end{align*}
We replace $h_{\varepsilon, \Xi, \rho/{|\vn_1|}}^{1}$ by a suitable $L^2(\R^2)$ function by the following construction: 
Let $\gamma \in C^\infty(\R)$ with $\suppp \gamma \in [-1/2,1/2]$, $\gamma(x)>0$ for $x \in (-1/2,1/2)$, and $\int_{-1/2}^{1/2} \gamma(x)\,\mathrm{d}x = 1$. We set, for $(x_1,x_2) \in \R^2$,
\begin{align}\label{eq:DefinitionOfFN}
f_N(x_1,x_2) \coloneqq \frac{|\vn_1|}{\varepsilon} g\left(\frac{|\vn_1|}{\rho \varepsilon}x_1\right) \cdot \left(\gamma * \chi_{[-Nr- \mathfrak{a}, Nr+\mathfrak{a}]}\right) (x_2).
\end{align}
It is not hard to see that
\begin{align*}
\mathrm{B}_\varepsilon = & \ \frac{1}{N} \int_0^\mathfrak{a}\int_{\max\{-\mathfrak{s}+\theta, -\nu_s\}}^{\min\{\mathfrak{s}+\theta, \nu_s\}} \int_{[-r,r]\times [-Nr,Nr]} \left|\left\langle f_N, \mu_{a,s,t,1}\right \rangle\right|^2  a^{-3}\, \,\mathrm{d}t \, \mathrm{d}s \, \mathrm{d}a.
\end{align*}
Next, we need to extend the domain of integration to $\R^2$. We start by extending to
$[-r,r] \times \R$. We have that
\begin{align}
\mathrm{B}_\varepsilon = & \ \frac{1}{N} \int_0^\mathfrak{a}\int_{\max\{-\mathfrak{s}+\theta, -\nu_s\}}^{\min\{\mathfrak{s}+\theta, \nu_s\}} \int_{\R^2} \left|\left\langle f_N, \mu_{a,s,t,1}\right \rangle\right|^2  a^{-3}\, \,\mathrm{d}t \, \mathrm{d}s \, \mathrm{d}a \nonumber\\
& - \ \frac{1}{N} \int_0^\mathfrak{a}\int_{\max\{-\mathfrak{s}+\theta, -\nu_s\}}^{\min\{\mathfrak{s}+\theta, \nu_s\}} \int_{\R^2 \setminus ((-r,r) \times (-N r,N r))} \left|\left\langle f_N, \mu_{a,s,t,1}\right \rangle\right|^2  a^{-3}\, \,\mathrm{d}t \, \mathrm{d}s \, \mathrm{d}a. \label{eq:SplittingOfB}
\end{align}
We have that
\begin{align}\label{eq:SplittingOfSets}
\R^2 \setminus ((-r,r) \times N(-r,r)) = \big((-r,r) \times (\R \setminus (-N r,N r))\big) \cup \big((\R \setminus [-r,r]) \times \R \big).
\end{align}
As $B_\varepsilon = 0$ if $|\theta| \geq \mathfrak{s} + 1$, we assume that $|\theta| \leq \mathfrak{s} + 1$ and hence
$$
(\mathfrak{s} + 1)^2 \geq \frac{\vn_2^2}{\vn_1^2} = \frac{1}{\vn_1^2} - 1
$$
and thus $|\vn_1| \geq 1/(\sqrt{1 + (\mathfrak{s} + 1)^2}) \geq 1/(\sqrt{(\mathfrak{s} + 2)^2})  =  1/(\mathfrak{s}+2)$. We choose $R>0$
such that $\suppp \mu_{a,s,t,1} \subset B_{a^{1/2} R}(t)$ for all $(a,s,t) \in [0,\mathfrak{a}] \times [- 1, 1] \times \R^2$.
Hence we have that $\left\langle f_N, \mu_{a,s,t,1}\right \rangle = 0$
if $|t_2| > \rho\varepsilon/|n_1| + R \max\{a^{1/2}, a\}$.
Thus, we set $\mathfrak{a}_{r,\varepsilon} \coloneqq (r-\rho\varepsilon (\mathfrak{s}+2))^2/R^2$ and compute, using $\nu_s \leq 1$,
\begin{align*}
 &\frac{1}{N} \left|\int_0^\mathfrak{a}\int_{\max\{-\mathfrak{s}+\theta, -\nu_s\}}^{\min\{\mathfrak{s}+\theta, \nu_s\}} \int_{(\R \setminus [-r,r]) \times \R} \left|\left\langle f_N, \mu_{a,s,t,1}\right \rangle\right|^2  a^{-3}\, \,\mathrm{d}t \, \mathrm{d}s \, \mathrm{d}a\right| \\
 &  =  \frac{1}{N} \left|\int_{\mathfrak{a}_{r,\varepsilon}}^\mathfrak{a}\int_{\max\{-\mathfrak{s}+\theta, -\nu_s\}}^{\min\{\mathfrak{s}+\theta, \nu_s\}} \int_{(\R \setminus [-r,r]) \times \R} \left|\left\langle f_N, \mu_{a,s,t,1}\right \rangle\right|^2  a^{-3}\, \,\mathrm{d}t \, \mathrm{d}s \, \mathrm{d}a\right|\\
 &\leq \frac{1}{N} \left|\int_{\mathfrak{a}_{r,\varepsilon}}^\mathfrak{a}\int_{\max\{-\mathfrak{s}+\theta, -\nu_s\}}^{\min\{\mathfrak{s}+\theta, \nu_s\}} \int_{\R^2} \left|\left\langle f_N, \mu_{a,s,t,1}\right \rangle\right|^2  a^{-3}\, \,\mathrm{d}t \, \mathrm{d}s \, \mathrm{d}a\right|.
\end{align*}
 We have by Plancherel's identity and Parseval's identity that
\begin{align*}
    & \ \frac{1}{N} \left|\int_{\mathfrak{a}_{r,\varepsilon}}^\mathfrak{a}\int_{\max\{-\mathfrak{s}+\theta, -\nu_s\}}^{\min\{\mathfrak{s}+\theta, \nu_s\}} \int_{\R^2} \left|\left\langle f_N, \mu_{a,s,t,1}\right \rangle\right|^2  a^{-3}\, \,\mathrm{d}t \, \mathrm{d}s \, \mathrm{d}a\right|\\
    & \quad =  \ \frac{1}{N} \left|\int_{\mathfrak{a}_{r,\varepsilon}}^\mathfrak{a}\int_{\max\{-\mathfrak{s}+\theta, -\nu_s\}}^{\min\{\mathfrak{s}+\theta, \nu_s\}} \int_{\R^2}\left|\left\langle \widehat{f_N}, \widehat{\mu}_{a,s,t,1}\right \rangle\right|^2  a^{-3}\, \,\mathrm{d}t \, \mathrm{d}s \, \mathrm{d}a\right|\\
     & \quad = \frac{1}{N} \int_{\mathfrak{a}_{r,\varepsilon}}^\mathfrak{a}\int_{\max\{-\mathfrak{s}+\theta, -\nu_s\}}^{\min\{\mathfrak{s}+\theta, \nu_s\}} \int_{\R^2} \left|\int_{\R^2} \widehat{f_N}(\xi) \widehat{\mu}_{a,s,0,1}(\xi) \mathrm{e}^{-2\pi i \langle \xi, t\rangle} \,\mathrm{d}\xi \right|^2  a^{-3}\, \,\mathrm{d}t \, \mathrm{d}s \, \mathrm{d}a\\
    & \quad =  \ \frac{1}{N} \int_{\mathfrak{a}_{r,\varepsilon}}^\mathfrak{a}\int_{\max\{-\mathfrak{s}+\theta, -\nu_s\}}^{\min\{\mathfrak{s}+\theta, \nu_s\}}\int_{\R^2}	\left|\widehat{f_N}(\xi)\right|^2| \widehat{\mu}_{a,s,0,1}(\xi)|^2 a^{-3} \, \,\mathrm{d}\xi \, \mathrm{d}s \, \mathrm{d}a \\
     & \quad =  \ \frac{1}{N}\int_{\R^2} \left|\widehat{f_N}(\xi)\right|^2   \int_{\mathfrak{a}_{r,\varepsilon}}^\mathfrak{a}\int_{\max\{-\mathfrak{s}+\theta, -\nu_s\}}^{\min\{\mathfrak{s}+\theta, \nu_s\}} | \widehat{\mu}_{a,s,0,1}(\xi)|^2 a^{-3} \mathrm{d}s \, \mathrm{d}a \, \,\mathrm{d}\xi.
\end{align*}
Taking the Fourier transform of \eqref{eq:DefinitionOfFN} shows that, for all $\xi \in \R^2$,
\begin{align}\label{eq:FourierTransform}
\widehat{f_N}(\xi) = N \mathrm{sinc}((N+\mathfrak{a}) \xi_2)  \widehat{\gamma}(\xi_2) \hat{g}\left(\frac{\rho \varepsilon}{|\vn_1|} \xi_1\right).
\end{align}
Moreover, we show in Appendix \ref{app:Computations} that, for all $\xi \in \R^2$,
\begin{align}
\int_{\mathfrak{a}_{r,\varepsilon}}^\mathfrak{a}\int_{\max\{-\mathfrak{s}+\theta, -\nu_s\}}^{\min\{\mathfrak{s}+\theta, \nu_s\}} | \widehat{\mu}_{a,s,0,1}(\xi)|^2 a^{-3}\mathrm{d}s \, \mathrm{d}a \lesssim \frac{1}{\mathfrak{a}_{r,\varepsilon}^2} \frac{1}{(1+|\xi_1|)^2}. \label{eq:FirstPartofBEpsilon}
\end{align}
We conclude that
\begin{align*}
    & \ \frac{1}{N} \left|\int_{\mathfrak{a}_{r,\varepsilon}}^\mathfrak{a}\int_{\max\{-\mathfrak{s}+\theta, -\nu_s\}}^{\min\{\mathfrak{s}+\theta, \nu_s\}} \int_{\R^2} \left|\left\langle f_N, \mu_{a,s,t,1}\right \rangle\right|^2  a^{-3}\, \,\mathrm{d}t \, \mathrm{d}s \, \mathrm{d}a\right| =  \mathcal{O}\left(\frac{1}{\mathfrak{a}_{r,\varepsilon}^2}\right), \text{ for } \varepsilon \to 0,
\end{align*}
where the implicit constant depends quadratically on $\|\Xi'\|_\infty$.
We proceed by estimating the integral over the second set given in \eqref{eq:SplittingOfSets}:
\begin{align*}
 &\frac{1}{N} \left|\int_0^\mathfrak{a}\int_{\max\{-\mathfrak{s}+\theta, -\nu_s\}}^{\min\{\mathfrak{s}+\theta, \nu_s\}} \int_{[-r,r] \times (\R \setminus [-N r, N r])} \left|\left\langle f_N, \mu_{a,s,t,1}\right \rangle\right|^2  a^{-3}\, \,\mathrm{d}t \, \mathrm{d}s \, \mathrm{d}a\right|  \eqqcolon \mathrm{III}.
\end{align*}
Since $\suppp \mu_{a,s,t, \iota} \subset B_{R \mathfrak{a}}(t)$ we conclude that
\begin{align}
 \mathrm{III} = &\ \frac{1}{N} \left|\int_0^\mathfrak{a}\int_{\max\{-\mathfrak{s}+\theta, -\nu_s\}}^{\min\{\mathfrak{s}+\theta, \nu_s\}} \int_{[-r,r] \times (\R \setminus [-N r,N r])} \left|\left\langle f_N \chi_{\R \times [-Nr + R\mathfrak{a}, Nr - R\mathfrak{a}]^c}, \mu_{a,s,t,1}\right \rangle\right|^2  a^{-3}\, \,\mathrm{d}t \, \mathrm{d}s \, \mathrm{d}a\right|\nonumber \\
 \leq &\ \frac{1}{N} \left|\int_0^\mathfrak{a}\int_{\max\{-\mathfrak{s}+\theta, -\nu_s\}}^{\min\{\mathfrak{s}+\theta, \nu_s\}} \int_{\R^2} \left|\left\langle f_N \chi_{\R \times [-Nr + R\mathfrak{a}, Nr - R\mathfrak{a}]^c}, \mu_{a,s,t,1}\right \rangle\right|^2  a^{-3}\, \,\mathrm{d}t \, \mathrm{d}s \, \mathrm{d}a\right| \nonumber \\
 \lesssim & \ \frac{1}{N} \|f_N \chi_{\R \times [-Nr + R\mathfrak{a}, Nr - R\mathfrak{a}]^c}\|_{L^2(\R^2)}^2, \label{eq:SecondPartofBEpsilon}
\end{align}
where the last step is due to the Bessel inequality of the shearlet transform, see Appendix \ref{app:Bessel}. 
Since $[-Nr + R\mathfrak{a}, Nr - R\mathfrak{a}]^c$ and $[-Nr - R\mathfrak{a} - 1, Nr + R\mathfrak{a} +1]$ intersect on a set the size of which is independent of $N$, 
we get by \eqref{eq:DefinitionOfFN} that $\|f_N \chi_{\R \times [-Nr + R\mathfrak{a}, Nr - R\mathfrak{a}]^c}\|_{L^2(\R^2)}^2 = \mathcal{O}(1/\varepsilon)$, for $\varepsilon \to 0$, with the implicit constant independent of $N$.

By equations \eqref{eq:SplittingOfB}, \eqref{eq:SplittingOfSets}, \eqref{eq:FirstPartofBEpsilon}, and \eqref{eq:SecondPartofBEpsilon}, we conclude that
\begin{align*}
 B_\varepsilon = \frac{1}{N} \int_0^\mathfrak{a}\int_{\max\{-\mathfrak{s}+\theta, -\nu_s\}}^{\min\{\mathfrak{s}+\theta, \nu_s\}} \int_{\R^2} \left|\left\langle f_N, \mu_{a,s,t,1}\right \rangle\right|^2  a^{-3}\, \,\mathrm{d}t \, \mathrm{d}s \, \mathrm{d}a + \mathcal{O}\left(\frac{1}{\mathfrak{a}_{r, \varepsilon}^2}\right) + \frac{1}{N} \mathcal{O} (1/\varepsilon).
\end{align*}

We have by Plancherel's identity and by Parseval's identity that
\begin{align*}
    & \ \frac{1}{N} \int_0^\mathfrak{a}\int_{\max\{-\mathfrak{s}+\theta, -\nu_s\}}^{\min\{\mathfrak{s}+\theta, \nu_s\}} \int_{\R^2} \left|\left\langle f_N, \mu_{a,s,t,1}\right \rangle\right|^2  a^{-3}\, \,\mathrm{d}t \, \mathrm{d}s \, \mathrm{d}a\\
    & =  \ \frac{1}{N} \int_0^\mathfrak{a}\int_{\max\{-\mathfrak{s}+\theta, -\nu_s\}}^{\min\{\mathfrak{s}+\theta, \nu_s\}} \int_{\R^2}\left|\left\langle \widehat{f_N}, \widehat{\mu}_{a,s,t,1}\right \rangle\right|^2  a^{-3}\, \,\mathrm{d}t \, \mathrm{d}s \, \mathrm{d}a\\
     & =  \ \frac{1}{N} \int_0^\mathfrak{a}\int_{\max\{-\mathfrak{s}+\theta, -\nu_s\}}^{\min\{\mathfrak{s}+\theta, \nu_s\}} \int_{\R^2} \left|\int_{\R^2} \widehat{f_N}(\xi) \widehat{\mu}_{a,s,0,1}(\xi) \mathrm{e}^{-2\pi i \langle \xi, t\rangle} \,\mathrm{d}\xi \right|^2  a^{-3}\, \,\mathrm{d}t \, \mathrm{d}s \, \mathrm{d}a\\
     & =  \ \frac{1}{N} \int_0^\mathfrak{a}\int_{\max\{-\mathfrak{s}+\theta, -\nu_s\}}^{\min\{\mathfrak{s}+\theta, \nu_s\}}\int_{\R^2}	 \left|\widehat{f_N}(\xi)\right|^2| \widehat{\mu}_{a,s,0,1}(\xi)|^2 a^{-3}\,\mathrm{d}\xi  \,\mathrm{d}s \,\mathrm{d}a\\
     & =  \ \frac{1}{N}\int_{\R^2} \left|\widehat{f_N}(\xi)\right|^2   \int_0^\mathfrak{a}\int_{\max\{-\mathfrak{s}+\theta, -\nu_s\}}^{\min\{\mathfrak{s}+\theta, \nu_s\}} | \widehat{\mu}_{a,s,0,1}(\xi)|^2 a^{-3}\mathrm{d}s \, \mathrm{d}a \,\mathrm{d}\xi.
\end{align*}
Two simple but lengthy computations, that we postpone to Appendix \ref{app:Computations}, show that
\begin{align} \label{eq:estimateLowFrequencyPart}
\int_{\mathfrak{a}}^{\infty}\int_{\R} | \widehat{\mu}_{a,s,0,1}(\xi)|^2 a^{-3} \mathrm{d}s \, \mathrm{d}a \lesssim \left(1+|\xi_1|\right)^{-2} \text{ for all } \xi \in \R^2,
\end{align}
and if $|\nu_s\xi_1|/2 > |\xi_2|$, then
\begin{align}\label{eq:estimateHighFrequencyPart}
 \int_0^\mathfrak{a}\int_{-\infty}^{-\nu_s} | \widehat{\mu}_{a,s,0,1}(\xi)|^2 a^{-3}\mathrm{d}s \, \mathrm{d}a \lesssim
\frac{|\xi_1|^2}{(1+|\nu_s \xi_1|)^4}\, \text{ and } \,  \int_0^\mathfrak{a}\int_{\nu_s}^{\infty} | \widehat{\mu}_{a,s,0,1}(\xi)|^2 a^{-3}\mathrm{d}s \, \mathrm{d}a \lesssim
\frac{|\xi_1|^2}{(1+|\nu_s \xi_1|)^4}.
\end{align}
If $|\nu_s\xi_1|/2 \leq |\xi_2|$, then since $|\widehat{\gamma}(\xi_2)| \lesssim (1+|\xi_2|)^4$, we can estimate
\begin{align}\label{eq:GammaDecay}
|\widehat{\gamma}(\xi_2)| \leq (1+|\xi_2|)^2 \left(1+ \nu_s \frac{|\xi_1|}{2}\right)^2.
\end{align}
The form of the Fourier transform in \eqref{eq:FourierTransform} and the estimates \eqref{eq:estimateHighFrequencyPart}, \eqref{eq:estimateLowFrequencyPart} and \eqref{eq:GammaDecay}, imply that
$$
 \frac{1}{N} \int_0^\mathfrak{a}\int_{\max\{-\mathfrak{s}+\theta, -\nu_s\}}^{\min\{\mathfrak{s}+\theta, \nu_s\}} \int_{\R^2} \left|\left\langle f_N, \mu_{a,s,t,1}\right \rangle\right|^2  a^{-3}\, \mathrm{d}a \, \mathrm{d}s \, \mathrm{d}t = \frac{1}{N}\int_{\R^+}\int_{\R}\int_{\R^2} \left|\left\langle f_N, \mu_{a,s,t,1}\right \rangle\right|^2  a^{-3}\, \mathrm{d}a \, \mathrm{d}s \, \mathrm{d}t +
 \mathcal{O}(\nu_s^{-1}),
$$
and thus
$$
\mathrm{B}_\varepsilon =\frac{1}{N}\int_{\R^+} \int_{\R}\int_{\R^2} \left|\left\langle f_N, \mu_{a,s,t,1}\right \rangle\right|^2  a^{-3}\, \mathrm{d}a \, \mathrm{d}s \, \mathrm{d}t + \frac{1}{N}O\left(\frac{1}{\varepsilon}\right) + \mathcal{O}\left(\frac{1}{(r-\rho\varepsilon (\mathfrak{s}+2))^4}\right) + \mathcal{O}\left(\varepsilon^{-\frac{1}{11}}\right).
$$
\step{5}{Use of the Parseval property of a homogeneous system}
By \eqref{eq:TheAssumption} we obtain that, for every $\varepsilon>0$,
$$
\frac{1}{N} \|f_N\|_{L^2(\R^2)}^2 \longrightarrow \mathrm{B}_\varepsilon + \mathcal{O}\left(\frac{1}{(r-\rho\varepsilon (\mathfrak{s}+2))^4}\right) + \mathcal{O}\left(\varepsilon^{-\frac{1}{11}}\right),\text{ for } N \to \infty.
$$
Moreover, by \eqref{eq:DefinitionOfFN}, we also observe that, if $\rho \varepsilon \leq r$, then 
$$
\frac{1}{N} \|f_N\|_{L^2(\R^2)}^2 \longrightarrow  \left\|h_{\varepsilon, \Xi, \rho/{|\vn_1|}}^{1}\right\|_{L^2((-r,r)^2)}^2, \text{ for } N \to \infty.
$$
Hence, we conclude that
\begin{align*}
\mathrm{B}_\varepsilon = & \ \left\|h_{\varepsilon, \Xi, \rho/{|\vn_1|}}^{1}\right\|_{L^2((-r,r)^2)}^2 + \mathcal{O}\left(\frac{1}{(r-\rho\varepsilon (\mathfrak{s}+2))^4}\right) + \mathcal{O}\left(\varepsilon^{-\frac{1}{11}}\right)\\
= &\ \left|H_{\varepsilon, \Xi, \rho/{|\vn_1|}}^{(1,0)}\right|_{H^1((-r,r)^2)}^2 + \mathcal{O}\left(\frac{1}{(r-\rho\varepsilon (\mathfrak{s}+2))^4}\right)  + \mathcal{O}\left(\varepsilon^{-\frac{1}{11}}\right) \text{ for } \epsilon \to 0.
\end{align*}
Since $\vec{\eta} \in \mathbb{S}^1$, we have that 
$$
\left(1+ \left|\min\left\{ \left|\frac{\vn_2}{\vn_1}\right|, \left|\frac{\vn_1}{\vn_2}\right|\right\} \right|^2\right)^{-1/2} = \max\{|\vn_1|,|\vn_2| \}.
$$
Using \eqref{eq:rotation} and \eqref{eq:rescaling} we conclude that
\begin{align*}
\mathrm{B}_\varepsilon &=  \ |\vn_1| \left(1+ \left|\min\left\{ \left|\frac{\vn_2}{\vn_1}\right|, \left|\frac{\vn_1}{\vn_2}\right|\right\} \right|^2\right)^{-1/2} \left|H_{\varepsilon, \Xi, \rho}^{\vn}\right|_{H^1((-r,r)^2)}^2 + \mathcal{O}\left(\frac{1}{(r-\rho\varepsilon (\mathfrak{s}+2))^4}\right)   + \mathcal{O}\left(\varepsilon^{-\frac{1}{11}}\right)\\
 & =  |\vn_1| \max\{|\vn_1|,|\vn_2| \} \left|H_{\varepsilon, \Xi, \rho}^{\vn}\right|_{H^1((-r,r)^2)}^2 + \mathcal{O}\left(\frac{1}{(r-\rho\varepsilon (\mathfrak{s}+2))^4}\right)  + \mathcal{O}\left(\varepsilon^{-\frac{1}{11}}\right), \text{ for } \varepsilon \to 0.
\end{align*}
Since, by \eqref{eq:DefinitionOfBeps} and \eqref{eq:AHandyVersionOfI1}, we have that $\mathrm{I}_{1}^{\varepsilon, \rho, \omega, \vn}  = \omega(1, \frac{\vn_2}{\vn_1})\mathrm{B}_\varepsilon$, we get that
\begin{align*}
\mathrm{I}_{1}^{\varepsilon, \rho, \omega, \vn}  = &\max\left\{|\vn_1|, |\vn_2|\right\}\left(|\vn_1| \omega\left(1, \frac{\vn_2}{\vn_1}\right)^2 \right)  \left|H_{\varepsilon, \Xi, \rho}^{\vn}\right|_{H^1((-r,r)^2)}^2\\
& \qquad+ r\, o\left(\frac{1}{\varepsilon}\log_2\left(\frac{1}{\varepsilon}\right)^{-1}\right) + \mathcal{O}\left(\frac{1}{(r-\rho\varepsilon (\mathfrak{s}+2))^4}\right)  + o\left(\frac{1}{\sqrt{\varepsilon}}\right), \text{ for } \varepsilon \to 0.
\end{align*}
Performing the same computations for $ \mathrm{I}_{-1}^{\varepsilon, \rho, \omega, \vn} $ yields that
\begin{align*}
\mathrm{I}_{1}^{\varepsilon, \rho, \omega, \vn}  + \mathrm{I}_{-1}^{\varepsilon, \rho, \omega, \vn}  = \max\left\{|\vn_1|, |\vn_2|\right\}\left(|\vn_1| \omega\left(1, \frac{\vn_2}{\vn_1}\right)^2 + |\vn_2|\omega\left(-1, \frac{\vn_1}{\vn_2}\right)^2 \right)  \left|H_{\varepsilon, \Xi, \rho}^{\vn}\right|_{H^1((-r,r)^2)}^2\\ + r\, o\left(\frac{1}{\varepsilon}\log_2\left(\frac{1}{\varepsilon}\right)^{-1}\right) + \mathcal{O}\left(\frac{1}{(r-\rho\varepsilon (\mathfrak{s}+2))^4}\right)  + o\left(\frac{1}{\sqrt{\varepsilon}}\right), \text{ for } \varepsilon \to 0.
\end{align*}
\end{proof}

The factor of $|\vn_1|^2$ or $|\vn_2|^2$ in \eqref{eq:PropHeavisideEquation} counteracts a certain stretching effect of the shearing operation.
Indeed, as we have seen in the proof above, the shearlet-based seminorm is, apart from the weight, invariant to shearing.
On the other hand, the $H^1$ seminorm clearly is not, since shearing stretches in one direction stronger than in the other.

\subsection{Construction of a recovery sequence for polygons}

We proceed by analysing the asymptotic behaviour of the shearlet-based Ginzburg--Landau energy for functions that admit a smooth phase-transition over a polygonal curve.
We start by giving a definition of a polygon, then we introduce a class of functions that admit a smooth phase-transition across the boundary of a polygon. After that we describe the asymptotic behaviour of the shearlet-based Besov seminorm of these functions. At the end of the section, we will describe how this yields a sequence of functions with a smooth phase-transition across a polygonal curve such that the associated shearlet-based Ginzburg--Landau energies evaluated on these functions converge to an anisotropic perimeter functional.

\begin{definition}
For $N \in \N$, a closed set $P \subset (0,1)^2$ such that $\partial P$ is a piecewise affine,
non-self-intersecting curve with $N$ pieces is called \emph{$N$-gon}.
The $N$-points where the boundary curve is not affine are called vertices.
We denote them by $x_1^P, \dots, x_N^P$ and we will always assume that they are ordered so that the line between $x_i^P$ and $x_{i+1}^P$ is in $\partial P$ for all $i = 1, \dots, N$, where $x_{N+1}^P \coloneqq x_1^P$.
The identification $x_{N+1}^P = x_1^P$ will also be used in the sequel without additional comments.

The length of the boundary curve of $P$ is then
$$
\ell(P) \coloneqq \sum_{i = 1, \dots, N} |x_i^P - x_{i+1}^P|.
$$

If there is no need to explicitly specify the number $N$, then we shall simply call such a set a \emph{polygon}.
\end{definition}
To make precise what we mean by functions with transitions along polygonal curves, we first need to introduce a more general notion of a transition profile.
A continuously differentiable function $\Xi: \mathbb{S}^1 \times [-1/2,1/2] \to [-1/2,1/2]$, such that, for each $\eta \in \mathbb{S}^1$, $\Xi(\eta, \cdot)$ is a transition profile, is called \emph{directional transition profile}. Additionally, we call a continuously differentiable function $W: \mathbb{S}^1 \to \R^+$ a \emph{directional transition width}.

For a polygon $P \subset (0,1)^2$ with vertices $x_1, \dots, x_N$, there exists a normal map
$$
\partial P \ni x \mapsto \vn_x \in \mathbb{S}^1,
$$
which is well-defined everywhere except at $x_1, \dots, x_N$. We replace this normal map by an auxiliary map which is well-defined everywhere and coincides with $\vn_x$ at every point, except in a neighbourhood of $x_1, \dots, x_N$.

For $\varepsilon>0$ and a directional transition width $W$, we pick a function $\vn_x^{\varepsilon,W}$ such that
$$
\partial P \ni x \mapsto \vn_x^{\varepsilon,W} \in \mathbb{S}^1,
$$
is smooth with derivative bounded by $\pi/\varepsilon$, $\vn_x^{\varepsilon, W} =  \vn_x$ whenever $|x - x_i|> \varepsilon \max\{1,\|W\|_{\infty}\}$, for all $i = 1, \dots, N$.

We then consider the projection operator
$$
\pi:(0,1)^2 \to \partial P:
x \mapsto \argmin_{y \in \partial P} W\left(\frac{x-y}{|x-y|}\right) |x-y|,
$$
which is well-defined almost everywhere. The reason we focus on polygons with directional transition profiles is that precisely these functions form recovery sequences for the classical anisotropic Ginzburg--Landau energy as analysed in \cite[Proposition 4.10]{braides1998approximation}.

For a polygon $P$, a directional transition profile $\Xi$,
and a directional transition width $W$, we now define
\begin{align*}
    P_{\varepsilon, \Xi, W}\coloneqq \left \{ \begin{array}{ll}
        \frac{1}{2} - \Xi\left(\vn_{\pi(x)}^{\varepsilon, W}, \ \frac{\pi(x)-x}{\varepsilon} / W\left(\vn_{\pi(x)}^{\varepsilon, W}\right)\right) & \text{ if } |x - \pi(x)|< \frac{1}{2}\varepsilon W\left(\vn_{\pi(x)}^{\varepsilon, W}\right), x \in P,\\
        \frac{1}{2} - \Xi\left(\vn_{\pi(x)}^{\varepsilon, W} , \ \frac{x-\pi(x)}{\varepsilon} / W\left(\vn_{\pi(x)}^{\varepsilon, W}\right)\right) & \text{ if } |x - \pi(x)|< \frac{1}{2}\varepsilon W\left(\vn_{\pi(x)}^{\varepsilon, W}\right), x \not\in P,\\
        1  & \text{ if } |x - \pi(x)|> \varepsilon W\left(\vn_{\pi(x)}^{\varepsilon, W}\right), x \in P,\\
        0 & \text{ if } |x - \pi(x)|> \varepsilon W\left(\vn_{\pi(x)}^{\varepsilon, W}\right), x \not\in P.
    \end{array}\right.
\end{align*}
We are now able to analyse the asymptotic behaviour of  $|P_{\varepsilon, \Xi, W}|_{\Bp}$ for $\varepsilon \to 0$ and fixed $\Xi, W$. Moreover, we will see, that the estimates even hold independently of $\Xi$ and $W$ if $\|\Xi'\|_\infty$ and $\|W\|_\infty$ satisfy certain growth conditions.

\begin{proposition}\label{prop:PolygonConvergence}
Let $\psi\in L^2(\R^2)$ satisfy the assumptions of Proposition \ref{prop:HeavisideConvergence}.
Let $\aninorm$ be a norm on $\R^2$ and let $\omega$ be the directional weight associated with $\aninorm$.
Assume that $P$ is a polygon, $\Xi$ is a directional transition profile, and $W$ is a directional weight.
Then, we have that
\begin{align}\label{assertionOfPolygonConvergenceProp}
|P_{\varepsilon, \Xi, W}|_{\Bp}^2 = \int_{(0,1)^2} \aninorm(\nabla P_{\varepsilon, \Xi, W})^2 \, \,\mathrm{d}x + o\left(\frac{1}{\varepsilon}\log_2\left(\frac{1}{\varepsilon}\right)^{-1}\right), \text{ for } \varepsilon \to 0.
\end{align}
The implicit constant is independent of $\Xi$ and $W$, if
\begin{align} \label{eq:specialAssumption}
\|W\|_\infty \leq \frac{1}{2}\varepsilon^{-\frac{1}{32}},\quad \text{ and } \quad \|\Xi'\|_\infty \left\| \frac{1}{W}\right\|_\infty \leq 2.
\end{align}

\end{proposition}
To prove the result above, we require the following auxiliary result that will be proved in Appendix \ref{app:ProofOfCovering}.
\begin{figure}[htb]
    \centering
    \includegraphics[width = 0.4\textwidth]{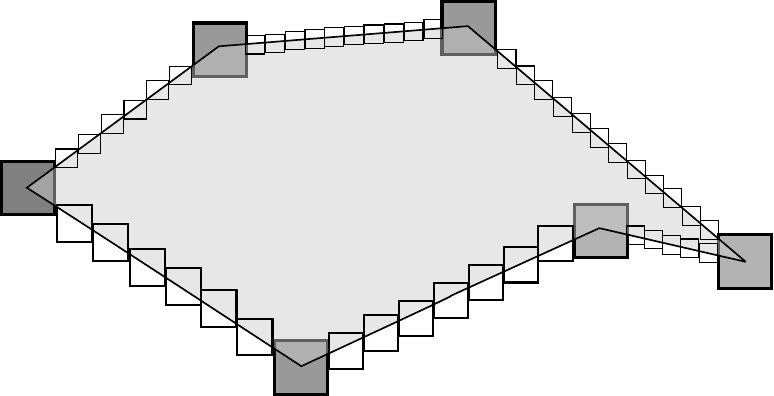}
    \caption{Covering of the boundary of a polygon as in Proposition \ref{prop:covering}}
    \label{fig:my_label}
\end{figure}

\begin{proposition}\label{prop:covering}
Let $N \in \N$ and let $P$ be an $N$-gon with vertices $x_1^P, \dots, x_N^P$.
Then, there exists an $r_0 >0$ and $c_0 \in (0, 1/2]$ such that for all $r \in (0, r_0)$
there are $L_1,\dots, L_N\in \N$ with $\ell(P)/r \geq \sum_{i=1}^N L_i \geq \ell(P)/(4r)$, and $(z_{k,i})_{k = 1,\dots, L_i, i = 1,\dots, N} \subset \partial P$, and $c_0 r/4 \leq s_1,\dots,  s_N \leq c_0 r/2$ such that
\begin{align} \label{eq:theCover}
\partial P &\subset \dot{\bigcup}_{i = 1,\dots, N} \left( C(x_i, r)   \ \dot{\cup} \  \dot{\bigcup}_{k = 1,\dots, L_i} C(z_{k,i}, s_i)\right), \text{ and }\\
g_i &\subset  C(x_i, r) \ \dot{\cup} \ C(x_{i+1}, r) \ \dot{\cup} \ \dot{\bigcup}_{k = 1,\dots, L_i} C(z_{k,i}, s_i), \text{ for all } i = 1, \dots, N, \nonumber
\end{align}
where $C(z, s) = z + (-s, s]^2$ is the semi-open cube of radius $s$ with center $z$ and $g_i$ is the line segment between $x_i$ and $x_{i+1}$.
Moreover, $\chi_{P \cap  C(z_{k,i}, s_k+ c_0 r/4)}( \cdot -z_{k,i})$ is a rotated Heaviside function on $[-s_k- c_0 r/4, s_k+ c_0 r/4]^2$ with the same normal direction as $\partial P$ at $z_{k,i}$ for all $k = 1, \dots,  L$.
\end{proposition}

\begin{proof}[Proof of Proposition \ref{prop:PolygonConvergence}]
Assume that $P$ is an $N$-gon and $\Xi$ is a directional transition profile. 
Let $W$ be a directional weight such that $W(\mathbb{S}^1) \subset [a,b]$, for $0<a<b<\infty$.

Let $r_0$ and $c_0$ be as in Proposition \ref{prop:PolygonConvergence}. Let $r \coloneqq 4 \varepsilon^{1/16}/c_0$ and $\varepsilon>0$ so small that $r < \min\{r_0, \mathrm{dist}(P, \R^2 \setminus (0,1)^2)\}$.
Let $L_1,\dots, L_N \in \N$, and  $(z_{k,i})_{k = 1,\dots, L_i, i = 1,\dots, N} \subset \partial P$, and $c_0 r/4 \leq s_1,\dots,  s_N \leq c_0 r/2$ be as in Proposition \ref{prop:covering}.
We define
$$
G^r \coloneqq [0,1]^2 \setminus \left(\dot{\bigcup}_{i = 1,\dots, N} \left(C(x_i, r) \  \dot{\cup}\  \dot{\bigcup}_{k = 1,\dots, L_i} C(z_{k,i}, s_i) \right)\right).
$$
We first prove that
\begin{align}\label{eq:geometricCondition}
\lambda\left(G^r \cap \left(\partial P + B_{\|W\|_\infty \varepsilon}(0)\right) \right) \lesssim \varepsilon^{2-1/16},
\end{align}
where $\lambda$ denotes the Lebesgue measure on $\R^2$. Given that, for $\varepsilon < 1$, $r > s_i \geq \varepsilon$, it is clear from Figure \ref{fig:my_label} that 
$\left|G^r \cap \left(\partial P + B_{\|W\|_\infty \varepsilon}(0)\right) \right|$ is bounded by the measure of $\mathcal{O}(\varepsilon^{-1/16})$ triangles with height at most $\varepsilon$ and base at most $\varepsilon$. This shows \eqref{eq:geometricCondition}.

We have that
\begin{align*}
|P_{\varepsilon, \Xi, W}|_{\Bp}^2 & = \sum_{i = 1}^N \sum_{k = 1}^{L_i} \sum_{\iota = -1,1}  \int_0^\mathfrak{a} a^{-2} \int_{-\mathfrak{s}}^{\mathfrak{s}} \int_{C(z_{k,i},s_i)} \left|\left \langle P_{\varepsilon, \Xi, W}, \psi_{a,s,t,\iota}^{\omega, per}\right \rangle\right|^2 a^{-3} \, \,\mathrm{d}t \, \mathrm{d}s \, \mathrm{d}a\\
& \quad + \sum_{i = 1}^{N} \sum_{\iota = -1,1}  \int_0^\mathfrak{a} a^{-2} \int_{-\mathfrak{s}}^{\mathfrak{s}} \int_{C(x_i,r)} \left|\left \langle  P_{\varepsilon, \Xi, W} , \psi_{a,s,t,\iota}^{\omega, per}\right \rangle\right|^2 a^{-3} \, \,\mathrm{d}t \, \mathrm{d}s \, \mathrm{d}a\\
&\quad \quad+ \sum_{\iota = -1,1}   \int_0^\mathfrak{a} a^{-2} \int_{-\mathfrak{s}}^{\mathfrak{s}}  \int_{G^r} \left|\left \langle  P_{\varepsilon, \Xi, W} , \psi_{a,s,t,\iota}^{\omega, per}\right \rangle\right|^2 a^{-3} \, \,\mathrm{d}t \, \mathrm{d}s \, \mathrm{d}a\\
& \geq \sum_{i = 1}^{N} \sum_{k = 1}^{L_i} \sum_{\iota = -1,1}  \int_0^\mathfrak{a} a^{-2} \int_{-\mathfrak{s}}^{\mathfrak{s}} \int_{C(z_{k,i},s_i)} \left|\left \langle  P_{\varepsilon, \Xi, W} , \psi_{a,s,t,\iota}^{\omega, per}\right \rangle\right|^2  a^{-3} \, \,\mathrm{d}t \, \mathrm{d}s \, \mathrm{d}a \eqqcolon \mathrm{I}.
\end{align*}
There exists an $R>0$ such that $\suppp \psi_{a,s,t,\iota}^{\omega} \subset B_{\sqrt{a}R}(t)$ for all $a\in (0, \mathfrak{a}]$, $s \in [-\mathfrak{s}, \mathfrak{s}]$,  $t \in \R^2$, and $\iota \in \{-1,1\}$.
Since $\varepsilon^{1/16} = c_0 r/4$, we conclude, from Proposition \ref{prop:covering} that all shearlets $\psi_{a,s,t,\iota}^{\omega, per}$ with $t \in C(z_{k,i},s_i)$ such that $P_{\varepsilon, \Xi, W} \neq H_{\varepsilon, \Xi(\vn_i), W(\vn_i)}^{\vn_i}(\cdot - z_{k,i})$ on $\suppp \psi_{a,s,t,\iota}^{\omega, per}$ have scale $a$ such that $\sqrt{a}R \geq \varepsilon^{1/16}-\|W\|_{\infty}\varepsilon$. Therefore, 
\begin{align*}
 \mathrm{I} = &\sum_{i = 1}^{N} \sum_{k = 1}^{L_i} \sum_{\iota = -1,1}  \int_0^\mathfrak{a} a^{-2} \int_{-\mathfrak{s}}^{\mathfrak{s}} \int_{C(z_{k,i},s_i)} \left|\left \langle  P_{\varepsilon, \Xi, W} , \psi_{a,s,t,\iota}^{\omega, per}\right \rangle\right|^2  a^{-3} \, \,\mathrm{d}t \, \mathrm{d}s \, \mathrm{d}a\\
  &= \sum_{i = 1}^{N} \sum_{k = 1}^{L_i} \sum_{\iota = -1,1}  \int_0^{\mathfrak{a}} a^{-2} \int_{-\mathfrak{s}}^{\mathfrak{s}} \int_{C(z_{k,i},s_i)} \left|\left \langle  H_{\varepsilon, \Xi(\vn_i), W(\vn_i)}^{\vn_i}(\cdot - z_{k,i}) , \psi_{a,s,t,\iota}^{\omega, per}\right \rangle\right|^2  a^{-3} \, \,\mathrm{d}t \, \mathrm{d}s \, \mathrm{d}a\\
  & \qquad + \ \sum_{i = 1}^{N} \sum_{k = 1}^{L_i} \sum_{\iota = -1,1} \int_{(\varepsilon^{1/16}-\|W\|_{\infty}\varepsilon)^2/R^2}^{\mathfrak{a}} a^{-2} \int_{-\mathfrak{s}}^\mathfrak{s} \int_{C(z_{k,i},s_i)} \left|\left \langle  P_{\varepsilon, \Xi, W}, \psi_{a,s,t,\iota}^{\omega, per}\right \rangle\right|^2 \\
  & \qquad \qquad - \left|\left \langle H_{\varepsilon, \Xi(\vn_i), W(\vn_i)}^{\vn_i}(\cdot - z_{k,i}) , \psi_{a,s,t,\iota}^{\omega, per}\right \rangle\right|^2  a^{-3} \, \,\mathrm{d}t \, \mathrm{d}s \, \mathrm{d}a.
\end{align*}
Assumption \ref{eq:theDecayAssumption} implies that the underlying shearlet system satisfies a Bessel inequality, see Appendix \ref{app:Bessel}. Hence,
\begin{align*}
     &\ \sum_{i = 1}^{N} \sum_{k = 1}^{L_i} \sum_{\iota = -1,1} \int_{(\varepsilon^{1/16}-\|W\|_{\infty}\varepsilon)^2/R^2}^{\mathfrak{a}} a^{-2} \int_{-\mathfrak{s}}^\mathfrak{s} \int_{C(z_{k,i},s_i)} \left|\left \langle  P_{\varepsilon, \Xi, W}, \psi_{a,s,t,\iota}^{\omega, per}\right \rangle\right|^2\\
     &\qquad - \left|\left \langle H_{\varepsilon, \Xi(\vn_i), W(\vn_i)}^{\vn_i}(\cdot - z_{k,i}) , \psi_{a,s,t,\iota}^{\omega, per}\right \rangle\right|^2  a^{-3} \, \,\mathrm{d}t \, \mathrm{d}s \, \mathrm{d}a\\
& \lesssim
    \ \sum_{i = 1}^{N} \sum_{k = 1}^{L_i} \sum_{\iota = -1,1}   (\varepsilon^{1/16}-\|W\|_{\infty}\varepsilon)^{-4} \int_{0}^{\mathfrak{a}} \int_{-\mathfrak{s}}^\mathfrak{s} \int_{C(z_{k,i},s_i)} \left| \left|\left \langle  P_{\varepsilon, \Xi, W}, \psi_{a,s,t,\iota}^{\omega, per}\right \rangle\right|^2 \right. \\
     &\qquad - \left. \left|\left \langle H_{\varepsilon, \Xi(\vn_i), W(\vn_i)}^{\vn_i}(\cdot - z_{k,i}) , \psi_{a,s,t,\iota}^{\omega, per}\right \rangle\right|^2 \right| a^{-3} \, \,\mathrm{d}t \, \mathrm{d}s \, \mathrm{d}a\\
& \lesssim \ \sum_{i = 1}^{N} \sum_{k = 1}^{L_i} \varepsilon^{-\frac14} \left\| P_{\varepsilon, \Xi, W} - H_{\varepsilon, \Xi(\vn_i), W(\vn_i)}^{\vn_i}(\cdot - z_{k,i})\right\|_{L^2((0,1)^2)}^2\\
& \lesssim \sum_{i = 1}^{N} \sum_{k = 1}^{L_i} \varepsilon^{-\frac14} =  \mathcal{O}\left(\varepsilon^{-\frac{1}{16}} \varepsilon^{-\frac14}\right) = \mathcal{O}\left(\varepsilon^{-\frac{1}{2}}\right), \text{ for } \varepsilon \to 0,
\end{align*}
where the implicit constants are independent of $\|W\|_\infty$ if $\|W\|_\infty\varepsilon \leq \varepsilon^{1/16}/2$. This is always satisfied if \eqref{eq:specialAssumption} holds. We get by the properties of the covering and Proposition \ref{prop:HeavisideConvergence} that
\begin{align*}
\mathrm{I}& =  \sum_{i = 1}^{N} \sum_{k = 1}^{L_i} \sum_{\iota = -1,1}  \int_0^\mathfrak{a} a^{-2} \int_{-\mathfrak{s}}^{\mathfrak{s}} \int_{C(z_{k,i},s_i)} \left|\left \langle H_{\varepsilon, \Xi(\vn_i), W(\vn_i)}^{\vn_i}(\cdot - z_{k,i}), \psi_{a,s,t,\iota}^\omega\right \rangle\right|^2  a^{-3} \, \,\mathrm{d}t \, \mathrm{d}s \, \mathrm{d}a + \mathcal{O}\left(\varepsilon^{-\frac{1}{2}}\right)\\
 & = \sum_{i = 1}^{N} \sum_{k = 1}^{L_i} \sum_{\iota = -1,1}  \int_0^\mathfrak{a} a^{-2} \int_{-\mathfrak{s}}^{\mathfrak{s}} \int_{[-s_i, s_i]^2} \left|\left \langle H_{\varepsilon, \Xi(\vn_i), W(\vn_i)}^{\vn_i}, \psi_{a,s,t,\iota}^\omega\right \rangle\right|^2  a^{-3} \, \,\mathrm{d}t \, \mathrm{d}s \, \mathrm{d}a + \mathcal{O}\left(\varepsilon^{-\frac{1}{2}}\right)\\
 &  \geq  \sum_{i = 1}^N  \sum_{k = 1}^{L_i} \aninorm\left( \vn_{i} \right)^2 \left|H_{\varepsilon, \Xi(\vn_{i}), W(\vn_{i})}^{\vn_{i}}\right|_{H^1((-s_i,s_i)^2)}^2\\
 & \qquad + \sum_{i = 1}^{N} \sum_{k = 1}^{L_i} \left(s_i\,  o\left(\frac{1}{\varepsilon}\log_2\left(\frac{1}{\varepsilon}\right)^{-1}\right) + o\left(\frac{1}{\sqrt{\varepsilon}}\right) + \mathcal{O}\left( \frac{1}{(\varepsilon^{\frac{1}{16}} - w \varepsilon(\mathfrak{s} + 2) )^4}\right)\right), \text{ for } \varepsilon \to 0.
 \end{align*}
As $\sum_{i = 1}^{N} \sum_{k = 1}^{L_i} s_i = \mathcal{O}(1)$ and $\sum_{i = 1}^{N} \sum_{k = 1}^{L_i} = \mathcal{O}(\varepsilon^{-1/16})$, for $\varepsilon \to 0$, we have that
\begin{align*}
|P_{\varepsilon, \Xi, W}|_{\Bp}^2 &  \geq  \sum_{i = 1}^N  \sum_{k = 1}^{L_i} \aninorm\left( \vn_{i} \right)^2  \left|H_{\varepsilon, \Xi(\vn_{i}) W(\vn_{i})}^{\vn_{i}}\right|_{H^1((-s_i,s_i)^2)}^2 +   o\left(\frac{1}{\varepsilon}\log_2\left(\frac{1}{\varepsilon}\right)^{-1}\right)\\
& = \int_{(0,1)^2 \setminus (\bigcup_{i = 1}^N C(x_i,r) \cup B^r)}  \aninorm(\nabla P_{\varepsilon, \Xi, W} )(x)^2 \, \,\mathrm{d}x +   o\left(\frac{1}{\varepsilon}\log_2\left(\frac{1}{\varepsilon}\right)^{-1}\right), \text{ for } \varepsilon \to 0,
 \end{align*}
where the last line follows because, for every vector $\vec{x} \in \R^2$, we have that $\aninorm(\vec{x}) = \aninorm(\vec{x}/|\vec{x}|) \cdot |\vec{x}|$ and $\nabla P_{\varepsilon, \Xi, W} / |\nabla P_{\varepsilon, \Xi, W} | = \vn_i$ on $C(z_{k,i},s_i)$ by construction. 
Moreover, the support of $\nabla P_{\varepsilon, \Xi, W}$ is contained in an $\varepsilon\|W\|_\infty$ neighbourhood of $\partial P$. Therefore
\begin{align}\label{eq:ThisIntegralWillVanish}
 \int_{ B^r \setminus ( \partial P  + B_{\varepsilon\|W\|_\infty}(0))}  \aninorm(\nabla P_{\varepsilon, \Xi, W} )(x)^2 \, \,\mathrm{d}x = 0.
\end{align}
Additionally, $\nabla P_{\varepsilon, \Xi, W}$ exists almost everywhere and $\|\nabla P_{\varepsilon, \Xi, W}\|_\infty^2 \lesssim \|\Xi'\|_\infty^2 \|1/W\|_\infty^2  \varepsilon^{-2}$ by construction.
Therefore, by \eqref{eq:ThisIntegralWillVanish} and \eqref{eq:geometricCondition},
\begin{align}\label{eq:VanishingOfTheBrTerm}
 \int_{ B^r}  \aninorm(\nabla P_{\varepsilon, \Xi, W} )(x)^2 \, \,\mathrm{d}x =  \int_{ B^r \cap ( \partial P  + B_{\varepsilon\|W\|_\infty}(0))}  \aninorm(\nabla P_{\varepsilon, \Xi, W} )(x)^2 \, \,\mathrm{d}x = \mathcal{O}(\varepsilon^{-1/16}),
\end{align}
for $\varepsilon \to 0$ and an implicit constant independent of $W$ and $\Xi$ for sufficiently small $\varepsilon>0$, or if $\|\Xi'\|_\infty^2 \cdot \|1/W\|_\infty^2 \leq 4$, which is always satisfied if \eqref{eq:specialAssumption} holds. 
Moreover, since $r = \mathcal{O}(\varepsilon^{1/16})$, for $\varepsilon \to 0$, this yields that
\begin{align} \label{eq:estimateOnCornerPoints}
 \int_{\bigcup_{i = 1}^N C(x_i,r)}  \aninorm(\nabla P_{\varepsilon, \Xi, W} )(x)^2 \, \,\mathrm{d}x \lesssim  N\cdot \left(\|\Xi'\|_\infty^2  \|1/W\|_\infty^2  \varepsilon^{-2}\right) \cdot \left(r \varepsilon \|W\|_\infty\right) = N \cdot \mathcal{O}(\epsilon^{-31/32})
\end{align}
for $\varepsilon \to 0$, with an implicit constant independent of $W$ and $\Xi$, if $\|\Xi'\|_\infty^2 \cdot \|1/W\|_\infty^2  \cdot \|W\|_\infty \leq \varepsilon^{-1/32}$, which is the case for sufficiently small $\varepsilon>0$ and fixed $\Xi$, $W$, and always if \eqref{eq:specialAssumption} is satisfied.
Hence,
\begin{align}\label{eq:TheUpperBound}
|P_{\varepsilon, \Xi, W}|_{\Bp}^2 &  \geq  \int_{(0,1)^2}  \aninorm(\nabla P_{\varepsilon, \Xi, W} )(x)^2 \, \,\mathrm{d}x + o\left(\frac{1}{\varepsilon}\log_2\left(\frac{1}{\varepsilon}\right)^{-1}\right), \text{ for } \epsilon \to 0.
\end{align}
To establish the lower bound, we define
$$
\widetilde{G}^r \coloneqq [0,1]^2 \setminus \left(\dot{\bigcup}_{i = 1,\dots, N} \left( C(x_i, r) \ \dot{\cup} \   \dot{\bigcup}_{k = 1,\dots, L_i} C\left(z_{k,i}, s_i + \varepsilon^{1/10}/2 \right)\right)\right).
$$
Then, we have that
\begin{align*}
    |P_{\varepsilon, \Xi, W}|_{\Bp}^2 & \leq  \sum_{i = 1}^{N} \sum_{k = 1}^{L_i} \sum_{\iota = -1,1} \int_0^\mathfrak{a} a^{-2} \int_{-\mathfrak{s}}^{\mathfrak{s}} \int_{C(z_{k,i},s_i +\varepsilon^{1/10}/2)} \left|\left \langle P_{\varepsilon, \Xi, W} , \psi_{a,s,t,\iota}^{\omega, per} \right \rangle\right|^2 a^{-3} \, \mathrm{d}s \, \,\mathrm{d}t \, \mathrm{d}a\\
    & \quad + \sum_{i = 1}^{N} \sum_{\iota = -1,1}  \int_0^\mathfrak{a} a^{-2} \int_{-\mathfrak{s}}^{\mathfrak{s}}  \int_{C(x_i,r)} \left|\left \langle  P_{\varepsilon, \Xi, W} , \psi_{a,s,t,\iota}^{\omega, per}\right \rangle\right|^2  a^{-3} \, \mathrm{d}s \, \,\mathrm{d}t \, \mathrm{d}a\\
    &\quad \quad+ \sum_{\iota = -1,1}  \int_0^\mathfrak{a} a^{-2} \int_{-\mathfrak{s}}^{\mathfrak{s}} \int_{\widetilde{G}^r} \left|\left \langle  P_{\varepsilon, \Xi, W} , \psi_{a,s,t,\iota}^{\omega, per}\right \rangle\right|^2 a^{-3} \, \mathrm{d}s \, \,\mathrm{d}t \, \mathrm{d}a.
\end{align*}
By construction, for $t \in \widetilde{G}^r$, all shearlets $\psi_{a,s,t,\iota}$ such that $\suppp \psi_{a,s,t,\iota}$ intersects $\partial P + B_{ \varepsilon \|W\|_\infty}(0)$ non-trivially must satisfy $\sqrt{a}R \geq \varepsilon^{1/10}/2 - \|W\|_\infty \varepsilon$. We assume that $\varepsilon$ is so small that
$\varepsilon^{1/10}/2 - \|W\|_\infty \varepsilon \geq \varepsilon^{1/10}/4$. This is always satisfied independent of $W$ if \eqref{eq:specialAssumption} holds. Thus, we have $a^{-2} \leq 256 R^4  \varepsilon^{-4/10}$ if  $\suppp \psi_{a,s,t,\iota} \cap (\partial P + B_{ \varepsilon \|W\|_\infty}(0)) \neq \emptyset$ .
Invoking again the Bessel inequality of the periodised shearlet system, see Appendix \ref{app:Bessel}, this implies that
\begin{align*}
&\sum_{\iota = -1,1}  \int_0^\mathfrak{a} a^{-2} \int_{-\mathfrak{s}}^{\mathfrak{s}} \int_{\widetilde{G}^r} \left|\left \langle  P_{\varepsilon, \Xi, W} , \psi_{a,s,t,\iota}^{\omega, per}\right \rangle\right|^2 a^{-3}\, \,\mathrm{d}t \, \mathrm{d}s \, \mathrm{d}a  \\
& \qquad \leq256 R^4  \varepsilon^{-4/10}  \sum_{\iota = -1,1}  \int_{0}^\mathfrak{a}\int_{-\mathfrak{s}}^{\mathfrak{s}} \int_{\widetilde{G}^r} \left|\left \langle  P_{\varepsilon, \Xi, W} , \psi_{a,s,t,\iota}^{\omega, per}\right \rangle\right|^2 a^{-3}\, \,\mathrm{d}t \, \mathrm{d}s \, \mathrm{d}a  = o\left(\frac{1}{\varepsilon}\log_2\left(\frac{1}{\varepsilon}\right)^{-1}\right), \text{ for } \varepsilon \to 0.
\end{align*}
Moreover, for $\varepsilon \to 0$,
\begin{align*}
    & |P_{\varepsilon, \Xi, W}|_{\Bp}^2\\ & \leq \sum_{i = 1}^{N} \sum_{k = 1}^{L_i} \sum_{\iota = -1,1} \int_0^\mathfrak{a} a^{-2} \int_{-\mathfrak{s}}^{\mathfrak{s}} \int_{C(z_{k,i},s_i + \varepsilon^{1/10}/2)} \left|\left \langle P_{\varepsilon, \Xi, W} , \psi_{a,s,t,\iota}^{\omega, per} \right \rangle\right|^2 a^{-3} \, \mathrm{d}s \, \,\mathrm{d}t \, \mathrm{d}a +   o\left(\frac{1}{\varepsilon}\log_2\left(\frac{1}{\varepsilon}\right)^{-1}\right)\\
   & = \sum_{i = 1}^{N} \sum_{k = 1}^{L_i} \sum_{\iota = -1,1} \int_0^{\varepsilon^{1/8}/(4R^2)} a^{-2} \int_{-\mathfrak{s}}^{\mathfrak{s}} \int_{C(z_{k,i},s_i + \varepsilon^{1/10}/2)} \left|\left \langle P_{\varepsilon, \Xi, W} , \psi_{a,s,t,\iota}^{\omega, per} \right \rangle\right|^2 a^{-3} \, \mathrm{d}s \, \,\mathrm{d}t \, \mathrm{d}a\\
      & \qquad +  \sum_{i = 1}^{N} \sum_{k = 1}^{L_i} \sum_{\iota = -1,1} \int_{\varepsilon^{1/8}/(4R^2)}^{\mathfrak{a}} a^{-2} \int_{-\mathfrak{s}}^{\mathfrak{s}} \int_{C(z_{k,i},s_i + \varepsilon^{1/10}/2)} \left|\left \langle P_{\varepsilon, \Xi, W} , \psi_{a,s,t,\iota}^{\omega, per} \right \rangle\right|^2 a^{-3} \, \mathrm{d}s \, \,\mathrm{d}t \, \mathrm{d}a +   o\left(\frac{1}{\varepsilon}\log_2\left(\frac{1}{\varepsilon}\right)^{-1}\right)\\
    & \eqqcolon \mathrm{II} + \mathrm{III} +   o\left(\frac{1}{\varepsilon}\log_2\left(\frac{1}{\varepsilon}\right)^{-1}\right).
      \end{align*}
We have that, if $a \leq \varepsilon^{1/8}/(4R^2)$ then $\sqrt{a}R \leq \varepsilon^{1/16}/2$ and hence
$$
\suppp \psi_{a,s,t,\iota}^{\omega, per} \subset C(z_{k,i},s_i + \varepsilon^{1/10}/2 + \varepsilon^{1/16}/2) \subset C(z_{k,i},s_i + c_0 r/4).
$$
Therefore, by Proposition \ref{prop:covering},
\begin{align}\label{eq:FormOfIIInMainResult}
\mathrm{II} =  \sum_{i = 1}^{N} \sum_{k = 1}^{L_i} \sum_{\iota = -1,1} \int_0^{\varepsilon^{1/8}/(4R^2)} a^{-2} \int_{-\mathfrak{s}}^{\mathfrak{s}} \int_{C(z_{k,i},s_i + \varepsilon^{1/10}/2)}  \left|\left \langle H_{\varepsilon, \Xi(\vn_i), W(\vn_i)}^{\vn_i}, \psi_{a,s,t,\iota}^\omega\right \rangle\right|^2 a^{-3} \, \mathrm{d}s \, \,\mathrm{d}t \, \mathrm{d}a.
\end{align}
Moreover, 
\begin{align*}
\mathrm{III} &\lesssim \varepsilon^{-\frac{1}{2}}\sum_{i = 1}^{N} \sum_{k = 1}^{L_i} \sum_{\iota = -1,1} \int_{\varepsilon^{1/8}/(4R^2)}^{\mathfrak{a}} \int_{-\mathfrak{s}}^{\mathfrak{s}} \int_{C(z_{k,i},s_i + \varepsilon^{1/10}/2)} \left|\left \langle P_{\varepsilon, \Xi, W} , \psi_{a,s,t,\iota}^{\omega, per} \right \rangle\right|^2 a^{-3} \, \mathrm{d}s \, \,\mathrm{d}t \, \mathrm{d}a\\
&= \mathcal{O}(\varepsilon^{-\frac{1}{2}}), \text{ for } \varepsilon \to 0,
\end{align*}
by the Bessel inequality of Appendix \ref{app:Bessel}.

We proceed by estimating $\mathrm{II}$. By \eqref{eq:FormOfIIInMainResult}, and Proposition \ref{prop:HeavisideConvergence}, we have that 
\begin{align*}      
     \mathrm{II} &  \leq  \sum_{i = 1}^{N} \sum_{k = 1}^{L_i} \aninorm(\vn_i) ^2 \left| H_{\varepsilon, \Xi(\vn_i), W(\vn_i)}^{\vn_i}\right|_{H^1((-s_i + \varepsilon^{1/10}/2,s_i + \varepsilon^{1/10}/2)^2)}^2\\
     & \leq \sum_{i = 1}^{N}  \left(1+ \frac{\varepsilon^{1/10}/2}{s_i}\right) \sum_{k = 1}^{L_i} \aninorm(\vn_i)^2  \left| H_{\varepsilon, \Xi(\vn_i), W(\vn_i)}^{\vn_i}\right|_{H^1((-s_i,s_i)^2)}^2\\
     & \leq \sum_{i = 1}^{N}  \sum_{k = 1}^{L_i} \aninorm(\vn_i)^2  \left| H_{\varepsilon, \Xi(\vn_i), W(\vn_i)}^{\vn_i}\right|_{H^1((-s_i,s_i)^2)}^2  + o\left(\frac{1}{\varepsilon}\log_2\left(\frac{1}{\varepsilon}\right)^{-1}\right) , \text{ for } \varepsilon \to 0,
 \end{align*}
where the last estimate follows since ${\varepsilon^{1/10}}/{ s_i} = \mathcal{O}(\varepsilon^{3/80})$ and $\sum_{i = 1}^{N}  \sum_{k = 1}^{L_i} \aninorm(\vn_i)^2  | H_{\varepsilon, \Xi(\vn_i), W(\vn_i)}^{\vn_i}|_{H^1((-s_i,s_i)^2)}^2 = \mathcal{O}(1/\varepsilon)$, for $\varepsilon \to 0$, where the implicit constant is independent of $\Xi$ and $W$ if $\|\Xi'\|\cdot \|1/W\|_\infty \leq 2$, i.e., always if \eqref{eq:specialAssumption} is satisfied.
We conclude that 
\begin{align}
      \mathrm{II}  & \leq \int_{(0,1)^2 \setminus (\bigcup_{i = 1}^N C(x_i,r) \cup \widetilde{G}^r)}  \aninorm(\nabla P_{\varepsilon, \Xi, W} )(x)^2 \, \,\mathrm{d}x +  o\left(\frac{1}{\varepsilon}\log_2\left(\frac{1}{\varepsilon}\right)^{-1}\right) , \text{ for } \varepsilon \to 0.
 \end{align}
Invoking \eqref{eq:estimateOnCornerPoints}, we observe that
\begin{align*}
|P_{\varepsilon, \Xi, W}|_{\Bp}^2 & \leq \int_{(0,1)^2 \setminus \widetilde{G^r}}  \aninorm(\nabla P_{\varepsilon, \Xi, W} )(x)^2 \, \,\mathrm{d}x +   o\left(\frac{1}{\varepsilon}\log_2\left(\frac{1}{\varepsilon}\right)^{-1}\right) , \text{ for } \varepsilon \to 0.
\end{align*}
Since $\widetilde{G}^r \subset G^r$, we conclude with \eqref{eq:VanishingOfTheBrTerm} that 
\begin{align*}
|P_{\varepsilon, \Xi, W}|_{\Bp}^2 & \leq \int_{(0,1)^2}  \aninorm(\nabla P_{\varepsilon, \Xi, W} )(x)^2 \, \,\mathrm{d}x +   o\left(\frac{1}{\varepsilon}\log_2\left(\frac{1}{\varepsilon}\right)^{-1}\right) , \text{ for } \varepsilon \to 0.
\end{align*}
Together with \eqref{eq:TheUpperBound} we obtain \eqref{assertionOfPolygonConvergenceProp}. We observed in the course of the proof that \eqref{eq:specialAssumption} implies that the implicit constants are independent of $\Xi$ and $W$ if \eqref{eq:specialAssumption} is satisfied.
\end{proof}

Proposition \ref{prop:PolygonConvergence} shows that for functions of the form $P_{\varepsilon, \Xi, W}$:
$$
\varepsilon \cdot \left| |P_{\varepsilon, \Xi, W}|_{\Bp}^2 -  \int_{(0,1)^2}  \aninorm(\nabla P_{\varepsilon, \Xi, W} )(x)^2 \, \,\mathrm{d}x \right| \to 0, \text{ for } \varepsilon \to 0.
$$
Hence,
$$
    \left|\mathrm{SGL}^\omega_\varepsilon(P_{\varepsilon, \Xi, W}) - \mathrm{GL}_\varepsilon^{\aninorm}(P_{\varepsilon, \Xi, W})\right| \to 0 , \text{ for } \varepsilon \to 0
$$
with constants independent of $\Xi, W$ if \eqref{eq:specialAssumption} holds.
It was demonstrated in \cite[Proposition 4.10]{braides1998approximation} that there exists a sequence of functions $(u_\varepsilon)_{\varepsilon >0}$ of the form $(P_{\varepsilon, \Xi_\varepsilon, W_\varepsilon})_{\varepsilon >0}$,
with $\min W_\varepsilon \leq \|\Xi_\varepsilon'\|_\infty \leq \max W_\varepsilon$ such that
\begin{align*}
    \lim_{\varepsilon \to 0}\mathrm{GL}_\varepsilon^{\aninorm}(u_\varepsilon) = P_\aninorm(\chi_P).
\end{align*}
Moreover, $\|\Xi_\varepsilon'\|_\infty \to \infty $ at any sufficiently slow rate. In particular, we can choose $\|\Xi_\varepsilon'\|_\infty  \leq \varepsilon^{-\frac{1}{128}}$. This yields existence of a recovery sequence.

\subsection{Recovery sequences for characteristic functions of sets with finite perimeter}\label{sec:recSequenceForFinitePerimeter}

The standard method for the construction of recovery sequences in $\Gamma$-convergence arguments is to establish lower semicontinuity of the limit functional and then reduce the problem to
finding a recovery sequence for functions from a dense, but more accessible space, only. Indeed the anisotropic perimeter functional is lower semi-continuous \cite[Theorem 20.1]{maggi2012sets} and the set of polygons is dense in the set of sets of finite perimeter. This yields the following proposition.

\begin{proposition}\label{prop:PieceSmoothConvergence}
Let $\psi \in L^2(\R^2)$ satisfy the assumptions of Proposition \ref{prop:HeavisideConvergence}.
Let $\aninorm$ be a norm on $\R^2$ and let $\omega$ be a directional weight associated with $\aninorm$ .
Assume that $D \subset (0,1)^2$ is a set of finite perimeter. Then, there exists a sequence $(u_\varepsilon)_{\varepsilon>0} \subset H^1((0,1)^2)$ with $u_\varepsilon \to \chi_D$ for $\varepsilon \to 0$:
\begin{align*}
\limsup_{\varepsilon>0} \mathrm{SGL}^\omega_\varepsilon(u_\varepsilon) \leq P_\aninorm(\chi_D).
\end{align*}
\end{proposition}

\section{The set $\mathcal{B}_{\aninorm}$}\label{sec:TheSetB}

The set $\mathcal{B}_{\aninorm}$ describes the feasible functions for the shearlet-based Ginzburg--Landau energy.
It is somewhat unsatisfying that the shearlet-based Ginzburg--Landau energy still depends---at least indirectly---on the anisotropic Sobolev seminorm.
Nonetheless, most of the relevant functions appearing in a phase-field problem are contained in $\mathcal{B}_{\aninorm}$. Indeed, it was demonstrated in Proposition \ref{prop:PolygonConvergence} that all
functions of the form $P_{\varepsilon, \Xi, W}$ for any directional transition profile $\Xi$ and any directional width $W$ are elements of $\mathcal{B}_{\aninorm}$, as long as $\varepsilon$ is sufficiently small.
Additionally, the definition of $\mathcal{B}_{\aninorm}$ demonstrates that for sufficiently small $\varepsilon>0$ also smooth perturbations of functions of the form $P_{\varepsilon, \Xi, W}$ are contained in $\mathcal{B}_{\aninorm}$.
Indeed, if $(f_\varepsilon)_{\varepsilon >0} \subset H^1((0,1)^2)$ with
$$
|f_\varepsilon|_{H^1((0,1)^2)}^2 = o\left(\varepsilon^{-1}/\log_2(2+\varepsilon^{-1})\right), \text{ for } \varepsilon \to 0,
$$
then, for sufficiently small $\varepsilon>0$, we have that $P_{\varepsilon, \Xi, W} + f_\varepsilon\in \mathcal{B}_{\aninorm}$. This is because,
$$
\left|\left|f_\varepsilon + P_{\varepsilon, \Xi, W}\right|_{\Bp} - \left| P_{\varepsilon, \Xi, W}\right|_{\Bp} \right| \leq \left|f_\varepsilon\right|_{\Bp} \leq \left|f_\varepsilon\right|_{H^1((0,1)^2)},
$$
by the reverse triangle inequality and by Proposition \ref{prop:equivalence}, particularly \eqref{eq:normInEqBdDomain}.

We observe with equation \eqref{eq:FreqRep} below (after adding weights), that we have an alternative representation of $|f|_{B}$ as $\|\mathcal{S}(f)\|_{L^2(\R^2)}$ where, for $f \in L^2(\R^2)$,
\begin{align*}
\mathcal{S}(f) \coloneqq & \ \mathcal{F}^{-1}\left(\xi \mapsto \hat{f}(\xi) \cdot \left( |\widehat{K}(\xi)|^2 + |\xi_1|^2 \cdot \int_{-\mathfrak{s}}^{\mathfrak{s}} \int_{0}^{\mathfrak{a}} |\omega(1, s)|^2 \left|\widehat{\mu}(a \xi_1, \sqrt{a}(\xi_2 + s\xi_1) )\right|^2  a^{-\frac32} \, \mathrm{d}a \, \mathrm{d}s \right.\right.\\
&\qquad  + \left. \left.|\xi_2|^2 \cdot  \int_{-\mathfrak{s}}^{\mathfrak{s}} \int_{0}^{\mathfrak{a}} |\omega(-1, s)|^2 \left|\widehat{\widetilde{\mu}}(\sqrt{a}(\xi_1 + s\xi_2) , a \xi_2)\right|^2  a^{-\frac32} \, \mathrm{d}a \, \mathrm{d}s \right)^{\frac{1}{2}}\right).
\end{align*}
If $\mathfrak{d}(\suppp u_\varepsilon, \R^2 \setminus (0,1)^2)$ is sufficiently large then there exists an $\mathfrak{a}_0>0$ such that
$\langle u_\varepsilon, \psi_{a,s,t,\iota}\rangle = 0$ for all $t \not \in (0,1)^2$ and $a < \mathfrak{a}_0$. In this case, we have that $|u_\varepsilon|_{\Bp} = \|\mathcal{S}_p (u_\varepsilon)\|_{L^2(\R^2)} + o(\|u_\varepsilon\|_{L^2(\R^2)})$, for $\varepsilon \to 0$, where
\begin{align*}
\mathcal{S}_p(f) \coloneqq & \ \mathcal{F}^{-1}\left(\xi \mapsto  \widehat{f}(\xi) \cdot \left( |\xi_1|^2 \cdot \int_{-\mathfrak{s}}^{\mathfrak{s}} \int_{0}^{\mathfrak{a}_0} |\omega(1, s)|^2 \left|\widehat{\mu}(a \xi_1, \sqrt{a}(\xi_2 + s\xi_1) )\right|^2  a^{-\frac32} \, \mathrm{d}a \, \mathrm{d}s \right.\right.\\
&\qquad  + \left. \left.|\xi_2|^2 \cdot  \int_{-\mathfrak{s}}^{\mathfrak{s}} \int_{0}^{\mathfrak{a}_0} |\omega(-1, s)|^2 \left|\widehat{\widetilde{\mu}}(\sqrt{a}(\xi_1 + s\xi_2) , a \xi_2)\right|^2  a^{-\frac32} \, \mathrm{d}a \, \mathrm{d}s \right)^{\frac{1}{2}}\right).
\end{align*}
For a sequence $u_\varepsilon \to \chi_D$, where $D \subset (0,1)^2$, it appears to be reasonable to
study the action of the pseudo-differential operator $\mathcal{S}_p$ on $u_\varepsilon$ to understand the behaviour of the $|u_\varepsilon|_{\Bp}$ norm. Due to these considerations, a promising direction to remove the restriction to be a member of the set $B_\aninorm$ could be to study the microlocal behaviour of the sequences $u_\varepsilon \to \chi_D$, where $D \subset (0,1)^2$ and such that $\sup_{\varepsilon>0} \mathrm{SGL}_\varepsilon(u_\varepsilon) < \infty$.

\section{The discrete shearlet-based Ginzburg--Landau energy}

The construction of the shearlet-based Ginzburg--Landau energy of Definition \ref{eq:ShearBasedFunctional} is based on the continuous shearlet transform. In practice, it is more appropriate to work with a discrete energy, where the integrals over the parameters in the shearlet-based Besov seminorm are replaced by sums. Indeed, we will show that the shearlet-based Ginzburg--Landau energy can be uniformly approximated by a discrete variant such that both energies have the same $\Gamma$-limit.

We start by defining a discrete shearlet-based Besov seminorm. Let $\mathfrak{a} > 1$ and $\mathfrak{s}>0$; then, for a \emph{sampling density} $c >0$, we define $J_{c} \coloneqq c\N_0 - \log_2(\mathfrak{a})$ and, for all $j\in J_c$, we set $K_{j, c} \coloneqq \left\{k \in -\mathfrak{s}  + c\Z \colon |k| \leq 2^{\frac{j}{2}}\mathfrak{s} \right\}$. Moreover, we define the following \emph{rounding operators}: for $x \in \R$,
\begin{align} \label{eq:roundingOp1}
[x]_c \coloneqq \max\big\{ x^* \in c\N_0 - \log_2(\mathfrak{a}), x^* \leq \max\{-\log_2(\mathfrak{a}), x\}\big\}
\end{align}
and, for $x \in [-2^{\frac{j}{2}}\mathfrak{s}, 2^{\frac{j}{2}}\mathfrak{s}]$,
\begin{align} \label{eq:roundingOp2}
[x]_{j,c} \coloneqq \max \left\{x^*\in K_{j, c}, x^*\leq  x\right\}.
\end{align}
Next, we define the \emph{discrete shearlet-based Besov seminorm with sampling density $c$}: let $\psi \in L^2(\R^2)$, let $\omega$ be a directional weight, $c>0$, and define $A_{1, 2^{-j}} \coloneqq A_{2^{-j}}$ and $A_{-1, 2^{-j}}\coloneqq \widetilde{A}_{2^{-j}}$; then,
\begin{align*}
|f|_{\DBp, c}^2 \coloneqq c^4\sum_{\iota = -1,1} \sum_{j \in J_c} \sum_{k \in K_{c,j}} \sum_{\substack{m \in c A_{\iota, 2^{-j}} \Z^2,\\ m \in [0,1]^2}} 2^{2j} \left|\left\langle f, \psi_{2^{-j},2^{-j/2} k, m, \iota}^{\omega, per} \right\rangle \right|^2 \in [0,\infty], \quad \text{ for } f \in L^2(\R^2).
\end{align*}
Let $T: (0,1] \to (0,1]$ be a map such that $T(x) \to 0$ for $x \to 0$. We then define the associated \emph{discrete Ginzburg--Landau energy} by
\begin{align} \label{eq:DiscrShearBasedFunctional}
\mathrm{{DSGL}}^\omega_{\varepsilon, T}(u) \coloneqq \left\{ \begin{array}{l l}
   \varepsilon| u |_{\DBp, T(\varepsilon)}^2 + \frac{1}{4\varepsilon}\int_{(0,1)^2}\mathcal{W}(u)(x) \, \,\mathrm{d}x, & \text{ if } u \in \mathcal{B}_\aninorm,\\
    \infty, & \text{ if } u \in BV \setminus \mathcal{B}_\aninorm. \end{array}\right.
\end{align}
We shall now study under what conditions on the map $T$ and a sequence $(u_\varepsilon)_{\varepsilon>0} \subset H^1((0,1)^2)$ we have that $|\mathrm{{DSGL}}^\omega_{\varepsilon, T}(u_\varepsilon) - \mathrm{{SGL}}^\omega_{\varepsilon}(u_\varepsilon) | \to 0$, for $\varepsilon \to 0$. In other words, we analyse how quickly the sampling density needs to decrease to guarantee convergence of the associated discrete shearlet-based Besov seminorm to the continuous seminorm.

This will turn out to be the case if $\|u_\varepsilon\|_{H^1((0,1)^2)} = o(T(\varepsilon)^{-1})$, for $\varepsilon \to 0$. Moreover, we will see that if $T(\varepsilon) = o(\sqrt{\varepsilon})$, for $\varepsilon \to 0$, then for any sequence $(u_\varepsilon)_{\varepsilon>0} \subset H^1((0,1)^2)$ either $|\mathrm{{DSGL}}^\omega_{\varepsilon, T}(u_\varepsilon) - \mathrm{{SGL}}^\omega_{\varepsilon}(u_\varepsilon) | \to 0$ or $\mathrm{{DSGL}}^\omega_{\varepsilon, T}(u_\varepsilon) \to \infty$ and $\mathrm{{SGL}}^\omega_{\varepsilon}(u_\varepsilon) \to \infty$, for $\varepsilon \to 0$.

\begin{theorem} \label{thm:discrete}
Let $\psi \in L^2(\R^2)$ satisfy the assumptions of Proposition \ref{prop:equivalence} and be such that
$$
|\widehat{\psi}(\xi)| \lesssim \frac{\min\{|\xi_1|^M, 1\}}{ (1+|\xi_1|^2)^{L/2} (1+|\xi_2|^2)^{L/2}}, \quad \text{ for all } \xi \in \R^2,
$$
with $M>3$ and $N >4$. Let $T: (0,1] \to (0,1]$, and let $(u_\varepsilon)_{\varepsilon>0}\subset H^1_0((0,1)^2)$ be such that
$
|u_\varepsilon|_{H^1((0,1)^2)} = o(T(\varepsilon)^{-1})$, for $\varepsilon \to 0$.
Then,
$
|\mathrm{{DSGL}}^\omega_{\varepsilon, T}(u_\varepsilon) - \mathrm{{SGL}}^\omega_{\varepsilon}(u_\varepsilon) | \to 0 \text{ for } \varepsilon \to 0.
$
If $\sup_{\varepsilon>0}\|u_\varepsilon\|_{L^2((0,1)^2)} < \infty$ and
$
\liminf_{\varepsilon>0} \sqrt{\varepsilon} \|u_\varepsilon\|_{H^1((0,1)^2)} = \infty
$
then
$
\liminf_{\varepsilon>0} \mathrm{{DSGL}}^\omega_{\varepsilon, T}(u_\varepsilon) = \infty.
$

In particular, if $T(\varepsilon) = o(\sqrt{\varepsilon})$ for $\varepsilon \to 0$, then $\mathrm{{DSGL}}^\omega_{\varepsilon, T}$ and $\mathrm{{SGL}}^\omega_{\varepsilon}$ have the same $\Gamma$-limit.

\end{theorem}
\begin{proof}
In view of \eqref{eq:DiscrShearBasedFunctional}, it is sufficient to show that, if $(u_\varepsilon)_{\varepsilon>0}\subset H^1_0((0,1)^2)$ such that $|u_\varepsilon|_{H^1((0,1)^2)} = o(T(\varepsilon)^{-1})$, for $\varepsilon \to 0$, then
$$
\left| |u_\varepsilon|_{\Bp}^2 - |u_\varepsilon|_{\DBp, T(\varepsilon)}^2 \right| \to 0 \quad \text{ for } \varepsilon \to 0,
$$
and if $\liminf_{\varepsilon>0} \varepsilon |u_\varepsilon|_{H^1((0,1)^2)}^2  = \infty$, then $\liminf_{\varepsilon>0} \varepsilon |u_\varepsilon|_{\DBp, T(\varepsilon)}^2 = \infty$.

We observe that there exists an $R \coloneqq R(\mathfrak{a}, \mathfrak{s})\in \N$ such that $\suppp \psi_{a,s,t,\iota} \subset B_{R}(t)$ for all $(a,s,t,\iota) \in (0,\mathfrak{a}] \times [-\mathfrak{s}, \mathfrak{s}] \times \R^2 \times \{-1,1\}$.

By similar arguments to those following equation \eqref{eq:upperbound}, we observe that, for all $f \in H^1_0((0,1)^2)$,
\begin{align}
|f|_{\Bp}^2 = & \ \frac{1}{N^2} \sum_{\iota = -1,1} \int_{0}^{\mathfrak{a}} \int_{-\mathfrak{s}}^{ \mathfrak{s}} a^{-2} \int_{[0,N]^2} \left|\left\langle f_{[-R,N +  R]}, \psi_{a,s,t, \iota}^{\omega}\right\rangle_{L^2(\R^2)}\right|^2 a^{-3} \, \,\mathrm{d}t \, \mathrm{d}s \, \mathrm{d}a \nonumber \\
 = & \ \frac{1}{N^2} \sum_{\iota = -1,1} \int_{0}^{\mathfrak{a}} \int_{-\mathfrak{s}}^{ \mathfrak{s}} a^{-2} \int_{\R^2} \left|\left\langle f_{[-R,N + R]}, \psi_{a,s,t, \iota}^{\omega}\right\rangle_{L^2(\R^2)}\right|^2 a^{-3} \, \,\mathrm{d}t \, \mathrm{d}s \, \mathrm{d}a \nonumber \\
 & \qquad - \frac{1}{N^2} \sum_{\iota = -1,1} \int_{0}^{\mathfrak{a}} \int_{-\mathfrak{s}}^{ \mathfrak{s}} a^{-2} \int_{\R^2\setminus [0,N]^2} \left|\left\langle f_{[-R,N + R]}, \psi_{a,s,t, \iota}^{\omega}\right\rangle_{L^2(\R^2)}\right|^2 a^{-3} \, \,\mathrm{d}t \, \mathrm{d}s \, \mathrm{d}a. \label{eq:SimplificationToRealsContinuousCase}
\end{align}
We further compute that
\begin{align*}
& \ \frac{1}{N^2} \sum_{\iota = -1,1} \int_{0}^{\mathfrak{a}} \int_{-\mathfrak{s}}^{ \mathfrak{s}} a^{-2} \int_{\R^2\setminus [0,N]^2} \left|\left\langle f_{[-R,N + R]}, \psi_{a,s,t, \iota}^{\omega}\right\rangle_{L^2(\R^2)}\right|^2 a^{-3} \, \,\mathrm{d}t \, \mathrm{d}s \, \mathrm{d}a\\
&\quad  = \ \frac{1}{N^2} \sum_{\iota = -1,1} \int_{0}^{\mathfrak{a}} \int_{-\mathfrak{s}}^{ \mathfrak{s}} a^{-2} \int_{\R^2\setminus [0,N]^2} \left|\left\langle f_{[-R,N + R]} - f_{[R,N - R]}, \psi_{a,s,t, \iota}^{\omega}\right\rangle_{L^2(\R^2)}\right|^2 a^{-3} \, \,\mathrm{d}t \, \mathrm{d}s \, \mathrm{d}a\\
 &\quad \leq \ \frac{1}{N^2} \sum_{\iota = -1,1} \int_{0}^{\mathfrak{a}} \int_{-\mathfrak{s}}^{ \mathfrak{s}} a^{-2} \int_{\R^2} \left|\left\langle f_{[-R,N + R]} - f_{[R,N - R]}, \psi_{a,s,t, \iota}^{\omega}\right\rangle_{L^2(\R^2)}\right|^2 a^{-3} \, \,\mathrm{d}t \, \mathrm{d}s \, \mathrm{d}a\\
 &\quad \lesssim  \  \frac{1}{N^2} \mathcal{O}\left( |f_{[-R,N + R]} - f_{[R,N - R]}|_{H^1(\R^2)}^2\right) \lesssim \ \frac{1}{N} \mathcal{O}\left( |f|_{H^1((0,1)^2)}^2\right),
\end{align*}
where the last inequality follows by equation \eqref{eq:R2EquivalenceShearlets} and equation \eqref{eq:TheEstimateOnTheDifferenceOfTwoPeriodicExtensions}.
Denoting $c_\varepsilon \coloneqq T(\varepsilon)$, we compute similarly to the continuous case
\begin{align}
 |f|_{\DBp, c_\varepsilon}^2 = & \ c_\varepsilon^4\sum_{\iota = -1,1} \sum_{j \in J_{c_\varepsilon}} \sum_{k \in K_{{c_\varepsilon},j}} \sum_{\substack{m \in {c_\varepsilon} A_{\iota, 2^{-j}}\Z^2,\\ m \in [0,N]^2}} 2^{2j} \left|\left\langle f, \psi_{2^{-j},2^{-j/2} k, m, \iota}^{\omega, per} \right\rangle_{L^2(\R^2)} \right|^2\nonumber \\
  = & \ \frac{c_\varepsilon^4}{N^2}\sum_{\iota = -1,1} \sum_{j \in J_{c_\varepsilon}} \sum_{k \in K_{{c_\varepsilon},j}} \sum_{m \in {c_\varepsilon} A_{\iota, 2^{-j}} \Z^2} 2^{2j} \left|\left\langle f_{[-R,N + R]}, \psi_{2^{-j},2^{-j/2} k,m, \iota}^{\omega} \right \rangle_{L^2(\R^2)} \right|^2\nonumber \\
  &\qquad -
  \frac{c_\varepsilon^4}{N^2} \sum_{\iota = -1,1} \sum_{j \in J_{c_\varepsilon}} \sum_{k \in K_{{c_\varepsilon},j}} \sum_{\substack{m \in {c_\varepsilon} A_{\iota, 2^{-j}} \Z^2 \\ m \not\in [0,N]^2}} 2^{2j} \left|\left\langle f_{[-R,N + R]}, \psi_{2^{-j},2^{-j/2} k, m, \iota}^{\omega} \right \rangle_{L^2(\R^2)} \right|^2. \label{eq:SimplificationToRealsDiscreteCase}
\end{align}
By assumption, we have that, for $j \in J_{c_\varepsilon}$, $k \in K_{{c_\varepsilon},j}$, and $m \in {c_\varepsilon} A_{\iota, 2^{-j}}$, we have that $\suppp \psi_{2^{-j},2^{-j/2} k, m, \iota}^{\omega} \subset B_{R}(m)$. Thus we have that
\begin{align*}
&\ \frac{c_\varepsilon^4}{N^2} \sum_{\iota = -1,1} \sum_{j \in J_{c_\varepsilon}} \sum_{k \in K_{{c_\varepsilon},j}} \sum_{\substack{m \in {c_\varepsilon} A_{\iota, 2^{-j}}\Z^2,\\ m \not\in [0,N]^2}} 2^{2j} \left|\left\langle f_{[-R,N + R]}, \psi_{2^{-j},2^{-j/2} k, m, \iota}^{\omega} \right \rangle_{L^2(\R^2)} \right|^2 \\
&\quad =  \ \frac{c_\varepsilon^4}{N^2} \sum_{\iota = -1,1} \sum_{j \in J_{c_\varepsilon}} \sum_{k \in K_{{c_\varepsilon},j}} \sum_{\substack{m \in {c_\varepsilon} A_{\iota, 2^{-j}}\Z^2,\\ m \not\in [0,N]^2}} 2^{2j} \left|\left\langle f_{[-R,N + R]} - f_{[R,N - R]}, \psi_{2^{-j},2^{-j/2} k, m, \iota}^{\omega} \right \rangle_{L^2(\R^2)} \right|^2\\
&\quad  \leq \ \frac{c_\varepsilon^4}{N^2} \sum_{\iota = -1,1} \sum_{j \in J_{c_\varepsilon}} \sum_{k \in K_{{c_\varepsilon},j}} \sum_{m \in {c_\varepsilon} A_{\iota, 2^{-j}}\Z^2} 2^{2j} \left|\left\langle f_{[-R,N + R]} - f_{[R,N - R]}, \psi_{2^{-j},2^{-j/2} k, m, \iota}^{\omega} \right \rangle_{L^2(\R^2)} \right|^2 \eqqcolon \mathrm{I}.
\end{align*}
We define $i_\iota \coloneqq 1$ if $\iota = 1$ and $i_\iota \coloneqq 2$ if $\iota = -1$ and $\gamma' \coloneqq \psi$ and obtain by partial integration that
\begin{align*}
\mathrm{I} = & \ \frac{c_\varepsilon^4}{N^2} \sum_{\iota = -1,1} \sum_{j \in J_{c_\varepsilon}} \sum_{k \in K_{{c_\varepsilon},j}} \sum_{m \in {c_\varepsilon} A_{\iota, 2^{-j}}\Z^2}\left|\left\langle \frac{\partial }{\partial x_{i_\iota}}\left(f_{[-R,N + R]} - f_{[R,N - R]}\right), \gamma_{2^{-j},2^{-j/2} k, S_{2^{-j/2}k} A_{2^{-j}} m, \iota}^{\omega} \right \rangle_{L^2(\R^2)} \right|^2\\
= & \ \frac{1}{N^2}\mathcal{O}\left(|f_{[-R,N + R]} - f_{[R,N - R]}|_{H^1(\R^2)}^2\right) = \frac{1}{N}\mathcal{O}\left(|f|_{H^1((0,1)^2)}^2\right),
\end{align*}
where the second to last step follows from a long but simple computation which, with minor changes,
is performed in \cite[Section 5.1.1]{KGLConstrCmptShear2012} and the last inequality follows with \eqref{eq:TheEstimateOnTheDifferenceOfTwoPeriodicExtensions}.
Let now $(u_\varepsilon)_{\varepsilon>0} \subset H^1_0((0,1)^2)$. We then define
$$
N(u_\varepsilon ) \coloneqq \left\lceil \max\left\{ |u_\varepsilon|_{H^1((0,1)^2)}^3 , \frac{1}{\varepsilon}\right\}\right\rceil.
$$
It is not hard to see that $| (u_\varepsilon)_{[-R, N(u_\varepsilon )+R]} / N(u_\varepsilon )|_{H^1(\R^2)} \sim |u_\varepsilon|_{H^1((0,1)^2)}$ and by \eqref{eq:SimplificationToRealsContinuousCase} and \eqref{eq:SimplificationToRealsDiscreteCase} we have that, for $\varepsilon \to 0$,
\begin{align*}
\left||u_\varepsilon|_{\Bp}^2 - \ \sum_{\iota = -1,1} \int_{0}^{\mathfrak{a}} \int_{-\mathfrak{s}}^{ \mathfrak{s}} a^{-2} \int_{\R^2} \left|\left\langle \frac{(u_\varepsilon)_{[-R,N + R]}}{N(u_\varepsilon)^2} , \psi_{a,s,t, \iota}^{\omega}\right\rangle_{L^2(\R^2)}\right|^2 a^{-3} \, \,\mathrm{d}t \, \mathrm{d}s \, \mathrm{d}a\right| \to 0
\end{align*}
and
\begin{align*}
\left||u_\varepsilon|_{\DBp, {c_\varepsilon}}^2 - \ c_\varepsilon^4\sum_{\iota = -1,1} \sum_{j \in J_{c_\varepsilon}} \sum_{k \in K_{{c_\varepsilon},j}} \sum_{m \in {c_\varepsilon} \Z^2} 2^{2j} \left|\left\langle \frac{(u_\varepsilon)_{[-R,N + R]}}{N(u_\varepsilon)} , \psi_{2^{-j},2^{-j/2} k,  m, \iota}^{\omega} \right \rangle_{L^2(\R^2)} \right|^2 \right|\to 0.
\end{align*}
Thus, the result follows if for all $(h_\varepsilon)_{\varepsilon>0} \subset H^1(\R^2)$ with $|h_\varepsilon|_{H^1(\R^2)} = \mathcal{O}(T(\varepsilon)^{-1})$, for $\varepsilon \to 0$, one has that
\begin{align}
& \ \left|\int_{\R^2}\int_{-\mathfrak{s}}^{\mathfrak{s}} \int_{0}^{\mathfrak{a}} a^{-2} |\langle h_\varepsilon, \psi_{a,s,t, \iota}^{\omega} \rangle_{L^2(\R^2)} |^2 a^{-3} \, \mathrm{d}a \, \mathrm{d}s \, \mathrm{d}t\right. \nonumber\\
&\qquad - \left.c_\varepsilon^4\sum_{\iota = -1,1} \sum_{j \in J_{c_\varepsilon}} \sum_{k \in K_{{c_\varepsilon},j}} \sum_{m \in {c_\varepsilon} A_{\iota, 2^{-j}}\Z^2} 2^{2j} \left|\left\langle h_\varepsilon , \psi_{2^{-j},2^{-j/2} k, m, \iota}^{\omega} \right \rangle_{L^2(\R^2)} \right|^2\right| \to 0 \quad \text{ for } \varepsilon \to 0, \label{eq:AuxStatement1}
\end{align}
and for all $(h_\varepsilon)_{\varepsilon>0} \subset H^1(\R^2)$ with $\sup_{\varepsilon>0}\|u_\varepsilon\|_{L^2(\R^2)}<\infty$ and $\liminf_{\varepsilon>0}\varepsilon|h_\varepsilon|_{H^1(\R^2)}^2  = \infty$:
\begin{align}
\liminf_{\varepsilon>0} \varepsilon \cdot c_\varepsilon^4 \sum_{\iota = -1,1} \sum_{j \in J_{c_\varepsilon}} \sum_{k \in K_{{c_\varepsilon},j}} \sum_{m \in {c_\varepsilon} A_{\iota, 2^{-j}}\Z^2} 2^{2j} \left|\left\langle h_\varepsilon , \psi_{2^{-j},2^{-j/2} k,  m, \iota}^{\omega} \right \rangle_{L^2(\R^2)} \right|^2  = \infty.
\label{eq:AuxStatement2}
\end{align}
We will now verify the statements \eqref{eq:AuxStatement1} and \eqref{eq:AuxStatement2}. Let $h \in H^1(\R^2)$; then, by partial integration, we compute that
\begin{align}
&\sum_{\iota = -1,1} \int_{\R^2}\int_{-\mathfrak{s}}^{\mathfrak{s}} \int_{0}^{\mathfrak{a}} a^{-2} |\langle h, \psi_{a,s,t, \iota}^{\omega} \rangle_{L^2(\R^2)} |^2 a^{-3} \, \mathrm{d}a \, \mathrm{d}s \, \mathrm{d}t\nonumber \\
&= \sum_{\iota = -1,1} \int_{\R^2}\int_{-\mathfrak{s}}^{\mathfrak{s}} \int_{0}^{\mathfrak{a}} \left|\left\langle \frac{\partial}{ \partial x_{i_\iota}} h, \gamma_{a,s,t, \iota}^{\omega} \right \rangle_{L^2(\R^2)} \right|^2 a^{-3} \, \mathrm{d}a \, \mathrm{d}s \, \mathrm{d}t. \label{eq:thisShouldBereWritten}
\end{align}
By the co-area formula, we can rewrite \eqref{eq:thisShouldBereWritten} as
\begin{align*}
&\sum_{\iota = -1,1} \int_{\R^2}\int_{-\mathfrak{s}}^{\mathfrak{s}} \int_{0}^{\mathfrak{a}} \left|\left\langle \frac{\partial}{ \partial x_{i_\iota}} h, \gamma_{a,s,t, \iota}^{\omega} \right \rangle_{L^2(\R^2)} \right|^2 a^{-3} \, \mathrm{d}a \, \mathrm{d}s \, \mathrm{d}t \nonumber\\
& = \sum_{\iota = -1,1} \int_{-\log_2(\mathfrak{a})}^{\infty}\int_{-2^{j/2}\mathfrak{s}}^{2^{j/2}\mathfrak{s}} \int_{\R^2} \left|\left\langle \frac{\partial}{ \partial x_{i_\iota}} h, \gamma_{2^{-j},2^{-j/2}k, m, \iota}^{\omega} \right \rangle_{L^2(\R^2)} \right|^2 2^{3 j /2} \, \mathrm{d}m \, \mathrm{d}k \, \mathrm{d}j \nonumber\\
& = \sum_{\iota = -1,1} \int_{-\log_2(\mathfrak{a})}^{\infty}\int_{-2^{j/2}\mathfrak{s}}^{2^{j/2}\mathfrak{s}} \int_{\R^2} \left|\left\langle  \mathcal{F}\left(\frac{\partial}{ \partial x_{i_\iota}}h\right), \mathcal{F}\left(\gamma_{2^{-j},2^{-j/2}k, m, \iota}^{\omega}\right) \right \rangle_{L^2(\R^2)} \right|^2 2^{3 j /2} \, \mathrm{d}m \, \mathrm{d}k \, \mathrm{d}j \nonumber\\
& = \sum_{\iota = -1,1} \int_{-\log_2(\mathfrak{a})}^{\infty}\int_{-2^{j/2}\mathfrak{s}}^{2^{j/2}\mathfrak{s}} \int_{\R^2}  \omega(2^{-j/2} k, \iota)^2 \left|\mathcal{F} \left(\frac{\partial}{ \partial x_{i_\iota}}h\right)(\xi)\right|^2 \left|\mathcal{F}\left(\widehat{\gamma}\left(A_{2^{-j}} S^{T}_{-2^{-j/2} k} \xi \right)  \right) \right|^2 \mathrm{d}\xi \, \mathrm{d}k\, \,\mathrm{d}j  \nonumber\\
& = \sum_{\iota = -1,1} \int_{\R^2} \left|\mathcal{F} \left(\frac{\partial}{ \partial x_{i_\iota}}h\right)(\xi)\right|^2 \int_{-\log_2(\mathfrak{a})}^{\infty}\int_{-2^{j/2}\mathfrak{s}}^{2^{j/2}\mathfrak{s}} \omega(2^{-j/2} k, \iota)^2 \left|\mathcal{F}\left(\widehat{\gamma}\left(A_{2^{-j}} S^{T}_{-2^{-j/2} k} \xi \right)  \right) \right|^2 \mathrm{d}k \, \mathrm{d}j\, \,\mathrm{d}\xi \nonumber\\
& \eqqcolon  \ S(h).
\end{align*}
We shall now perform a similar computation for the discrete shearlet-based Besov seminorm. We have by Plancherel's identity that, for $c>0$,
\begin{align}
    &\ {c}^4 \sum_{j \in {c} \Z} \sum_{ k \in K_{{c},j} } \sum_{m \in {c_\varepsilon} A_{\iota, 2^{-j}}\Z^2} 2^{2j} \left|\left\langle h, \psi_{2^{-j},2^{-j/2} k, m, \iota}^{\omega, per} \right\rangle_{L^2(\R^2)} \right|^2\nonumber \\
    & =  \ {c}^4 \sum_{j \in {c} \Z} \sum_{ k \in K_{{c},j} } \sum_{m \in {c_\varepsilon} A_{\iota, 2^{-j}}\Z^2}  \left|\left\langle \mathcal{F}\left(\frac{\partial}{\partial x_{i_\iota}} h\right), \mathcal{F}\left(\gamma_{2^{-j},2^{-j/2} k, m, \iota}^{\omega} \right) \right\rangle_{L^2(\R^2)} \right|^2. \label{eq:NowApplyTheVeryLongComputationFromGitta}
\end{align}
Following once again the computation from \cite[Section 5.1.1]{KGLConstrCmptShear2012} shows that \eqref{eq:NowApplyTheVeryLongComputationFromGitta} equals:
\begin{align*}
    &\ W_1^\iota(h, c) + W_2^\iota(h, c) \\
    &  \coloneqq \ {c}^2 \sum_{j \in {c} \Z} \sum_{ k \in K_{{c},j} } \int_{\R^2} \omega(2^{-j/2} k, \iota)^2 \left| \mathcal{F}\left(\frac{\partial}{\partial x_{i_\iota}} h\right)\right|^2 \left| \widehat{\gamma}\left(A_{2^{-j}} S_{-2^{j/2}k}^T \xi\right )\right|^2 \,\mathrm{d}\xi+ W_2^\iota\left(\frac{\partial}{\partial x_{i_\iota}} h, c \right),
\end{align*}
where for a function $\mathcal{R}:(0,1] \to (0,1]$:
\begin{align}\label{eq:EstimateWithR}
     \left| W_2^\iota\left(\frac{\partial}{\partial x_{i_\iota}} h, c \right)\right| \leq \mathcal{R}({c}) \left\| \frac{\partial}{\partial x_{i_\iota}} h \, \right\|_{L^2(\R^2)}^2 ,
\end{align}
and, for sufficiently small ${c>0}$,
$$
W_1^\iota(h, c) \geq \left\| |\xi_{{i_\iota}}| \cdot \hat{h}  \cdot \chi_{\mathcal{C}^\iota} \right\|_{L^2(\R^2)}^2.
$$
Here $\mathcal{C}^1\coloneqq \{x \colon |x_1|, |x_2| \geq 1, |x_1| \geq |x_2|\}$, and $\mathcal{C}^{-1}\coloneqq \{x \colon |x_1|, |x_2| \geq 1, |x_1| \leq |x_2|\}$. Hence, we conclude that
$$
W_1^\iota(h, c) + W_2^\iota(h, c) \gtrsim \left \| |\xi| \hat{h}_{|\mathcal{C}^1 \cup \mathcal{C}^{-1}}\right \|_{L^2(\R^2)}^2 - \mathcal{R}(c) |h|_{H^1(\R^2)}^2.
$$
We have that
$$
\left\| |\xi| \hat{h}_{|\mathcal{C}^1 \cup \mathcal{C}^{-1}}\right\|_{{L^2(\R^2)}}^2 \gtrsim |h|_{H^1(\R^2)}^2 - \|h \|_{{L^2(\R^2)}}^2.
$$
If we now insert $u_\varepsilon$ for $h$ we deduce that
\begin{align} \label{eq:ThisImpliesTheAuxStatement2}
\varepsilon \big( W_1^\iota(u_\varepsilon, c_\varepsilon) + W_2^\iota(u_\varepsilon, c_\varepsilon)\big) \gtrsim \varepsilon (1- \mathcal{R}(c_\varepsilon)) |u_\varepsilon|_{H^1(\R^2)}^2 - \varepsilon \|u_\varepsilon\|_{L^2(\R^2)}^2.
\end{align}
Applying \cite[Proposition 3.3]{KGLConstrCmptShear2012} to \cite[Equation (38)]{KGLConstrCmptShear2012} yields that $\mathcal{R}({c_\varepsilon}) \lesssim {c_\varepsilon}$ and thus
\eqref{eq:ThisImpliesTheAuxStatement2} implies \eqref{eq:AuxStatement2}.
To obtain \eqref{eq:AuxStatement1}, we now estimate with \eqref{eq:EstimateWithR}:
\begin{align} \label{eq:ARandomSplittingThatMightBeImportantLater}
\left|S(h) - \sum_{\iota = -1,1} \big(W_1^\iota(h, c) + W_2^\iota(h, c)\big)\right| \leq \left|S(h) - W_1^1(h, c) + W_1^{-1}(h, c)\right| + \mathcal{O}\left(\mathcal{R}(c) |h|_{H^1(\R^2)}^2\right).
\end{align}
We rewrite $W_1^\iota$ as a triple integral by using the rounding operators \eqref{eq:roundingOp1} and \eqref{eq:roundingOp2}:
\begin{align*}
    W_1^\iota(h, c) = &\ \int_{-\log_2(\mathfrak{a})}^\infty \int_{-2^{[j]_c/2} \mathfrak{s}}^{[2^{[j]_{c}/2} \mathfrak{s}]_{j,{c}} + {c}} \int_{\R^2} \omega\left(2^{-[j]_{c}/2} [k]_{j,{c}}, \iota\right)^2 \nonumber \\
    & \qquad \cdot \left| \mathcal{F}\left(\frac{\partial}{\partial x_{i_\iota}} h\right)\right|^2 \left| \widehat{\gamma}\left(A_{2^{-[j]_c}} S_{-2^{[j]_{c}/ 2} [k]_{j,{c}} }^T \xi\right )\right|^2 \,\mathrm{d}\xi \, \mathrm{d}k \, \mathrm{d}j. 
\end{align*}
We proceed by estimating $|S(h) - W_1^{1}(h, c) - W_1^{-1}(h, c)|$:
\begin{align}
    & |S(h) - W_1^{1}(h, c) - W_1^{-1}(h, c)|\nonumber
    \\
    & \leq \ \int_{\R^2} \left| \mathcal{F}\left(\frac{\partial}{\partial x_{i_\iota}} h\right)\right|^2 \int_{-\log_2(\mathfrak{a})}^\infty \int_{-2^{j/2} \mathfrak{s}}^{2^{j/2} \mathfrak{s}}  \left| \omega\left(2^{-[j]_c/2} [k]_{j,{c}}, \iota\right)^2 \left| \widehat{\gamma}\left(A_{2^{-[j]_{c}}} S_{-2^{[j]_{c}/ 2} [k]_{j,{c}} }^T \xi\right )\right|^2 \right. \nonumber\\
    & \quad  - \left. \omega\left(2^{-j/2} k, \iota\right)^2 \left| \widehat{\gamma}\left(A_{2^{-j}} S_{-2^{j/ 2} k }^T \xi\right )\right|^2\right| \, \mathrm{d}k \, \mathrm{d}j \, \mathrm{d}\xi \nonumber\\
    & \qquad +  \int_{\R^2} \left| \mathcal{F}\left(\frac{\partial}{\partial x_{i_\iota}} h\right)\right|^2 \int_{-\log_2(\mathfrak{a})}^\infty \int_{2^{j/2} \mathfrak{s}}^{[2^{[j]_{c}/2} \mathfrak{s}]_{j,{c}} + {c}} \omega\left(2^{-[j]_{c}/2} [k]_{j,{c}}, \iota\right)^2 \left| \widehat{\gamma}\left(A_{2^{-[j]_{c}}} S_{-2^{[j]_{c}/ 2} [k]_{j,{c}} }^T \xi\right )\right|^2 \, \mathrm{d}k \, \mathrm{d}j \, \mathrm{d}\xi \nonumber\\
    & \leq  \ |h|_{H^1(\R^2)}^2 \sup_{\xi \in \R^2} \left( \int_{-\log_2(\mathfrak{a})}^\infty \int_{-2^{j/2} \mathfrak{s}}^{2^{j/2} \mathfrak{s}}  \left| \omega\left(2^{-[j]_{c}/2} [k]_{j,{c}}, \iota\right)^2 \left| \widehat{\gamma}\left(A_{2^{-[j]_{c}}} S_{-2^{[j]_{c}/ 2} [k]_{j,{c}} }^T \xi\right )\right|^2 \right. \right. \nonumber \\
    & \quad - \left. \omega(2^{-j/2} k, \iota)^2 \left| \widehat{\gamma}\left(A_{2^{-j}} S_{-2^{j/ 2} k }^T \xi\right )\right|^2\right| \, \mathrm{d}k \, \mathrm{d}j \nonumber\\
    & \qquad + \left. \int_{-\log_2(\mathfrak{a})}^\infty \int_{2^{j/2} \mathfrak{s}}^{[2^{[j]_{c}/2} \mathfrak{s}]_{j,{c}} + {c}} \omega\left(2^{-[j]_{c}/2} [k]_{j,{c}}, \iota\right)^2 \left| \widehat{\gamma}\left(A_{2^{-[j]_{c}}} S_{-2^{[j]_{c}/ 2} [k]_{j,{c}} }^T \xi\right )\right|^2  \, \mathrm{d}k \, \mathrm{d}j \right) \nonumber\\
   &  \eqqcolon \ |h|_{H^1(\R^2)}^2 \text{ ess\!}\sup_{\xi \in \R^2} \big( U_1(\xi, c) + U_2(\xi, c)\big), \label{eq:WeNeedToPlugOurEstimatesHere}
\end{align}
where
\begin{align*}
    U_1(\xi, c) \coloneqq&\ \int_{-\log_2(\mathfrak{a})}^\infty \int_{-2^{j/2} \mathfrak{s}}^{2^{j/2} \mathfrak{s}}\left| \omega\left(2^{-[j]_{c}/2} [k]_{j,{c}}, \iota\right)^2 \left| \widehat{\gamma}\left(A_{2^{-[j]_{c}}} S_{-2^{[j]_{c}/ 2} [k]_{j,{c}} }^T \xi\right )\right|^2 \right.\\
    & \left.\qquad - \omega\left(2^{-j/2} k, \iota\right)^2 \left| \widehat{\gamma}\left(A_{2^{-j}} S_{-2^{j/ 2} k }^T \xi\right )\right|^2\right| \, \mathrm{d}k \, \mathrm{d}j \quad \text{ and }\\
    U_2(\xi, c) \coloneqq&\ \int_{-\log_2(\mathfrak{a})}^\infty \int_{2^{j/2} \mathfrak{s}}^{[2^{[j]_{c}/2} \mathfrak{s}]_{j,{c}} + {c}} \omega\left(2^{-[j]_{c}/2} [k]_{j,{c}}, \iota\right)^2 \left| \widehat{\gamma}\left(A_{2^{-[j]_{c}}} S_{-2^{[j]_{c}/ 2} [k]_{j,{c}} }^T \xi\right )\right|^2 \, \mathrm{d}k \, \mathrm{d}j.
\end{align*}
We estimate $U_1(\xi, c)$ and $U_2(\xi, c)$ individually and start with $U_2(\xi, c)$. We have that, for all $\xi \in \R^2$,
\begin{align}\label{eq:DecayEstimate}
|\widehat{\gamma}(\xi)| \lesssim \frac{\min\{|\xi_1|^M, 1\}}{(1+|\xi_1|^2)^\frac{L}{2} (1+|\xi_2|^2)^\frac{L}{2}}.
\end{align}
Hence,
$$
|U_2(\xi, c)| \lesssim \int_{-\log_2(\mathfrak{a})}^\infty \int_{2^{j/2} \mathfrak{s}}^{[2^{[j]_{c}/2} \mathfrak{s}]_{j,{c}} + {c}}  \frac{\min\{|2^{-j} \xi_1|^{2M}, 1\}}{(1+|2^{-j}\xi_1|^2)^L (1+|2^{-j/2}(\xi_2 - 2^{-[j]_{c}/ 2} [k]_{j,{c}} \xi_1|^2))^L}  \, \mathrm{d}k \, \mathrm{d}j.
$$
Since $[k]_{j,{c}} = {[2^{[j]_{c}/2} \mathfrak{s}]_{j,{c}}}$ on the domain of integration we have that, if $\xi_1 \neq 0$, then
\begin{align}
|U_2(\xi, c)| \lesssim & \ {c}  \int_{-\log_2(\mathfrak{a})}^\infty   \frac{\min\{|2^{-j} \xi_1|^{2M}, 1\}}{(1+|2^{-j}\xi_1|^2)^L (1+|2^{-j/2}(\xi_2 - 2^{-[j]_{c}/ 2}{[2^{[j]_{c}/2} \mathfrak{s}]_{j,{c}}} \xi_1|^2))^L}   \, \mathrm{d}j \nonumber\\
\leq &\ {c} \int_{-\log_2(\mathfrak{a})}^\infty  \frac{\min\{|2^{-j} \xi_1|^{2M}, 1\}}{(1+|2^{-j}\xi_1|^2)^L} \,\mathrm{d}j \nonumber \\
\leq & \ {c} \int_{-\infty}^\infty  \frac{\min\{|2^{-j + \log_2(|\xi_1|)}|^{2M}, 1\}}{(1+|2^{-j + \log_2(|\xi_1|)}|^2)^L} \, \mathrm{d}j = \ {c} \int_{-\infty}^\infty  \frac{\min\{|2^{-j}|^{2M}, 1\}}{(1+|2^{-j}|^2)^L} \,\mathrm{d}j
\lesssim {c}. \label{eq:EstimateU2}
\end{align}
We continue by estimating $U_1(\xi, c)$. We compute with the binomial formula that, for all $\xi \in \R^2$,
\begin{align}
  &U_1(\xi, c) \nonumber \\
  & = \ \int_{-\log_2(\mathfrak{a})}^\infty \int_{-2^{j/2} \mathfrak{s}}^{2^{j/2} \mathfrak{s}} \left| \omega\left(2^{-[j]_c/2} [k]_{j,{c}}, \iota\right)^2 \left| \widehat{\gamma}\left(A_{2^{-[j]_{c}}} S_{-2^{[j]_{c}/ 2} [k]_{j,{c}} }^T \xi\right )\right|^2 \right. \nonumber\\
    & \qquad - \left. \omega(2^{-j/2} k, \iota)^2 \left| \widehat{\gamma}\left(A_{2^{-j}} S_{-2^{j/ 2} k }^T \xi\right )\right|^2 \right| \, \mathrm{d}k \, \mathrm{d}j \nonumber\\
     & =  \ \int_{-\log_2(\mathfrak{a})}^\infty \int_{-2^{j/2} \mathfrak{s}}^{2^{j/2} \mathfrak{s}}  \left| \omega\left(2^{-[j]_{c}/2} [k]_{j,{c}}, \iota\right)\widehat{\gamma}\left(A_{2^{-[j]_{c}}} S_{-2^{[j]_{c}/ 2} [k]_{j,{c}} }^T \xi\right )
     - \  \omega(2^{-j/2} k, \iota) \widehat{\gamma}\left(A_{2^{-j}} S_{-2^{j/ 2} k }^T \xi\right )\right|\nonumber\\
     & \qquad  \cdot \left| \omega\left(2^{-[j]_{c}/2} [k]_{j,{c}}, \iota\right) \widehat{\gamma}\left(A_{2^{-[j]_{c}}} S_{-2^{[j]_{c}/ 2} [k]_{j,{c}} }^T \xi\right )
     + \  \omega(2^{-j/2} k, \iota) \widehat{\gamma}\left(A_{2^{-j}} S_{-2^{j/ 2} k }^T \xi\right )\right| \, \mathrm{d}k \, \mathrm{d}j \label{eq:Step1TowardsU1Estimate}.
\end{align}
Using the decay estimate \eqref{eq:DecayEstimate}, we obtain that if $\iota = 1$, then, for all $\xi \in \R^2$,
\begin{align}
& \left| \omega\left(2^{-[j]_{c}/2} [k]_{j,{c}}, \iota\right) \widehat{\gamma}\left(A_{2^{-[j]_{c}}} S_{-2^{[j]_{c}/ 2} [k]_{j,{c}} }^T \xi\right )
     + \  \omega(2^{-j/2} k, \iota) \widehat{\gamma}\left(A_{2^{-j}} S_{-2^{j/ 2} k }^T \xi\right )\right|\nonumber \\
     &\quad \lesssim  \ \frac{\min\{|2^{-j} \xi_1|, 1\}^M}{(1+|2^{-j} \xi_1|^2)^\frac{L}{2} (1+|2^{-j/2} (\xi_2 + 2^{-j/2} k \xi_1)|^2)^\frac{L}{2}}. \label{eq:Step2TowardsU1Estimate}
\end{align}
If $\iota = -1$, then the roles of $\xi_1$ and $\xi_2$ are reversed.
Additionally, by the Lipschitz continuity of $\omega$, $\widehat{\gamma}$ and the Lipschitz continuity of $x \mapsto 2^{x}$ and $x \mapsto 2^{x/2}$ for $x$ in a neighbourhood of $0$, we conclude that, for all $\xi \in \R^2$,
\begin{align}
 &\left| \omega\left(2^{-[j]_{c}/2} [k]_{j,{c}}, \iota\right)\widehat{\gamma}\left(A_{2^{-[j]_{c}}} S_{-2^{[j]_{c}/ 2} [k]_{j,{c}} }^T \xi\right )
     - \  \omega(2^{-j/2} k, \iota) \widehat{\gamma}\left(A_{2^{-j}} S_{-2^{j/ 2} k }^T \xi\right )\right| \nonumber \\
     &\lesssim \ \left|A_{2^{-j}} S_{-2^{j/ 2} k }^T \xi - A_{2^{-[j]_{c}}} S_{-2^{[j]_{c}/ 2} [k]_{j,{c}} }^T \xi\right|  \nonumber\\
     & \leq \ \left| \left(2^{-j}-2^{-[j]_{c}}\right) \xi_1, \left(2^{-j/2}-2^{-[j]_{c}/2}\right) \xi_2 - \left(2^{-j}k -2^{-[j]_{c}} [k]_{j,{c}}  \right)\xi_1\right|  \nonumber\\
      & \leq  \ \left| \left({c} 2^{-j} \xi_1, \left(2^{-j/2}-2^{-[j]_{c}/2}\right) \xi_2 - \left(2^{-j}k - 2^{-[j]_{c}/2}2^{-j/2} k\right)\xi_1 - \left(2^{-[j]_{c}/2}2^{-j/2} k - 2^{-[j]_{c}} [k]_{j,{c}}\right) \xi_1 \right)\right|  \nonumber\\
      & \leq  \ \left| \left({c} 2^{-j} \xi_1, \left(2^{-j/2}-2^{-[j]_{c}/2}\right) \left(\xi_2 - 2^{-j/2} k \xi_1\right) - \left(2^{-[j]_{c}/2}2^{-j/2} k - 2^{-[j]_{c}} [k]_{j,{c}}\right) \xi_1\right)\right|  \nonumber\\
      & \leq  \ {c}\left| 2^{-j} \xi_1, \left(2^{-j/2}\right) \left(\xi_2 - 2^{-j/2} k \xi_1\right)\right| + {c} \left|2^{-[j]_{c}} \xi_1\right|  \nonumber\\
      & \lesssim  \ {c}\left(1+\left|2^{-j} \xi_1\right|^2\right)^{\frac{1}{2}} \left(1+\left|2^{-j/2} \left(\xi_2 + 2^{-j/2} k \xi_1\right)\right|^2\right)^{\frac{1}{2}}.   \label{eq:Step3TowardsU1Estimate}
\end{align}
Combining \eqref{eq:Step1TowardsU1Estimate}, \eqref{eq:Step2TowardsU1Estimate} and \eqref{eq:Step3TowardsU1Estimate} yields that
\begin{align}
 U_1(\xi, c) \lesssim {c} \int_{-\log_2(\mathfrak{a})}^\infty \int_{-2^{j/2} \mathfrak{s}}^{2^{j/2} \mathfrak{s}}\ \frac{\min\{|2^{-j} \xi_1|, 1\}^M}{(1+|2^{-j} \xi_1|^2)^{\frac{L}{2}-1} (1+|2^{-j/2} (\xi_2 + 2^{-j/2} k \xi_1)|^2)^{\frac{L}{2}-1}} \, \mathrm{d}k \, \mathrm{d}j \lesssim {c}. \label{eq:EstimateU1}
\end{align}
Additionally, applying \eqref{eq:EstimateU1}, \eqref{eq:EstimateU2} to \eqref{eq:WeNeedToPlugOurEstimatesHere} and \eqref{eq:ARandomSplittingThatMightBeImportantLater} shows that
$$
\left|S(h) - \sum_{\iota = -1,1} \big(W_1^\iota(h, c) + W_2^\iota(h, c)\big)\right|  = \mathcal{O} \left(\left(c + R(c)\right)|h|_{H^1(\R^2)}^2\right) = \mathcal{O} \left(c|h|_{H^1(\R^2)}^2\right),
$$
which after replacing $h$ by $u_\varepsilon$ and $c$ by $c_\varepsilon$ implies \eqref{eq:AuxStatement1}. This completes the proof.

\end{proof}

\section*{Acknowledgements}

Philipp Petersen is supported by a DFG Research Fellowship "Shearlet-based energy functionals for anisotropic phase-field methods". He is grateful to the Mathematical Institute at the University
of Oxford for the hospitality and support during his visit.

\appendix

\section{Equivalence of the shearlet-based Besov seminorm and the $H^1$ seminorm}\label{app:H1Frame}

We first compute that
\begin{align}\label{eq:FTrafoOfShearlet}
\mathcal{F}(\psi_{a,s,0})(\xi) = a^{\frac34} \widehat{\psi}\left( a \xi_1, \sqrt{a}(\xi_2 + s \xi_1)\right) \quad \text{ for all } \xi \in \R^2.
\end{align}
Now let $f \in L^2(\R^2)$, $\mathfrak{a}>0$ and $\mathfrak{s}>0$; then, by Parseval's identity, the convolution theorem, and \eqref{eq:FTrafoOfShearlet} we have that
\begin{align*}
&\int_{\R^2} |\langle f, K(\cdot -t) \rangle|^2 \, \,\mathrm{d}t + \int_{\R^2}\int_{-\mathfrak{s}}^{\mathfrak{s}} \int_{0}^{\mathfrak{a}} a^{-2}|\langle f, \psi_{a,s,t} \rangle |^2 a^{-3} \, \mathrm{d}a \, \mathrm{d}s \, \mathrm{d}t\\
&\qquad  + \int_{\R^2}\int_{-\mathfrak{s}}^{\mathfrak{s}} \int_{0}^{\mathfrak{a}} a^{-2} \left|\left\langle f, \widetilde{\psi}_{a,s,t} \right\rangle \right|^2 a^{-3} \, \mathrm{d}a \, \mathrm{d}s \, \mathrm{d}t\\
& =  \int_{\R^2} |\hat{f}(\xi)|^2 \left( |\widehat{K}(\xi)|^2 + \int_{-\mathfrak{s}}^{\mathfrak{s}} \int_{0}^{\mathfrak{a}} \left|\widehat{\psi}(a \xi_1, \sqrt{a}(\xi_2 + s\xi_1) )\right|^2  a^{-\frac72} \, \mathrm{d}a \, \mathrm{d}s \right.\\
&\qquad  + \left. \int_{\R^2}\int_{-\mathfrak{s}}^{\mathfrak{s}} \int_{0}^{\mathfrak{a}} \left|\widehat{\widetilde{\psi}}(\sqrt{a}(\xi_1 + s\xi_2) , a \xi_2)\right|^2  a^{-\frac72} \, \mathrm{d}a \, \mathrm{d}s \right) \,\mathrm{d}\xi \eqqcolon \mathrm{I}(f).
\end{align*}
Let, for $\xi \in \R^2$, $\widehat{\mu}(\xi) \coloneqq \widehat{\psi}(\xi)/\xi_1$ and $\widehat{\widetilde{\mu}}(\xi) \coloneqq \widehat{\widetilde{\psi}}(\xi)/\xi_2$; then,
\begin{align}\label{eq:FreqRep}
\mathrm{I}(f) =  &\int_{\R^2} |\hat{f}(\xi)|^2 \left( |\widehat{K}(\xi)|^2 + |\xi_1|^2 \cdot \int_{-\mathfrak{s}}^{\mathfrak{s}} \int_{0}^{\mathfrak{a}} \left|\widehat{\mu}(a \xi_1, \sqrt{a}(\xi_2 + s\xi_1) )\right|^2  a^{-\frac32} \, \mathrm{d}a \, \mathrm{d}s \right.\\
&\qquad  + \left. |\xi_2|^2 \cdot  \int_{\R^2}\int_{-\mathfrak{s}}^{\mathfrak{s}} \int_{0}^{\mathfrak{a}} \left|\widehat{\widetilde{\mu}}(\sqrt{a}(\xi_1 + s\xi_2) , a \xi_2)\right|^2  a^{-\frac32} \, \mathrm{d}a \, \mathrm{d}s \right) \,\mathrm{d}\xi.\nonumber
\end{align}
It was shown in \cite[Theorem 4.3]{grohs2011continuous} and \cite[Lemma 4.2]{grohs2011continuous} that there exist
 $\mathfrak{a}^*>0$ and $\mathfrak{s}^*>0$ such that for all $\mathfrak{a} > \mathfrak{a}^*, \mathfrak{s}> \mathfrak{s}^*$ there exist $A, B\in (0,\infty)$ such that
\begin{align}
A \leq &\int_{-\mathfrak{s}}^{\mathfrak{s}} \int_{0}^{\mathfrak{a}} \left|\widehat{\mu}(\sqrt{a}(\xi_1 + s\xi_2) , a \xi_2)\right|^2  a^{-\frac32} \, \mathrm{d}a \, \mathrm{d}s, \quad\text{ for all } \xi \text{ such that }|\xi_{1}| \geq |\xi_2| \text{ and } |\xi_1|\geq 1; \label{eq:FirstOccuranceOfA}\\
A \leq &\int_{-\mathfrak{s}}^{\mathfrak{s}} \int_{0}^{\mathfrak{a}} \left|\widehat{\widetilde{\mu}}(\sqrt{a}(\xi_1 + s\xi_2) , a \xi_2)\right|^2  a^{-\frac32} \, \mathrm{d}a \, \mathrm{d}s, \quad \text{ for all } \xi \text{ such that } |\xi_2| \geq |\xi_1| \text{ and } |\xi_2|\geq 1;\label{eq:SecondOccuranceOfA}\\
&\int_{-\mathfrak{s}}^{\mathfrak{s}} \int_{0}^{\mathfrak{a}} \left|\widehat{\mu}(\sqrt{a}(\xi_1 + s\xi_2) , a \xi_2) \right |^2  a^{-\frac32} \, \mathrm{d}a \, \mathrm{d}s \leq B,  \quad \text{ for all } \xi \in \R^2; \label{eq:FirstOccuranceOfB}\\
&\int_{-\mathfrak{s}}^{\mathfrak{s}} \int_{0}^{\mathfrak{a}} \left|\widehat{\widetilde{\mu}}(\sqrt{a}(\xi_1 + s\xi_2) , a \xi_2)\right|^2  a^{-\frac32} \, \mathrm{d}a \, \mathrm{d}s \leq B,  \quad \text{ for all } \xi \in \R^2. \label{eq:SecondOccuranceOfB}
\end{align}
By \cite[Equation 5]{grohs2011continuous}, we observe that $B$ in \eqref{eq:FirstOccuranceOfB} and \eqref{eq:SecondOccuranceOfB} can be chosen independently from $\mathfrak{a}$ and $\mathfrak{s}$. Moreover, since increasing $\mathfrak{a}$ and $\mathfrak{s}$ only increases the right-hand sides of \eqref{eq:FirstOccuranceOfA} and \eqref{eq:SecondOccuranceOfA}, we conclude that $A,B$ are independent of $\mathfrak{s}, \mathfrak{a}$.

We have that, if $|\xi_{1}| \geq |\xi_2|$, then $|\xi|^2/2 \leq |\xi_1|^2 \leq |\xi|^2 $ and if $|\xi_{2}| \geq |\xi_1|$, then $|\xi|^2/2 \leq |\xi_2|^2 \leq |\xi|^2$. Additionally,
by the assumptions of Theorem \ref{thm:EmbeddingH1R2}, we have that there exist $0<\widetilde{A}< \widetilde{B}$ such that $|K(\xi)|/|\xi|\in [\widetilde{A}, \widetilde{B}]$ for all $\xi \in [-1,1]^2$. Hence, for $\mathfrak{a} > \mathfrak{a}^*$ and $\mathfrak{s}> \mathfrak{s}^*$ we have that
\begin{align*}
 \mathrm{I}(f) = \int_{\R^2} |\hat{f}(\xi)|^2 |S(\xi)|^2 \,\mathrm{d}\xi,
\end{align*}
where $S: \R^2 \to \R$ is such that
$$
\min\left\{\widetilde{A}, \frac{A}{2}\right\} |\xi|^2 \leq |S(\xi)|^2 \leq \max \left\{ B, \widetilde{B} \right\}|\xi|^2 \quad \text{ for all } \xi \in \R^2.
$$
Since $|f|_{H^1(\R^2)} \sim \int_{\R^2} |\xi|^2 |\hat{f}(\xi)|^2 \,\mathrm{d}\xi$, this shows that $|f|_{H^1(\R^2)} \sim I(f)$ for all $f \in H^1(\R^2)$, where the implicit constants depend on $\mathfrak{s}^*$ and $\mathfrak{a}^*$, but not on $\mathfrak{s}$, and $\mathfrak{a}$.

\section{Canonical choices of associated directional weights}\label{app:canonicalWeight}

Let $\mathfrak{s} > \mathfrak{s}^* \geq  1$ and let $\mathcal{N}$ be a norm on $\R^2$. Below we construct an directional weight associated with $\mathcal{N}$. We define $t_\mathfrak{s}: \R \to \R$ to be a continuously differentiable cubic spline with nodes
$\{-\mathfrak{s}, -{1}/{\mathfrak{s}}, 1/\mathfrak{s}, \mathfrak{s}\}$, satisfying $t_\mathfrak{s} \equiv 1$ on $[-{1}/{\mathfrak{s}}, {1}/{\mathfrak{s}}]$ and $t_\mathfrak{s} \equiv 0$ on $[-\mathfrak{s}, \mathfrak{s}]^c$. 

Next we set, for $x \in \R$, 
\begin{align*}
\omega(1,x) &\coloneqq \left(\frac{|\cos(\arctan(x))|^{-1}}{\max \{|\cos(\arctan(x))| , |\sin(\arctan(x))| \}} \cdot t_\mathfrak{s}(x)\right)^{\frac{1}{2}} \cdot \mathcal{N}\left(\binom {\cos(\arctan(x))} {\sin(\arctan(x))}\right)
\end{align*}
Moreover, we set, for $x \neq 0$,
\begin{align*}
\omega(-1,x) \coloneqq \left(\frac{|\sin(\arctan(x))|^{-1}}{\max \{|\cos(\arctan(x^{-1}))| , |\sin(\arctan(x^{-1}))| \}} \cdot \left( 1- t_\mathfrak{s}\left(x^{-1}\right)\right)\right)^{\frac{1}{2}} \cdot \mathcal{N}\left(\binom {\cos(\arctan(x^{-1}))} {\sin(\arctan(x^{-1}))}\right),
\end{align*}
and $\omega(-1,0) \coloneqq 0$.
We have by construction that $t_\mathfrak{s}^{1/2}$ and $\mathcal{N}$ are Lipschitz continuous. Moreover, $\max \{|\cos(\arctan(x))| , |\sin(\arctan(x))| \}$ is bounded below by 
$$
\min_{x \in [-\pi/2, \pi/2]} \max\{|\cos(x)| , |\sin(x)| \} = 1/\sqrt{2}.
$$
Additionally, we have that 
$\suppp t_\mathfrak{s} \subset [-\mathfrak{s}, \mathfrak{s}]$ and hence 
$$
\cos(\arctan(x))> c 
$$
for $x \in \suppp t_\mathfrak{s}$ and a constant $c>0$. Similarly, $\suppp (1-t_\mathfrak{s})\subset [-1/\mathfrak{s}, 1/\mathfrak{s}]^c$ and hence 
$$
\sin(\arctan(x^{-1})) > c' 
$$
for a constant $c'>0$. We conclude that $\omega(\iota,\cdot)$ is Lipschitz continuous for $\iota = \pm 1$. 
Furthermore, we have that there exists a $c''>0$ such that, for all $x \in \R$, $\mathcal{N}(x) > c''$.

Therefore, 
\begin{align*}
\suppp\omega(1,x) &= \suppp t_\mathfrak{s}(x) \subset [-\mathfrak{s}, \mathfrak{s}],\\
\suppp\omega(-1,x) &= \suppp \left(1 - t_\mathfrak{s}(x^{-1})\right) \subset [-\mathfrak{s}, \mathfrak{s}],\\
\min_{x\in [-\mathfrak{s}^*,\mathfrak{s}^*]} \omega(1,x) &> \frac{c''}{2^{1/4}} \sqrt{c} \min_{x\in [-\mathfrak{s}^*,\mathfrak{s}^*]} \sqrt{t_\mathfrak{s}(x)} > 0,\\
\min_{x\in [-\mathfrak{s}^*,\mathfrak{s}^*]} \omega(-1,x) &> \frac{c''}{2^{1/4}} \sqrt{c'} \min_{x\in [-\mathfrak{s}^*,\mathfrak{s}^*]} \sqrt{\left(1 - t_\mathfrak{s}\left(x^{-1}\right)\right)} > 0.
\end{align*}
We conclude that $\omega$ is a directional weight according to Definition \ref{def:directionalWeight}. Next, we analyse if it is also associated with the norm $\mathcal{N}$ in the sense of \eqref{eq:associatedNormDefinition}. This is essentially clear by construction. Indeed, we have that, for all $(\eta_1,\eta_2) \in \mathbb{S}^1$,
$$
\binom{\cos(\arctan(\eta_2/ \eta_1))}{\sin(\arctan(\eta_2/ \eta_1))} = \pm \binom{\eta_1}{\eta_2}.
$$ 
Therefore, for all $(\eta_1,\eta_2) \in \mathbb{S}^1$,
\begin{align*}
&\max\left\{|\vn_1|, |\vn_2|\right\}\left(|\vn_1| \omega\left(1, \frac{\vn_2}{\vn_1}\right)^2 + |\vn_2|\omega\left(-1, \frac{\vn_1}{\vn_2}\right)^2 \right)\\
= \ &t_\mathfrak{s}\left(\frac{\vn_2}{\vn_1}\right) \cdot \mathcal{N}\left(\pm \binom { \eta_1} {\eta_2}\right)^2 + \left( 1- t_\mathfrak{s}\left(\frac{\vn_2}{\vn_1}\right)\right) \cdot \mathcal{N}\left(\pm \binom{\eta_1}{\eta_2}\right)^2\\
 = \ & \mathcal{N}\left(\pm \binom{\eta_1}{\eta_2}\right)^2 = \mathcal{N}\left(\binom{\eta_1}{\eta_2}\right)^2.
\end{align*}
Hence, we have that the weight $\omega$ is a directional weight associated with $\mathcal{N}$.

\section{Proof of Proposition \ref{prop:covering}}\label{app:ProofOfCovering}

\begin{proof}
Let $\alpha$ be the smallest interior or exterior angle of $P$ and define $c_0 \coloneqq \min\{ \sin(\alpha), 1\}/2$. We denote, for $i = 1,\dots, N$, by $g_i$ the line segment with start point $x_i$ and end point $x_{i+1}$.
Let $r_1 \coloneqq \min \{ \mathfrak{d}(g_i, g_j): \ i,j \in \{1,\dots, N\}, \ i>j, \ i-j >1, \  (i,j)\neq (N,1)\}$ be the smallest distance between each line segment and the closest non-neighbouring line segment. Since $\partial P$ is non-self-intersecting, we have that $r_1>0$.
Let $r_2 \coloneqq \min_{i = 1,\dots, N} |x_i-x_{i+1}|$, where $x_{N+1} = x_1$. Again, we have that $r_2>0$ since otherwise two vertices would coincide. We set $r_0 \coloneqq \min\{r_1, r_2\}/8$.

Let $r < r_0$ and let $i \in \{1,\dots, N\}$; then, by the Pythagorean theorem, there is a coordinate $t \in \{1,2\}$ such that 
\begin{align}\label{eq:EstimateForT}
\kappa_i\coloneqq |(x_i)_t-(x_{i+1})_t| \geq |x_i - x_{i+1}|/\sqrt{2} \geq |x_i - x_{i+1}|/4 + r_2/4 \geq |x_i - x_{i+1}|/4 + 2 r.
\end{align}
 Define $L_i \coloneqq \lceil(\kappa_i- 2r)/(r c_0)\rceil$ and $s_i \coloneqq (\kappa_i-2r) / (2 L_i)$. We have that $L_i \geq  |x_i - x_{i+1}|/(4r)$ by equation \eqref{eq:EstimateForT}. Moreover, we have that
\begin{align*}
 s_i = \frac{\kappa_i-2r}{2 \left\lceil(\kappa_i- 2r)/(c_0 r )\right\rceil} \leq c_0 r \frac{\kappa_i-2r}{2 (\kappa_i- 2r)} = \frac{c_0 r}{2}
\end{align*}
and
\begin{align*}
 s_i \geq \frac{\kappa_i-2r}{(2 + 2(\kappa_i- 2r))/(c_0 r)} = c_0 r \frac{\kappa_i-2r}{2 + 2(\kappa_i- 2r) } \geq \frac{c_0 r}{4},
 \end{align*}
where the last estimate follows since $0 < \kappa_i - 2r < 1$ and hence $\left(2 + 2(\kappa_i- 2r) \right) \leq 4$.
It is not hard to see that $g_i$ can be parametrised by
$$
g_i(w) = x_i + \binom{\lambda_i^{t-1} w}{\lambda_i^{2-t} w}, \quad w \in [0, \kappa_i],
$$
for a $|\lambda| = 1$. We define, for $k = 1, \dots, L_i$,  $w_k\coloneqq r + 2 s_i k - s_i$ and
$$
z_{k,i} \coloneqq  x_i + \binom{\lambda_i^{t-1} w_k}{\lambda_i^{2-t} w_k}.
$$
Since $|\lambda| = 1$ it follows  that $|(z_{k,i})_t - (z_{k,i+1})_t| = 2 s_i$ for $i = 1, \dots, L_i - 1$ and
$|(z_{k,i})_{t'} - (z_{k,i+1})_{t'}| \leq 2 s_i$ for $i = 1, \dots,  L_i - 1$, where $t' = 2$ if $t = 1$ and $t' = 1$ if $t= 2$. Moreover, $|(x_i)_t- (z_{1,i})_t|= r$ and $|(x_i)_{t'}- (z_{1,i})_{t'}|\leq r$ and $|(x_{i+1})_t- (z_{L_i,1})_t|= r$ and $|(x_{i+1})_{t'}- (z_{L_i,i})_{t'}|\leq r$. This shows that
\begin{align}\label{eq:overlap}
g_i \subset \dot{\bigcup}_{k = 1, \dots, L_i} C(z_{k,i}, s_i) \ \dot{\cup} \ \dot{\bigcup}_{\ell = i, i+1} C(x_\ell, r).
\end{align}

Next, we compute for any $i \leq N$ and $k_i \leq L_i$ the distance between $z_{k_i,i}$ and any line segment $g_j$, $j \neq i$. If $|j-i| > 1$ and $(i,j) \neq (1,N), (N,1)$, then we know by assumption that $\mathfrak{d}(z_{k_i,i}, g_j) \geq r_1 \geq 8 r$. We also obtain that $\mathfrak{d}(x_{i}, g_j) \geq 8 r$ and $\mathfrak{d}(x_{i+1}, g_j) \geq 8 r$.
Hence, for any $k_j \leq L_j$ we have
\begin{align}\label{eq:EmptyNonNeighborIntersection}
C(z_{k_i,i}, s_i) \cap C(z_{k_j,j}, s_j) =  \emptyset \quad \text{ and } \quad C(z_{k_i,i}, s_i) \cap C(x_{j}, r)=  \emptyset.
\end{align}
If $j = i+1$ or $j= 1$, $i = N$, then let $z^* \in g_j$ be such that
$\mathfrak{d}(z_{k,i}, z^*) = \mathfrak{d}(z_{k,i}, g_{i+1})$. This choice is possible since $g_j$ is compact.
If  $z^* = x_{j}$ then we have by construction $\mathfrak{d}(z_{k,i}, z^*) \geq s_i + r$.
If the smallest angle between $g_i$ and $g_{j}$ is larger than or equal to $\pi/2$ then it is clear that $z^* = x_{j}$, so we can proceed by assuming that the smallest angle $\beta$ between $g_i$ and $g_{j}$ satisfies $\alpha \leq \beta < \pi/2$.
By construction, the triangle with vertices $z_{k,i}, z^*, x_{j}$ has a right angle at $z^*$.
By the law of sines, we conclude that $\mathfrak{d}(z_{k,i}, z^*)/\sin(\beta) = |z_{k,i} - x_{j}|$. Using that $|z_{k,i} - x_{j}| > r + s_i$ we conclude that $\mathfrak{d}(z_{k,i}, z^*) \geq  2(r + s_i) \sin(\alpha) = 2(r + s_i) c_0$.

We conclude, for all $i \leq N$, $k_1 \leq L_i$ and $k_2 \leq L_{i+1}$, that
\begin{align}\label{eq:EmptyneighborIntersection}
C(z_{k_1,i}, s_i) \cap C(z_{k_2,i}, s_{i+1}) \subset B_{\sqrt{2}s_i}(z_{k_1,i}) \cap B_{\sqrt{2}s_{i+1}}(z_{k_2,i+1}) = \emptyset,
\end{align}
since $2\sqrt{2}s_i < r c_0/\sqrt{2} < |z_{k_2,j} - z_{k_1,i}|$. From \eqref{eq:overlap}, \eqref{eq:EmptyNonNeighborIntersection}, and \eqref{eq:EmptyneighborIntersection} we conclude that for any $i, j \leq N$, $k_i \leq L_i$, $k_j \leq L_j$ we have that
\begin{align} \label{eq:emptyintersections}
C(z_{k_i,i}, s_i) \cap C(z_{k_j,j}, s_j) &= \emptyset, \text{ if } (i,k_i) \neq (j,k_j), \text{ and }\\
C(x_{i}, r) \cap C(x_{i}, r) &= \emptyset, \text{ if } j \neq i,\nonumber\\
C(z_{k_i,i}, s_i) \cap C(x_{j}, r) &= \emptyset.\nonumber
\end{align}
We set $L \coloneqq \sum_{i=1}^N L_i$ and get that $L \geq \sum_{i=1}^N  |x_i - x_{i+1}|/4 = \ell(P)/4$. Next, we conclude from \eqref{eq:overlap} and \eqref{eq:emptyintersections} that \eqref{eq:theCover} holds. Finally, it is not hard to see that, $C(z_{k, i}, s_i + c_0 r/4) \cap g_j = \emptyset$ for all $j \neq i$ and therefore $\chi_{P \cap  C(z_{k,i}, s_i + c_0 r/4)}( \cdot -z_{k,i})$ is a rotated Heaviside function on $[-s_i -  c_0 r / 4 , s_i + c_0 r / 4]^2$ with the same normal direction as $\partial P$ at $z_{k,i}$.
\end{proof}

\section{Some integral estimates}\label{app:Computations}
We start by showing equation \eqref{eq:estimateLowFrequencyPart}. We have that
\begin{align*}
\int_{\mathfrak{a}}^{\infty}\int_{\R} | \widehat{\mu}_{a,s,0,1}(\xi)|^2 a^{-3} \mathrm{d}s \, \mathrm{d}a = & \ \int_{\mathfrak{a}}^{\infty}\int_{\R} \left| a^{\frac34} \frac{\min\{|a \xi_1|^3,1\}}{(1+|a\xi_1|)^4 (1+|\sqrt{a}(s \xi_1 + \xi_2)|)^4} \right|^2 a^{-3} \mathrm{d}s \, \mathrm{d}a\\
 \leq & \left( \frac{1}{(1+|\mathfrak{a} \xi_1|)^2} \right) \int_{\mathfrak{a}}^{\infty}\int_{\R} \left|\frac{\min\{|a \xi_1|^3,1\}}{(1+|a\xi_1|)^3 (1+|\sqrt{a}(s \xi_1 + \xi_2)|)^4} \right|^2 a^{-\frac32} \mathrm{d}s \, \mathrm{d}a\\
 \lesssim & \frac{1}{|\mathfrak{a}|^2} \left( \frac{1}{(1+|\xi_1|)^2} \right) \int_{\R^+}\int_{\R} \left|\frac{\min\{|a \xi_1|^3,1\}}{(1+|a\xi_1|)^3 (1+|\sqrt{a}(s \xi_1 + \xi_2)|)^4} \right|^2 a^{-\frac32} \mathrm{d}s \, \mathrm{d}a.
\end{align*}
By \cite[Equation 5]{grohs2011continuous}, we have that
\begin{align}\label{eq:Equation5Grohs}
\int_{\R^+}\int_{\R} \left|\frac{\min\{|a \xi_1|^3,1\}}{(1+|a\xi_1|)^3 (1+|\sqrt{a}(s \xi_1 + \xi_2)|)^4} \right|^2 a^{-\frac32} \mathrm{d}s \, \mathrm{d}a < \infty,
\end{align}
which yields \eqref{eq:estimateLowFrequencyPart}.

Next, we demonstrate \eqref{eq:estimateHighFrequencyPart}. If $\nu_s |\xi_1|/2\geq |\xi_2|$ then
\begin{align*}
&\int_0^\mathfrak{a}\int_{-\infty}^{-\nu_s} | \widehat{\mu}_{a,s,0,1}(\xi)|^2 a^{-3}\mathrm{d}s \, \mathrm{d}a\\
& =  \int_0^\mathfrak{a} \int_{\R} \left| a^{\frac34} \frac{\min\{|a \xi_1|^3,1\}}{(1+|a\xi_1|)^4 (1+|\sqrt{a}(s \xi_1 + \xi_2)|)^4} \right|^2 a^{-3} \mathrm{d}s \, \mathrm{d}a\\
& \leq \int_0^\mathfrak{a} \int_{\R} \left| a^{\frac34} \frac{\min\{|a \xi_1|^3,1\}}{ (1+|\sqrt{a}(s \xi_1)|)^2  (1+|a\xi_1|)^4 (1+|\sqrt{a}(s \xi_1 + \xi_2)|)^2} \right|^2 a^{-3} \mathrm{d}s \, \mathrm{d}a\\
 & \leq \int_0^\mathfrak{a}   \frac{\min\{|a \xi_1|^2,1\}}{(1+|\sqrt{a}(s \xi_1)|)^4 } \int_{\R} \left| a^{\frac34} \frac{\min\{|a \xi_1|^3,1\}}{  (1+|a\xi_1|)^4 (1+|\sqrt{a}(s \xi_1 + \xi_2)|)^2} \right|^2 a^{-3} \mathrm{d}s \, \mathrm{d}a\\
 & \leq  \frac{|\xi_1|^2}{(1+|(s \xi_1)|)^4 } \int_0^\mathfrak{a} \int_{\R} \left| a^{\frac34} \frac{\min\{|a \xi_1|^3,1\}}{  (1+|a\xi_1|)^4 (1+|\sqrt{a}(s \xi_1 + \xi_2)|)^2} \right|^2 a^{-3} \mathrm{d}s \, \mathrm{d}a\\
& \lesssim  \frac{|\xi_1|^2}{(1+|(s \xi_1)|)^4 },
\end{align*}
where we again invoked \cite[Equation 5]{grohs2011continuous} to estimate the integral over $s$ and $a$. This completes the estimate of \eqref{eq:estimateHighFrequencyPart}.

\section{The Bessel inequality}\label{app:Bessel}

Let $f \in L^2(\R^2)$ and $\mathfrak{a}>0,\mathfrak{s}>0$; we then compute, by Plancherel's theorem, that
\begin{align*}
\int_0^\mathfrak{a} \int_{-\mathfrak{s}}^\mathfrak{s} \int_{\R^2}| \langle f, \psi_{a,s,t,\iota} \rangle |^2 a^{-3} \, \,\mathrm{d}t \, \mathrm{d}s \, \mathrm{d}a = \int_0^\mathfrak{a} \int_{-\mathfrak{s}}^\mathfrak{s} \int_{\R^2} \left| \left \langle \hat{f} \widehat{\psi}_{a,s,0,\iota},
\mathrm{e}^{-2\pi i \langle \cdot, m\rangle}\right\rangle \right|^2 a^{-3} \, \,\mathrm{d}t \, \mathrm{d}s \, \mathrm{d}a.
\end{align*}
Using Parseval's identity and Fubini's theorem, we obtain that	
\begin{align*}
\int_0^\mathfrak{a} \int_{-\mathfrak{s}}^\mathfrak{s} \int_{\R^2}\left| \left\langle \hat{f} \  \widehat{\psi}_{a,s,0,\iota}, \mathrm{e}^{-2\pi i \langle \cdot, m\rangle}\right\rangle \right|^2 a^{-3} \, \,\mathrm{d}t \, \mathrm{d}s \, \mathrm{d}a
= & \ \int_0^\mathfrak{a} \int_{-\mathfrak{s}}^\mathfrak{s} \left\|\hat{f} \widehat{\psi}_{a,s,0,\iota}\right\|_{L^2(\R^2)}^2 a^{-3} \mathrm{d}s \, \mathrm{d}a \\
 = & \ \int_{\R^2} \left|\hat{f}(\xi)\right|^2 \int_0^\mathfrak{a} \int_{-\mathfrak{s}}^\mathfrak{s} \left| \widehat{\psi}_{a,s,0,\iota}\right|^2 a^{-3} \mathrm{d}s \, \mathrm{d}a \, \mathrm{d} \xi.
\end{align*}
Equation \eqref{eq:Equation5Grohs} and
$$
\left|\widehat{\phi^1}(\xi)\right| \leq (1+|\xi|)^{-4} \quad \text{ and } \quad \left|\widehat{\psi^1}(\xi)\right| \leq \frac{\min\{ |\xi|, 1\}^4}{(1+|\xi|)^4}
$$
imply that $\int_0^\mathfrak{a} \int_{-\mathfrak{s}}^\mathfrak{s} | \widehat{\psi}_{a,s,0,\iota}|^2 a^{-3} \mathrm{d}s \, \mathrm{d}a$ is uniformly bounded by a constant $C > 0$, independent of $\mathfrak{a}$ and $\mathfrak{s}$, which yields that
\begin{align}\label{eq:BesselInequalityIndepOfsa}
\int_0^\mathfrak{a} \int_{-\mathfrak{s}}^\mathfrak{s} \int_{\R^2}| \langle f, \psi_{a,s,t,\iota} \rangle |^2 a^{-3} \, \,\mathrm{d}t \, \mathrm{d}s \, \mathrm{d}a \leq C \|f\|_{L^2(\R^2)}^2, \text{ for all } \mathfrak{a},\mathfrak{s} > 0.
\end{align}
If $f \in L^2((0,1)^2)$, then
$$
\int_0^\mathfrak{a} \int_{-\mathfrak{s}}^\mathfrak{s} \int_{(0,1)^2} \left| \left\langle f, \psi_{a,s,t,\iota}^{per} \right\rangle \right|^2 a^{-3} \, \,\mathrm{d}t \, \mathrm{d}s \, \mathrm{d}a
\leq \int_0^\mathfrak{a} \int_{-\mathfrak{s}}^\mathfrak{s} \int_{\R^2} \left| \left\langle f^{per}, \psi_{a,s,t,\iota} \right\rangle \right|^2 a^{-3} \, \,\mathrm{d}t \, \mathrm{d}s \, \mathrm{d}a,
$$
where $f^{per}(x) = f(x - \lfloor x \rfloor)$ for all $x \in [-2,2]^2$, where $\lfloor \cdot \rfloor$ is applied componentwise, and $f^{per}(x) = 0$ on $\R^2 \setminus [-2,2]^2$.
Hence, by \eqref{eq:BesselInequalityIndepOfsa},
$$
\int_0^\mathfrak{a} \int_{-\mathfrak{s}}^\mathfrak{s}\int_{(0,1)^2} \left| \left\langle f, \psi_{a,s,t,\iota}^{per} \right\rangle \right|^2 a^{-3} \, \,\mathrm{d}t \, \mathrm{d}s \, \mathrm{d}a \leq C \left\|f^{per}\right\|_{L^2(\R^2)}^2 = C \|f\|_{L^2((0,1)^2)}^2.
$$
Performing the same argument for the second cone yields that
\begin{align}\label{eq:BesselPeriodic}
\sum_{\iota = -1,1}\int_0^\mathfrak{a} \int_{-\mathfrak{s}}^\mathfrak{s} \int_{(0,1)^2} | \langle f, \psi_{a,s,t,\iota}^{per} \rangle |^2 a^{-3} \, \,\mathrm{d}t \, \mathrm{d}s \, \mathrm{d}a \leq C \|f\|_{L^2((0,1)^2)}^2, \quad \text{ for all } f \in L^2((0,1)^2).
\end{align}

\section{Counterexample to the convergence result of Bertozzi and Dobrosotskaya}\label{app:Counterex}

In \cite[Theorem 1.1]{dobrosotskaya2011wavelet} it is claimed, that for any sequence $u_n \in H^1((0,1)^2)$ converging to $\chi_D$ in $L^1((0,1)^2)$, where $D$ is a set of finite perimeter, the quotient
$$
 \frac{|u_n|_{\Bp}}{|u_n|_{H^1((0,1)^2)}}
$$
converges to $\mathcal{R}(\chi_D)$, and $\mathcal{R}(\chi_D)$ is a nonconstant function of $\partial D$. This theorem cannot hold as the following counterexample demonstrates.

For a contradiction, we assume that \cite[Theorem 1.1]{dobrosotskaya2011wavelet} holds. Pick $D \subset (0,1/2)^2$ with finite perimeter and a sequence $(u_n)_{n\in \N} \subset H^1((0,1)^2)$, such that $u_n \to \chi_D$ in $L^1((0,1)^2)$, and $|u_n|_{H^1((0,1)^2)} \sim |u_n|_{\Bp} \to \infty$ for $n \to \infty$, and $\suppp u_n \subset (0,1/2)^2$.
Let $g_n$ be a sequence in $H^1((0,1)^2)$ with $\|g_n\|_{L^1((0,1)^2)} \to 0$, $\suppp g_n \subset (1/2,1)^2$ and $|g_n|_{H^1((0,1)^2)} \sim |u_n|_{H^1((0,1)^2)}^2$.
Since the $H^1((0,1)^2)$ seminorm of $g_n$ is asymptotically much larger than the $H^1((0,1)^2)$ seminorm of $u_n$, one can easily observe that
\begin{align*}
\frac{|u_n + g_n|_{\Bp}}{|u_n+ g_n|_{H^1((0,1)^2)}} -  \frac{|g_n|_{\Bp}}{|g_n|_{H^1((0,1)^2)}} \to 0,
\end{align*}
for $n \to \infty$. Hence, by \cite[Theorem 1.1]{dobrosotskaya2011wavelet},
$$
\frac{|u_n|_{\Bp}}{|u_n|_{H^1((0,1)^2)}}, \quad \frac{|u_n + g_n|_{\Bp}}{|u_n+ g_n|_{H^1((0,1)^2)}}, \text{ and }  \frac{|g_n|_{\Bp}}{|g_n|_{H^1((0,1)^2)}},
$$
all converge to the same limit $\mathcal{R}(\chi_D)$ for $n \to \infty$. The limit of $|g_n|_{\Bp}/|g_n|_{H^1((0,1)^2)}$, however, is independent of $\partial D$, completing the asserted contradiction.

The main mistake in \cite{dobrosotskaya2011wavelet} causing the incorrect conclusion of \cite[Theorem 1.1]{dobrosotskaya2011wavelet} can be found in \cite[Lemma 2.6]{dobrosotskaya2011wavelet}, where in the second to last line an upper bound of $2^J$ is mistaken for a lower bound. The best achievable lower bound from the results in \cite{dobrosotskaya2011wavelet} appears to be $2^{J/2}$.

\bibliographystyle{abbrv}
\bibliography{references}

\end{document}